\documentclass[a4paper, 12pt]{article}
\usepackage{tikz}
\usepackage{pdfrender}
\usepackage{mathrsfs}
\usepackage{amsmath,amsthm,amssymb,amsfonts,amscd,amstext}
\usepackage{latexsym}
\usepackage{amssymb}
\usepackage{appendix}
\usepackage{xspace}
\usepackage{xcolor}
\usepackage[unicode]{hyperref}
\usepackage[utf8]{inputenc}
\usepackage{enumerate}
\usepackage{bm}
\usepackage{bbm}
\usepackage{verbatim}
\usepackage[sort&compress,numbers,square]{natbib}
\theoremstyle{plain}
\newtheorem{theorem}{Theorem}[section]

\newtheorem{proposition}[theorem]{Proposition}

\newtheorem{remark}[theorem]{Remark}
\theoremstyle{definition}
\newtheorem{definition}[theorem]{Definition}

\numberwithin{theorem}{section}
\numberwithin{equation}{section}

\def\R{\mathbb{R}}

\def\N{\mathbb{N}}
\def\Z{\mathbb{Z}}

\usepackage{algorithm, algpseudocode} 
\usepackage{mathrsfs} 

\usepackage{lipsum}
\title{Controllability of a Fluid-Structure Interaction System Governed by the Heat and Damped Beam Equations }
\author{Mehdi Badra\footnote{Email: \href{mailto:le mail de X}{mehdi.badra@math.univ-toulouse.fr}}~~, J\'er\'emi Dard\'e\footnote{Email: \href{mailto:le mail de Y}{jeremi.darde@math.univ-toulouse.fr}}~~and~~Emmanuel Zongo\footnote{Email: \href{mailto:}{emmanuel.zongo@math.univ-toulouse.fr}}}
\date{
	\footnotesize{Institut de Math\'ematiques de Toulouse, UMR5219,\\ Universit\'e de Toulouse, CNRS, \\ UPS, F-31062 Toulouse Cedex 9, France.}
}
\begin{document}
\pdfrender{StrokeColor=black,TextRenderingMode=2,LineWidth=0.2pt}
\maketitle
\begin{abstract}
In this article, we study the null-controllability and observability properties of a bi-dimensio\-nal fluid-structure interaction system, governed by the heat equation coupled with the damped beam equation.
To do so, we prove a global Carleman estimate for the coupled system, making explicit the dependence on the damping parameter. 
In the course of the study, we demonstrate several inequalities for the beam equation alone, valid for different damping parameter regimes.
\\
\textbf{ Funding:} The second author was supported by the ANR LabEx CIMI (under grant ANR-11-LABX-0040) within the French State Programme “Investissements d’Avenir”

The first and third authors was supported by the Agence Nationale de la Recherche, Project
 TRECOS, ANR-20-CE40-0009.
\end{abstract}
\tableofcontents
\section{Introduction and main results}

Fluid-Structure interaction (FSI) models are central to understanding the interplay between fluid dynamics and deformable boundaries in physical, biological, engineering systems, etc. In particular, the cardiovascular system illustrates such interactions, where pulsatile blood flow interacts with the elastic properties of large vessels walls. The comprehensive modeling of these phenomena has been the focus of foundational works detailing the complex coupling between fluid dynamics and vascular wall mechanics in precise geometries. The model consists in a coupling between the 2d-Navier-Stokes equations in a channel and a damped beam-equation that describes the deformation of the boundary of the domain. We refer to \cite{quarteroni} and \cite{GrandmontMaday2003} for the model derivation and we refer to \cite{Grandmont2019} (and references therein) for the mathematical well posedness of the model. Controllability issues have been tackled very recently in \cite{Buffe, Buffe2025} by coupling microlocal analysis and global Carleman inequalities.

In contrast to these realistic models, the present work deals with the null-controllability of a simplified fluid-structure interaction system where the Navier-Stokes equations are replaced by the heat equation. The proposed model retains the core framework of a fluid equation coupled at the boundary with a structural equation such as the damped beam equation, thereby capturing the basic dynamic interaction between a diffusive field and a deformable interface. The aim is not to fully capture the biomechanical behavior of blood vessels, but rather to explore the essential mathematical features of a coupled system within a tractable setting. Specifically, our approach relies entirely on global Carleman inequalities, and allows us to consider arbitrarily large structural damping for the beam -- features absent from the previously cited works  (see Section \ref{Sect1.4}).

In the present study, we focus on a bi-dimensional interaction system set in a rectangular domain, with periodic boundary  conditions on the lateral boundary of the rectangle. Mathematically, we set $\mathcal I = \R / 2\pi \Z$ the one-dimensional torus, and, for some $T>0$, define 
$$\Omega:=\mathcal{I}\times (0,1),~~~\Gamma_1=\mathcal{I}\times\{1\},~~\Gamma_0=\mathcal{I}\times\{0\},~~ Q=(0,T)\times\Omega.$$
In the following, we will also identify $\mathcal I$ and $\Gamma_1$ when it does not lead to confusion. The precise meaning of this identification is given in Section \ref{section_functional_setting}.

\begin{figure}[h!]
\centering
\begin{tikzpicture}

\draw [thick] (0,0) rectangle (6,3);

\draw[thick, red] (0,3) -- (6,3);
\node at (3,3.4) { $\Gamma_1$ (Beam)};

\node at (3,-0.4) { $\Gamma_0$};

\node at (-0.01,1.5) { $=$}; 
\node at (6.,1.5) { $=$};  

\node at (-0.2,3) { $1$};
\node at (-0.3,-0.3) { $0$};

\node at (6,-0.3) { $2\pi$};

\node at (3,1.5) { $\Omega$};

\end{tikzpicture},
\caption{Geometrical setting of our study}
\label{fig}
\end{figure}
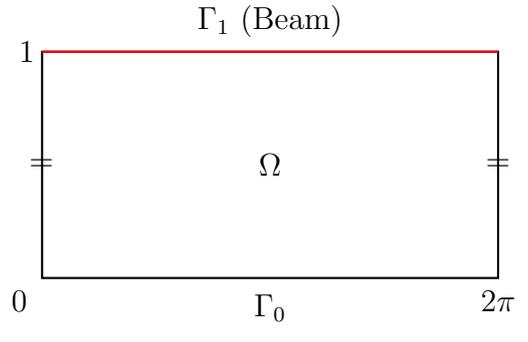

In this geometrical setting, we consider, for some initial conditions $(w_0,\zeta_0,\zeta_1) \in L^2(\Omega) \times H^2(\mathcal I) \times L^2(\mathcal I)$ and some source terms $G \in L^2(Q)$ and $H \in L^2((0,T) \times \mathcal I)$, the forward system of equations
\begin{equation} \label{eq: forward system}
 \begin{cases}
\partial_t w -\Delta w = G& \text{in}~Q,\\
w=0 &\text{on}~(0,T) \times\Gamma_0,\\
w=\partial_t\zeta  &\text{on}~(0,T) \times\Gamma_1,\\
\partial^2_{t}\zeta-\alpha\partial^2_{x_1}\partial_t\zeta+\alpha\partial_t\zeta+\partial^4_{x_1}\zeta+\zeta=-\partial_nw+H  & \text{in}~(0,T) \times\mathcal{I},\\
w(0,\cdot)=w_0,~\text{in}~~\Omega,~~ \zeta(0,\cdot)=\zeta_0,~ \partial_t\zeta(0,\cdot)=\zeta_1 &\text{in}~\mathcal{I}.
\end{cases}  
\end{equation}
In \eqref{eq: forward system}, the first heat equation in the domain $\Omega$ plays the role of a fluid equation submitted to some external forces $G$. The condition $w=\partial_t\zeta$ on the boundary indicates that the velocity of the fluid coincides with the velocity of the beam's displacement. The fourth line in \eqref{eq: forward system} represents the structural response of the beam under the external force $H$ and under the force exerted by the fluid that is represented by $-\partial_n w$. The coefficient $\alpha>0$ is a damping parameter.

 We emphasize that the terms $\zeta$ and $\alpha \partial_t \zeta$ in the fourth equation of system \eqref{eq: forward system} are included for technical reasons only. This simplifies the presentation of the  paper  by straightforwardly ensuring the positivity of the underlying elliptic operators (see \eqref{DefA1} and \eqref{DefA2} below). Note that in the original fluid-structure system, the incompressibility condition ensure that the beam deformation $\zeta$ has a zero mean value (see the introduction of \cite{BeiroVeiga2004}), and the elliptic operators are inherently positive without the additional terms.

System \eqref{eq: forward system} is well posed (see Section \ref{section_functional_setting}): it admits a unique solution $(w,\zeta)$ satisfying 
\begin{equation}\label{RegClassic1}
    (w,\zeta,\partial_t\zeta)\in C([0,T];L^2(\Omega) \times H^2(\mathcal I)\times L^2(\mathcal I))
\end{equation}
and the following regularity for all $ \varepsilon>0$:
\begin{equation}\label{RegClassic2}
\begin{array}{c}
\displaystyle w \in L^2(\varepsilon,T; H^2(\Omega)) \cap H^1(\varepsilon,T;L^{2}(\Omega)),\medskip \\
\displaystyle \zeta \in L^2(\varepsilon,T;H^4(\mathcal I)) \cap
H^1(\varepsilon,T;H^2(\mathcal I))\cap H^2(0,T;L^2(\mathcal I)).
\end{array}
\end{equation}

\subsection{Control problem}

We now focus on our main problem of interest, the null-controllability of the interaction system. More precisely, we consider $\omega$, a non-empty open subset of $\Omega$, and $J$, a non-empty open subset of $\mathcal I$. In the context of controllability, the version of system \eqref{eq: forward system}
we are interested in reads 
\begin{align}\label{eq: parabolic}
\begin{cases}
\partial_t w -\Delta w =\mathbbm{1}_{\omega} g&\text{in}~Q,\\
w=0 &\text{on}~(0,T) \times\Gamma_0,\\
w=\partial_t\zeta  &\text{on}~(0,T) \times\Gamma_1,\\
\partial^2_{t}\zeta-\alpha\partial^2_{x_1}\partial_t\zeta+\alpha\partial_t\zeta+\partial^4_{x_1}\zeta+\zeta=-\partial_nw+\mathbbm{1}_{J} h & \text{in}~(0,T) \times\mathcal{I},\\
w(0,\cdot)=w_0,~\text{in}~~\Omega,~~ \zeta(0,\cdot)=\zeta_0,~ \partial_t\zeta(0,\cdot)=\zeta_1 &\text{in}~\mathcal{I},
\end{cases}
\end{align}
where $g \in L^2((0,T)\times \omega)$ and $h \in L^2(0,T)\times J)$ are the distributed controls acting respectively on the heat equation and on the damped beam equation. The functions $\mathbbm{1}_{\omega}$ and $\mathbbm{1}_{J}$ denote the characteristic functions on $\omega$ and $J$ respectively.

\begin{definition}
We say that system \eqref{eq: parabolic} is null-controllable at time $T>0$ if there exists a constant $C_T>0$ such that, for all initial conditions $w_0\in L^2(\Omega)$, $\zeta_0\in H^2(\mathcal I)$ and $\zeta_1 \in L^2(\mathcal I)$, there exist  $g \in L^2((0,T)\times \omega)$ and $h \in L^2((0,T) \times J)$ such that $(w,\zeta)$ solution of \eqref{eq: parabolic} satisfies $(w(T),\zeta(T),\partial_t\zeta(T)) = (0,0,0)$, 
and 
$$
\Vert g\Vert_{L^2((0,T)\times \omega)} + \Vert h \Vert_{L^2((0,T)\times J)} \leq C_T
\left( \Vert w_0\Vert_{L^2(\Omega)} + \Vert \zeta_0 \Vert_{H^2(\mathcal I)} + \Vert \zeta_1 \Vert_{L^2(\mathcal I)} \right).
$$
\end{definition}

In this study, we prove that system \eqref{eq: parabolic} is null-controllable in any arbitrary small time $T>0$, for any choice of the damping parameter $\alpha>0$.

\begin{theorem} \label{thm_nullcontrollability_intro}
Let $T>0$, and $\alpha>0$. System \eqref{eq: parabolic} is null-controllable at time $T$, and there exists $C>0$, independent of $T$, such that we can choose $C_T = C e^{\frac{C}{T}}$.
\end{theorem}
\begin{remark}
In Theorem \ref{thm_nullcontrollability_intro}, the dependence of the constant $C_T$ on the time horizon $T$ is improved compared to \cite{Buffe}, where the constant takes the form $C_T = e^{\frac{C}{T^2}}$. 
This is because \cite{Buffe} utilizes the Carleman inequality for the damped beam equation presented in \cite{Sourav}, where the dependence of the Carleman parameter $s$ on $T$ is suboptimal. We improve this dependence in Theorem \ref{poutre3} and in Theorem \ref{poutre-theorem}, see Remark \ref{Rkcout2}.

Moreover, the dependence of the constant $C_T$ on the damping parameter $\alpha$ can be quantified precisely. According to Theorem \ref{thm_observability_intro} and Theorem \ref{observability} below, we have $C_T=Ce^{e^{C\lambda_\alpha}\left(1+\frac{1}{T}\right)}$ where $\lambda_\alpha=\tau\alpha^{-2}$ if $\alpha<1$ and $\lambda_\alpha=\tau\alpha^{2/3}$ if $\alpha\geq 1$, for some $C>0$ and $\tau>0$ that are independent on $\alpha>0$ .
\end{remark}

To prove this theorem, we use the classical duality strategy, and prove an observabilty inequality for the adjoint system, which in our context reads (see Section \ref{section_functional_setting}):

\begin{equation} \label{eq: parabolic adjoint}
    \begin{cases}
    \partial_t u - \Delta u = 0, & \text{ in } Q, \\
    u = 0 , & \text{ on } (0,T) \times \Gamma_0, \\
    u = \partial_t \eta,& \text{ on } (0,T) \times \Gamma_1, \\
    \partial^2_{t}\eta-\alpha\partial^2_{x_1}\partial_t\eta+\alpha\partial_t\eta+\partial^4_{x_1}\eta+\eta=-\partial_n u& \text{in}~(0,T) \times\mathcal{I},\\
    u(0,\cdot)=u_0,~\text{in}~~\Omega,~~ \eta(0,\cdot)=\eta_0,~ \partial_t\eta(0,\cdot)=\eta_1 &\text{in}~\mathcal{I}.
    \end{cases}
\end{equation}

Here, the initial conditions satisfy $u_0 \in L^2(\Omega)$, $\eta_0 \in H^2(\mathcal I)$, $\eta_1 \in L^2(\mathcal I)$, and as a consequence system \eqref{eq: parabolic adjoint} is well posed.

\begin{definition}
Let $T>0$. We say that system \eqref{eq: parabolic adjoint} is observable at time $T$ from the observation sets $\omega$ and $J$ if there exists a constant $K_T>0$ such that any solution of system \eqref{eq: parabolic adjoint} satisfies
\begin{equation} \label{eq: observability inequality introduction}
\Vert u(T) \Vert_{L^2(\Omega)} + \Vert \eta(T) \Vert_{H^2(\mathcal I)} + \Vert \partial_t \eta(T) \Vert_{L^2(\mathcal I)} \leq K_T \left( \Vert u \Vert_{L^2((0,T) \times \omega)} + \Vert \partial_t \eta\Vert_{L^2((0,T) \times J)} \right).
\end{equation}
\end{definition}

Very classically, system \eqref{eq: parabolic} is null-controllable at time $T$ if and only if system \eqref{eq: parabolic adjoint} is observable at time $T$, and in that case the constants $C_T$ and $K_T$ can be chosen equal (see Section \ref{section_functional_setting}). We prove the following result:

\begin{theorem} \label{thm_observability_intro}
For all $T>0$, and all $\alpha > 0$, the system \eqref{eq: parabolic adjoint} is observable at time $T$ from the observation sets $\omega$ and $J$. Furthermore, there exists a constant $C>0$, independent of $T$, such that $K_T = C e^{\frac{C}{T}}$.
\end{theorem}

The main ingredient in our proof of Theorem \ref{thm_observability_intro} is a new global Carleman estimate for the whole heat-damped beam system \eqref{eq: parabolic adjoint}. This Carleman estimate, which constitutes the main novelty of our study, is of interest in its own right. We present this new result in the next section.

Another novelty of Theorem \ref{thm_observability_intro} is its validity for any positive damping parameter $\alpha$. Maybe unexpectedly, new difficulties appear for large parameters $\alpha$, for which we need the coupling between the heat equation and the damped beam equation to obtain the global Carleman estimate.

\subsection{Carleman estimate for the coupled system}\label{carleman-pour-les deux}

We now present our new global Carleman estimate for the system \eqref{eq: parabolic adjoint}.
First, we need to introduce the following weight functions. We recall that the existence of such functions is proven in \cite[Section 2.2]{Buffe}. We consider the function $\ell$ defined by
$$\ell(t):=t(T-t),~~\text{for}~~t\in (0,T).$$ 
Let $\omega_0$ and $J_0$ be two nonempty open sets satisfying $\omega_0\Subset\omega$ and $J_0\Subset J.$
We consider two smooth functions $\psi_\Omega$ and $\psi_{\mathcal{I}}$ satisfying 
\begin{equation}\label{w1}
\psi_{\Omega}>0~~\text{in}~~\Omega,~~~\psi_\Omega=0~~\text{and}~~\partial_n\psi_{\Omega}=-1~~\text{on}~~\partial\Omega,~~\nabla\psi_\Omega(x)=0\Rightarrow x\in\omega_0.
\end{equation}
\begin{equation}\label{w2}
\psi_{\mathcal{I}}>0~~\text{on}~~\mathcal{I},~~\psi'_{\mathcal{I}}(x_1)=0\Rightarrow x_1\in J_0.
\end{equation}
Let us consider 
\begin{equation} \label{eq_def_PSI} 
\Psi:=\|\psi_\Omega\|_{L^{\infty}(\Omega)}+\|\psi_{\mathcal{I}}\|_{L^{\infty}(\mathcal{I})}\end{equation}
and for $k\geq 2$ and $\lambda\geq \mu>0$, let us define the following functions:
\begin{equation}\label{weights2}
\varphi(t,x_1,x_2)=\frac{e^{\mu\psi_{\mathcal{I}}(x_1)+\lambda\psi_\Omega(x_1,x_2)+8\lambda\Psi}-e^{10\lambda\Psi}}{\ell(t)^{k/2}},~~~~\xi(t,x_1,x_2)=\frac{e^{\mu\psi_{\mathcal{I}}(x_1)+\lambda\psi_\Omega(x_1,x_2)+8\lambda\Psi}}{\ell(t)^{k/2}},
\end{equation}
\begin{equation}\label{weights1}
 \varphi_0(t,x_1)=\frac{e^{\mu\psi_{\mathcal{I}}(x_1)+8\lambda\Psi}-e^{10\lambda\Psi}}{\ell(t)^{k/2}},~~~~\xi_0(t,x_1)=\frac{e^{\mu\psi_{\mathcal{I}}(x_1)+8\lambda\Psi}}{\ell(t)^{k/2}}.
\end{equation}
Since $\psi_\Omega$ vanishs on $\partial\Omega$ we observe that $\varphi_0=\varphi|_{\partial\Omega}$  and $\xi_0=\xi|_{\partial\Omega}$.
\
\\
The Carleman estimate for the adjoint system \eqref{eq: parabolic adjoint} is given in the following theorem. 
Set
\begin{align*}
 H(s, \lambda, u)&:= s^{-1}\iint_{(0,T)\times\Omega}~\xi^{-1}~e^{2s\varphi}\left(|\partial_tu|^2+|\Delta u|^2\right)+s\lambda^2\iint_{(0,T)\times\Omega}~\xi~e^{2s\varphi}|\nabla u|^2\\ 
   &+s^3\lambda^4\iint_{(0,T)\times\Omega}~\xi^3~e^{2s\varphi}|u|^2+2s^3\lambda^3\iint_\Sigma\xi_0^3~e^{2s\varphi_0}|u|^2+2s\lambda\iint_\Sigma\xi_0~e^{2s\varphi_0}\big|\frac{\partial u}{\partial n}\big|^2,
\end{align*}
and
\begin{align*}
 B_\alpha(s, \mu, \eta)&:=  \frac{\alpha^2}{1+\alpha^2}\frac1s\iint_{(0,T)\times\mathcal{I}}\frac1\xi_0~e^{2s\varphi_0}(|\partial^4_{x_1}\eta|^2+|\partial^2_t\eta|^2)+\frac{\alpha^2}{s}\iint_{(0,T)\times\mathcal{I}}\frac1\xi_0~e^{2s\varphi_0}(|\partial_t\partial^2_{x_1}\eta|^2)\\& \notag+\iint_{(0,T)\times\mathcal{I}}e^{2s\varphi_0}\Big(s^7\mu^8\xi^7_0|\eta|^2+s^5\mu^6\xi_0^5|\partial_{x_1}\eta|^2+s^3\mu^4\xi_0^3|\partial^2_{x_1}\eta|^2\\ \notag &+(1+\alpha^2)s^3\mu^4\xi_0^3|\partial_t\eta|^2+s\mu^2\xi_0|\partial^3_{x_1}\eta|^2+(1+\alpha^2)s\mu^2\xi_0|\partial_t\partial_{x_1}\eta|^2\Big).
\end{align*}
\begin{theorem}\label{Carleman-couplé}
Let $\omega$ and $J$ be nonempty open subsets of $\Omega$ and $\mathcal{I}$ respectively and let $k\geq 2$. There exist constants $c_0>0$, $c_1 \geq 1$, $\mu_0\geq 1$ and $\widehat{s}_0\geq 1$ such that, for all $\alpha>0$, for all $\lambda\geq \mu\geq \mu_0$ such that $\lambda\geq c_1(1+1/\alpha^2)$ and $\lambda\in [c_1\alpha^{2/3}\mu^{\frac43}, \frac{1+\alpha^2}{c_1}\mu^2]$, for all $T>0$ and for all $s\geq \widehat{s}_0(1+\alpha)(T^k+T^{k-1})$, the following inequality holds for all $u\in C^1([0,T];C(\overline{\Omega}))$ and $\eta\in C^2([0,T];C^4(\mathcal{I}))$ such that $u=\partial_t \eta$ on $(0,T)\times\Gamma_1$ and $u=0$ on $(0,T)\times\Gamma_0$:
\begin{multline}\label{EqCarleman-couplé}
 H(s, \lambda, u)+ B_\alpha (s, \mu, \eta)\leq c_0\Big( s^3\lambda^4\iint_{(0,T)\times\omega} \xi^3~~e^{2s\varphi}| u|^2+ \left(\alpha^{-8}+\alpha^4\right)s^7\mu^8\iint_{(0,T)\times J}\xi^7_0~e^{2s\varphi_0}|\eta|^2\\ 
 +\iint_{(0,T)\times\Omega}~e^{2s\varphi}|\partial_tu-\Delta u|^2+\iint_{(0,T)\times\mathcal{I}}e^{2s\varphi_0}|\partial^2_t\eta-\alpha\partial^2_{x_1}\partial_t\eta+\alpha\partial_t\eta+\partial^4_{x_1}\eta+\eta+\partial_nu|^2\Big).
\end{multline}
\end{theorem}

\begin{remark}\label{RqCarleman-couplé} 
In Theorem \ref{Carleman-couplé} one can choose $\lambda$ proportional to $\alpha^{2/3}$ if $\alpha\geq 1$  and $\lambda$ proportional to $1/\alpha^{2}$ if $\alpha<1$. More precisely, we can choose $\theta>0$ large enough
 and $\tau\in [c_1\theta^{4/3}, \theta^2/c_1]$ such that the assumptions of Theorem \ref{Carleman-couplé} are satisfied with $\lambda=\tau\alpha^{2/3}$ and $\mu=\theta$ for all $\alpha\geq 1$, or with $\lambda=\frac{\tau}{\alpha^2}$ and $\mu=\frac{\theta}{\alpha}$ for all $0<\alpha<1$. 
\end{remark}

The proof of Theorem \ref{Carleman-couplé} is provided in Appendix \ref{preuve-carleman}.
Our main strategy of proof consists in first to prove a Carleman estimate for the heat equation and the damped beam equation separately. Next, the two Carleman estimates are combined to get the above global estimate.
 The Carleman estimates for the damped beam equation is proven in Section \ref{carleman-poutre}.
 The combination argument is provided in Section \ref{Sect33}.
\\
\\
Note that we obtain naturally an observation term in $\eta$ for the damped beam equation. To prove the observability inequality \eqref{eq: observability inequality introduction}, we need to get an observation term in $\partial_t\eta.$ To do so, we rely on the strategy used in \cite[Section 5]{Buffe}.
Let $\varphi_1,$ $\varphi_2,$ $\xi_1$ and $\xi_2$ be weight functions depending only on the time variable $t$, defined by
\begin{equation*}\label{poids en t1_intro}
\varphi_1(t)=\frac{e^{8\lambda\Psi}-e^{10\lambda\Psi}}{\ell^{k/2}(t)},~~\xi_1(t)=\frac{e^{8\lambda\Psi}}{\ell^{k/2}(t)},
\end{equation*}
\begin{equation*}\label{poids en t2_intro}
\varphi_2(t)=\frac{e^{9\lambda\Psi}-e^{10\lambda\Psi}}{\ell^{k/2}(t)},~~\xi_2(t)=\frac{e^{9\lambda\Psi}}{\ell^{k/2}(t)}.
\end{equation*}
We prove the following:
\begin{proposition}\label{obs-dteta}
Let $\omega$ and $J$ be nonempty open subsets of $\Omega$ and $\mathcal{I}$ respectively and let $\lambda$ and $s$ satisfy the conditions stated in Theorem \ref{Carleman-couplé}. There exist $c_0>0$ such that any solution of \eqref{eq: parabolic adjoint} satisfies:
\begin{multline}\label{dteta}
\int_0^Ts^{3}\lambda^2~\xi^{3}_1~e^{2s\varphi_1}\left(\|u\|^2_{L^2(\Omega)}+\|\eta\|^2_{H^2(\mathcal{I})}+\|\partial_t\eta\|^2_{L^2(\mathcal{I})}\right)\\
\leq c_0\Big( (\alpha^{-4}+\alpha^{16})\iint_{(0,T)\times\omega}s^{11}\lambda^{4}\xi_2^{11}~e^{4s\varphi_2-2s\varphi_1} |u|^2
+ (\alpha^{-8}+\alpha^{10})\iint_{(0,T)\times J}s^{7}\lambda^8\xi_2^{7}~e^{2s\varphi_2}|\partial_t\eta|^2\Big).
\end{multline}
\end{proposition}

The proof of Proposition \ref{obs-dteta} is given in Section \ref{section_observability}.
The Theorem \ref{thm_observability_intro} follows then by a standard dissipation argument, see Theorem \ref{observability} below.

\subsection{Further results -- Carleman estimates for the one-dimensional beam equation}

In the course of our study, we obtain Carleman estimates for the one-dimensional beam equation. 
Such estimates might be of interest on their own. We therefore give them independently in the following. The proofs of the following results are given in Section  \ref{carleman-poutre}.

\subsubsection{The undamped beam equation}

We first obtain the following Carleman estimates. 

\begin{theorem}\label{poutre2}
Let $J$ be a nonempty open subset of $\mathcal{I}$ and let $k\geq 2$. There exist $c_0>0$, $\mu_0\geq 1$ and $\widehat{s}_0\geq 1$ such that for all $\mu\geq\mu_0$, for all $\alpha\geq 0$, for all $T>0$ and for all $s\geq \widehat{s}_0(1+\alpha)(T^k+T^{k-1})$, the following inequality holds for all function $\eta\in C^2([0,T];C^4(\mathcal{I}))$:
\begin{multline}\label{poutre-theorem3}
\iint_{(0,T)\times\mathcal{I}}e^{2s\varphi_0}\Big(s^7\mu^8\xi^7_0|\eta|^2+s^5\mu^6\xi_0^5|\partial_{x_1}\eta|^2+s^3\mu^4\xi_0^3|\partial^2_{x_1}\eta|^2\\
+(1+\alpha^2)s^3\mu^4\xi_0^3|\partial_t\eta|^2+s\mu^2\xi_0|\partial^3_{x_1}\eta|^2+(1+\alpha^2)s\mu^2\xi_0|\partial_t\partial_{x_1}\eta|^2\Big)\\
\leq c_0\Big(\iint_{(0,T)\times J}e^{2s\varphi_0}\Big(s^7\mu^8\xi^7_0|\eta|^2+(1+\alpha^2)s^3\mu^4\xi_0^3|\partial_t\eta|^2+s\mu^2\xi_0|\partial^3_{x_1}\eta|^2+(1+\alpha^2)s\mu^2\xi_0|\partial_t\partial_{x_1}\eta|^2\Big)\\
+\alpha^2 s^3\mu^4\iint_{(0,T)\times\mathcal{I}}\xi_0^3~e^{2s\varphi_0}|\partial_t \eta|^2+\iint_{(0,T)\times\mathcal{I}}e^{2s\varphi_0}|\partial^2_t\eta-\alpha\partial^2_{x_1}\partial_t\eta+\partial^4_{x_1}\eta|^2\Big).
\end{multline} 
\end{theorem}

This Carleman estimate is valid for $\alpha=0$, but with observation terms of order strictly greater than zero (i.e, $\partial^3_{x_1}\eta$, $\partial_{x_1}\partial_t \eta$) on $J$. As a consequence, we cannot obtain the expected controllability result for the beam equation from this inequality.

However, note that in the case $\alpha = 0$, Theorem \ref{poutre-theorem3} implies immediately the following unique continuation property: if $\eta$ satisfies $\partial^2_t \eta + \partial^4_{x_1} \eta = 0$ on $(0,T)\times \mathcal I$, and $\eta \equiv 0$ on $(0,T) \times J$, then $\eta \equiv 0$ on $(0,T) \times \mathcal I$.

The proof of Theorem \ref{poutre2} is given in Section \ref{section_proof_CarlEstpoutre_alphanonneg}.

\subsubsection{The damped beam equation}

We now focus on the case with active damping, \textit{i.e.} $\alpha > 0$. In that context, we obtain the following Carleman estimate.

\begin{theorem}\label{poutre}
 Let $\alpha>0$. Let $J$ be a nonempty open subset of $\mathcal{I}$ and let $k\geq 2$. There exist $c_0>0$, $\mu_0\geq 1$ and $\widehat{s}_0\geq 1$ such that for all $\mu\geq\mu_0$, for all $\alpha>0$, for all $T>0$ and for all $s\geq \widehat{s}_0(1+\alpha)(T^k+T^{k-1})$, the following inequality holds for all function $\eta\in C^2([0,T];C^4(\mathcal{I}))$:
\begin{multline}\label{poutre-theorem}
\frac{\alpha^2}{1+\alpha^2}\frac1s\iint_{(0,T)\times\mathcal{I}}\frac1\xi_0~e^{2s\varphi_0}(|\partial^4_{x_1}\eta|^2+|\partial^2_t\eta|^2)+\frac{\alpha^2}{s}\iint_{(0,T)\times\mathcal{I}}\frac1\xi_0~e^{2s\varphi_0}|\partial_t\partial^2_{x_1}\eta|^2\\
+\iint_{(0,T)\times\mathcal{I}}e^{2s\varphi_0}\Big(s^7\mu^8\xi^7_0|\eta|^2+s^5\mu^6\xi_0^5|\partial_{x_1}\eta|^2+s^3\mu^4\xi_0^3|\partial^2_{x_1}\eta|^2\\+(1+\alpha^2)s^3\mu^4\xi_0^3|\partial_t\eta|^2+s\mu^2\xi_0|\partial^3_{x_1}\eta|^2+(1+\alpha^2)s\mu^2\xi_0|\partial_t\partial_{x_1}\eta|^2\Big)\\
\leq c_0\Big(\left(\alpha^{-8}+\alpha^4\right)\iint_{(0,T)\times J}s^7\mu^8\xi^7_0~e^{2s\varphi_0}|\eta|^2+\iint_{(0,T)\times\mathcal{I}}e^{2s\varphi_0}|\partial^2_t\eta-\alpha\partial^2_{x_1}\partial_t\eta+\partial^4_{x_1}\eta|^2\\
+\alpha^2 s^3\mu^4\iint_{(0,T)\times\mathcal{I}}\xi_0^3~ e^{2s\varphi_0}|\partial_t\eta|^2\Big).
\end{multline} 
\end{theorem}

We now make several comments on the result: first of all, estimate \eqref{poutre-theorem} is better than estimate \eqref{poutre-theorem3} because, in \eqref{poutre-theorem}, we retrieve an observation of uniquely $\eta$ on $J$. However, it is also clear that the estimate \eqref{poutre-theorem} is uniquely valid for $\alpha$ positive, as the inverse of $\alpha^8$ appears on the right-hand side.

Nonetheless, estimate \eqref{poutre-theorem} is not sufficient to obtain the observability of the damped beam equation. This is due to the presence of the last term in the right-hand side, involving an observation of $\partial_t \eta$ on the whole domain $\mathcal I$. 
For large values of the parameter $\alpha$, this term cannot be absorbed into the left-hand side, regardless of the choice of the parameters $s$ and $\mu$. In contrast, the same term can be absorbed for $\alpha$ small enough. It is therefore natural to try to determine the range of paramaters $\alpha$ that can be directly absorbed into the left-hand side. An answer is provided in Theorem \ref{poutre3} below.
In our study, we get rid of this term by coupling the Carleman estimate \eqref{poutre-theorem} with a Carleman estimate for the heat equation. Indeed, the term $\partial_t  \eta$ is directly linked to the heat solution via the boundary conditions (see \eqref{eq: parabolic adjoint}).  In other words, we deeply use the coupling between the heat equation and the beam equation.
\medskip

In \cite{Sourav}, the author obtain a Carleman estimate for the one-dimensional damped beam equation similar to \eqref{poutre-theorem} but without the last term on the right-hand side involving $\partial_t \eta$, with the particular choice of damping parameter $\alpha = 1$. It is natural to ask whether we can generalize the result of \cite{Sourav} for any positive value of the parameter $\alpha$, which would give the observability of the beam equation alone. It turns out that this is not the case, as the techniques developed in \cite{Sourav} fail for large values of the parameter $\alpha$. More precisely, setting
\begin{equation}\label{DefAlphastar}
\alpha^*=\sqrt{\frac{6(2-\beta^*)}{4+\beta^*}}\quad \mbox{where $\beta^*$ is the unique real root of $ \;36\beta^3+111\beta^2+77\beta-58$},
\end{equation}
which gives $\beta^* =  0,437765644120981\pm 10^{-15}$ and $\alpha^*=1,45333768702221\pm 10^{-15}$, we obtain using the strategy of \cite{Sourav} the following Carleman estimate.
\begin{theorem}\label{poutre3}
 Let $\alpha\in (0,\alpha^*)$. Let $J$ be a nonempty open subset of $\mathcal{I}$ and let $k\geq 2$. There exist $c_0>0$, $\mu_0\geq 1$ and $s_0\geq 1$ such that for all $\mu\geq\mu_0$, for all $T>0$ and for all $s\geq s_0(T^k+T^{k-1})$, the following inequality holds for all function $\eta\in C^2([0,T];C^4(\mathcal{I}))$:
\begin{multline}\label{poutre-theorem2}
\iint_{(0,T)\times\mathcal{I}}e^{2s\varphi_0}\Big(s^7\mu^8\xi^7_0|\eta|^2+s^5\mu^6\xi_0^5|\partial_{x_1}\eta|^2+s^3\mu^4\xi_0^3|\partial^2_{x_1}\eta|^2+s^3\mu^4\xi_0^3|\partial_t\eta|^2+s\mu^2\xi_0|\partial^3_{x_1}\eta|^2+s\mu^2\xi_0|\partial_t\partial_{x_1}\eta|^2\Big)\\
+
\frac1s\iint_{(0,T)\times\mathcal{I}}\frac1\xi_0~e^{2s\varphi_0}(|\partial^4_{x_1}\eta|^2+|\partial^2_t\eta|^2+|\partial_t\partial^2_{x_1}\eta|^2)\\
\leq c_0\Big(\iint_{(0,T)\times J}s^7\mu^8\xi^7_0~e^{2s\varphi_0}|\eta|^2+\iint_{(0,T)\times\mathcal{I}}e^{2s\varphi_0}|\partial^2_t\eta-\alpha\partial^2_{x_1}\partial_t\eta+\partial^4_{x_1}\eta|^2\Big).
\end{multline} 
\end{theorem}

This result classically leads to the observability of the corresponding damped beam equation, for any $\alpha$ in the range $(0,\alpha^*)$. To the best of our knowledge, 
obtaining a global Carleman estimate of type \eqref{poutre-theorem2}, valid for large values of the damping parameter $\alpha$, remains an open problem. We underline that a Carleman estimate can be obtained by factorizing the Beam operator as two heat-like operator. But this leads to an inequality with strictly lower exponents of $s$ than as in \eqref{poutre-theorem2} (see e.g. \cite{Fu2020}) that does not permit to handle coupled system such as \eqref{eq: parabolic adjoint}.

\begin{remark}\label{Rkcout2}
Although the dependence of $s$ is not explicitly given in \cite{Sourav}, a check of the proof shows that $s$ can be assumed to satisfy $s\geq s_0(T^4+T^2)$ there, for some $s_0$ independent on $T$. This would lead to a cost of the control proportional to $\exp(C/T^2)$. Here we improve this dependence in $T$: for $k\geq 2$ the parameter $s$ can be chosen $s\geq s_0(T^k+T^{k-1})$ which leads to a cost of the control proportional to $\exp(C/T)$ (see Theorem \ref{observability}).
\end{remark}

\subsection{Scientific context}\label{Sect1.4}

The main contributions that are closely connected to the present work are \cite{Buffe, Buffe2025} where the authors investigate the local null-controllability of an incompressible fluid in a periodic channel, described by the Navier-Stokes equations, with a structure located on the boundary of the fluid domain, described by a damped beam equation. In \cite{Buffe} two distributed controls are considered, one in the fluid domain and one on the beam. In \cite{Buffe2025} the authors manage to suppress the control on the structure by using a Fourier decomposition, but this approach is very specific to the geometry. We underline that in \cite{Buffe, Buffe2025} only the particular case of a structural damping coefficient $\alpha=1$ is considered. We may also mention \cite{buffe2022} where the damped beam equation is replaced by a heat equation. 

The controllability properties of fluid-structure interaction systems have been tackled mainly in
the case where the structure is described by a finite dimensional equation. For the case of a rigid body immersed in a fluid we refer to \cite{MG, MA, UT, IA, Raymond-Vanninathan2009}. For a fluid coupled with a structure described by a system of ordinary differential equations corresponding to a ﬁnite dimensional approximation of equations modeling deformations of an
elastic body we refer to \cite{Raymond-Vanninathan2010, Lequeurre2013}. 

Let us also mention some works concerning the controllability of the beam equation. In \cite{Sourav}, the null-controllability of the damped beam equation is proved, when the damping parameter is equal to one. The proof relies on a Carleman inequality of type \eqref{poutre-theorem2} for $\alpha=1$. In \cite{Fu2020}, a Carleman inequality for the Beam equation is obtained by factorizing the Beam equation as two heat-like equations. However, this results in a Carleman inequality with strictly lower exponents for s than in \eqref{poutre-theorem2}. We also mention the reference \cite{Avd} where the authors proved the null-controllability for the beam equation with structural damping using the moments method, and the reference \cite{Miller} where the author proves the same result using a spectral approach.

\bigskip

\indent The outline of the article is as follows: in Section \ref{section_functional_setting}, we precise the functional setting associated with our controllability problem. In Section \ref{SectCarlemanEst}, we
obtain the Carleman estimate for the coupled system, by coupling a Carleman estimate for the heat equation with a Carleman estimate for the damped beam equation. We deduce the observability inequality of the adjoint system in Section \ref{section_observability}. Finally, Section \ref{carleman-poutre} is devoted to the proofs of several Carleman estimates for the damped or undamped beam equation.

\section{Functional setting} \label{section_functional_setting}
We define the following operators associated with the damped beam equation \eqref{eq: parabolic}. Let $\Lambda : L^2(\mathcal{I})\to L^2(\partial\Omega)$ defined by  
\begin{equation}\label{defLambda}
(\Lambda\zeta)(x_1,x_2)= \left\{\begin{array}{ll}\zeta(x_1)& \mbox{ if }(x_1,x_2)\in \Gamma_1,\medskip  \\0& \mbox{ if }(x_1,x_2)\in \Gamma_0.\end{array}\right.
\end{equation}
The adjoint $\Lambda^*: L^2(\partial\Omega) \to L^2(\mathcal{I})$ is given by
\begin{equation}\label{defLambdastar}
(\Lambda^* w)(x_1)= w(x_1,1)\quad (x_1\in \mathcal{I}). 
\end{equation}
We introduce the following linear operators $A_i$ with domains $\mathcal{D}(A_i)$ (for $i=1,2$) defined on $L^2(\mathcal{I})$,
\begin{equation}\label{DefA1}
    A_1\zeta=\partial_{x_1}^4\zeta+\zeta,~~~~~~~~\mathcal{D}(A_1)=H^4(\mathcal{I})
\end{equation}
\begin{equation}\label{DefA2}
A_2\zeta=\alpha\left(-\partial_{x_1}^2\zeta+\zeta\right),~~~~~~~~\mathcal{D}(A_2)=H^2(\mathcal{I}).
\end{equation}
One can check that
$$
\mathcal{D}(A_1^{1/2})=H^2(\mathcal{I}).
$$
Then the damped beam equation in (\ref{eq: parabolic}) writes as:
$$\partial^2_t\zeta+A_1\zeta+ A_2\partial_t\zeta=-\Lambda^* \partial_nw+\mathbbm{1}_{J} h .$$
We define the space 
$$\mathscr{H}=\left\{w\in L^2(\Omega),~~\zeta_1\in \mathcal{D}(A_1^{1/2}),~~\zeta_2\in L^2(\mathcal{I})\right\}$$
equipped with the scalar product:
$$\left((w,\zeta_1,\zeta_2),(\widetilde{w},\widetilde{\zeta}_1,\widetilde{\zeta}_2)\right)_{\mathscr{H}}=\int_\Omega w~\widetilde{w}+\int_{\mathcal{I}}A_1^{1/2} \zeta_1 ~ A_1^{1/2} \widetilde{\zeta}_1+\int_{\mathcal{I}}\zeta_2~ \widetilde{\zeta}_2.$$
We define the linear operator $\mathcal{A}:\mathcal{D}(\mathcal{A})\subset \mathscr{H}\to \mathscr{H}$ by
\begin{equation}\label{non-b-opDomain}
\mathcal{D}(\mathcal{A})=\left\{(w,\zeta_1,\zeta_2)\in H^2(\Omega)\times H^4(\mathcal{I})\times H^2(\mathcal{I})~:~w=\Lambda \zeta_2~~\text{on}~\partial\Omega\right \}.
\end{equation}
and 
\begin{equation}\label{non-b-op}
\mathcal{A}\begin{bmatrix}
w\\
\zeta_1\\
\zeta_2
\end{bmatrix}=\begin{bmatrix}
  \Delta w\\
  \zeta_2\\
  -A_1\zeta_1-A_2\zeta_2-\Lambda^*\partial_nw
\end{bmatrix}.
\end{equation}
From an easy adaptation of the proof of \cite[Propositions 3.4, 3.5 and 3.11]{MT} we deduce the following result.
\begin{proposition}\label{PropAnalytiqueEtStable}
The operator $\mathcal{A}$ defined by \eqref{non-b-opDomain}, \eqref{non-b-op} is the infinitesimal generator of an analytic and uniformly stable semigroup on  $\mathscr{H}$.
\end{proposition}

A consequence of Proposition \ref{PropAnalytiqueEtStable} is that for $F\in L^2(0,T;\mathscr{H})$ and $Y_0\in \mathscr{H}$, the linear system 
\begin{equation} \label{eq_forwardprob_semigroupform0}
\frac{d Y}{dt}=\mathcal{A}Y+F~~\text{in}~~(0,T),~~Y(0)=Y_0,
\end{equation}
admits a unique ''classical solution", that is a solution $Y$ that belongs to $H^1(\varepsilon,T;\mathscr{H})\cap L^2(\varepsilon,T;\mathcal{D}(\mathcal{A}))\cap C([0,T];\mathscr{H})$ for all $\varepsilon>0$ (see \cite[Definition 3.1 p.129 and Proposition 3.8 p.145]{BDDM}). 

In particular, for $(w_0,\zeta_0,\zeta_1) \in L^2(\Omega) \times H^2(\mathcal I) \times L^2(\mathcal I)$ and $(G,H) \in L^2(Q) \times  L^2((0,T) \times \mathcal I)$, if we set
\begin{equation}\label{VarYetY0}
Y=\begin{bmatrix}
w\\
\zeta\\
\partial_t\zeta
\end{bmatrix},\qquad Y_0=\begin{bmatrix}
    w_0\\
    \zeta_0\\
    \zeta_1
\end{bmatrix}
\end{equation}
and 
$$
F=\begin{bmatrix}
G\\
0\\
H
\end{bmatrix}
$$
then system \eqref{eq: forward system} can be rewritten in the form \eqref{eq_forwardprob_semigroupform0}, and it implies the existence and uniqueness of a solution to \eqref{eq: forward system} satisfying \eqref{RegClassic1} and \eqref{RegClassic2}.

Using that $0$ is in the resolvent set of $\mathcal A$ (by Proposition \ref{PropAnalytiqueEtStable}), we deduce the following result:
\begin{proposition}\label{positivité-de-A}
There exists a constant $c>0$, such that for all non-negative values of the parameter $\alpha$,
any solution $Y = (w,\zeta,\partial_t\zeta)$ of \eqref{eq_forwardprob_semigroupform0} with $F = 0$ satisfies
\begin{equation}\label{estimate-MT}
\|\partial_t[w, \zeta, \partial_t\zeta](t)\|_{\mathscr{H}}=\|\mathcal{A}[w, \zeta, \partial_t\zeta](t)\|_{\mathscr{H}}\geq \frac{c}{1+\alpha}\|[w, \zeta, \partial_t\zeta](t)\|_{\mathscr{H}}\qquad \forall t\in (0,T).
\end{equation} 
\end{proposition}
\noindent Proposition \ref{positivité-de-A} can be proven similarly as \cite[Proposition 3.5]{MT}. In particular, by following the same steps as in the proof of \cite[Proposition 3.5]{MT} we observe that the constant in \eqref{estimate-MT} can be chosen of the form $\frac{c}{1+\alpha}$ with $c>0$ independent of $\alpha$.

If $Y_0=0$ we also have the following usefull maximal regularity property: for $F\in L^2(0,T;\mathscr{H})$, the solution of 
\begin{equation} \label{eq_forwardprob_semigroupform}
\frac{d Y}{dt}=\mathcal{A}Y+F~~\text{in}~~(0,T),~~Y(0)=0,
\end{equation}
belongs to $H^1(0,T;\mathscr{H})\cap L^2(0,T;\mathcal{D}(\mathcal{A}))$ and satisfies
\begin{equation}\label{estimate-semig}
  \|Y\|_{L^2(0,T;\mathcal{D}(\mathcal{A}))}+\|Y\|_{H^1(0,T; \mathscr{H})}\leq  M_\alpha \|F\|_{L^2(0,T; \mathscr{H})}.
\end{equation}
{Note that it is not straightforward to explicit the dependence in $\alpha$ of the constant $M_\alpha$ in \eqref{estimate-semig} from the resolvent estimate in \cite[Proposition 3.11]{MT}. However, by adapting the approach of \cite[Section 3]{M-T-2022} one can obtain that it is of the form $M_\alpha=c(\alpha^{-1}+\alpha)$ with $c>0$ independent of $\alpha>0$.

\smallskip

Next, we introduce the control operator $B:L^2(\omega)\times L^2(J)\to \mathscr{H}$ defined by
\begin{equation}\label{DefControlOperator}B(g,h)=\begin{bmatrix}
\mathbbm{1}_{\omega} g\\
0\\
\mathbbm{1}_{J} h \end{bmatrix}.
\end{equation}
Hence,  with \eqref{VarYetY0} system (\ref{eq: parabolic}) can be rewritten in the following form:
\begin{equation}\label{sys-simple}
\frac{d Y}{dt}=\mathcal{A}Y+B(g,h)~~\text{in}~~(0,T),~~Y(0)=Y_0.
\end{equation}

\begin{definition}
We say that system \eqref{sys-simple} is null-controllable in time $T>0$ if for any $Y_0\in\mathscr{H},$ there exists a control $(g,h)\in L^2(0,T; L^2(\omega)\times L^2(J))$ such that the solution of system \eqref{sys-simple} satisfies $Y(T)=0$.
\end{definition}

\begin{proposition} 
    The adjoint of the operator $\mathcal{A}$ is given by $\mathcal{D}(\mathcal{A}^*)=\mathcal{D}(\mathcal{A})$ and 
    \begin{equation}\label{non-b-op-adj}
\mathcal{A}^*\begin{bmatrix}
u\\
\eta_1\\
\eta_2
\end{bmatrix}=\begin{bmatrix}
  \Delta u\\
  -\eta_2\\
  A_1\eta_1-A_2\eta_2-\Lambda^*\partial_n u
\end{bmatrix}.
\end{equation}
\end{proposition}
\begin{proof}
First, let us prove $\mathcal{D}(\mathcal{A}^*)\subset \mathcal{D}(\mathcal{A})$. For that, suppose that $[u,\eta_1,\eta_2]\in \mathcal{D}(\mathcal{A}^*)$ and set $[F,G,H]=\mathcal{A}^*[u,\eta_1,\eta_2]\in \mathscr{H}$. We have
$$\left(\mathcal{A}[w,\zeta_1,\zeta_2],[u,\eta_1,\eta_2]\right)_{\mathscr{H}}=\left([w,\zeta_1,\zeta_2], [F,G,H]\right)_{\mathscr{H}}\quad \forall ~[w,\zeta_1,\zeta_2]\in \mathcal{D}(\mathcal{A}).$$
The above equality rewrites
\begin{multline}\label{Eq1proofAstar}
\int_\Omega\Delta w~u+\int_{\mathcal{I}}A_1^{1/2} \zeta_2 ~ A_1^{1/2}\eta_1-\int_{\mathcal{I}}(\Lambda^*\partial_n w+A_1\zeta_1+A_2\zeta_2)\eta_2\\
=\int_\Omega w~F+\int_{\mathcal{I}} A_1^{1/2}\zeta_1~ A_1^{1/2} G + \int_{\mathcal{I}}H~\zeta_2.
\end{multline}

First, with $w=0$ and $\zeta_2=0$ we find
$$\int_{\mathcal{I}}A_1 \zeta_1(G+\eta_2)=0 \quad\forall \zeta_1\in H^4(\mathcal{I}).$$ 
from which we deduce that $\eta_2=-G\in \mathcal{D}(A_1^{1/2})$. 

Next, with $\zeta_1=\zeta_2=0$ and $w\in C^\infty_c(\Omega)$, we obtain
$$\int_\Omega u \Delta w=\int_\Omega F w\quad \forall w\in C^\infty_c(\Omega).$$ 
from which we deduce $\Delta u=F\in L^2(\Omega)$. It implies that the trace and the normal derivative of $u$ are well defined in $H^{-1/2}(\partial\Omega)$ and in $H^{-3/2}(\partial\Omega)$ respectively, and,  that the following formula is true for all $w \in H^2(\Omega)$,
\begin{equation}\label{ipp}
\int_\Omega\Delta w\,u-\int_\Omega \Delta u\,w=\langle\partial_n w,u\rangle_{H^{1/2}(\partial\Omega),H^{-1/2}(\partial\Omega)}-\langle w,\partial_n u\rangle_{H^{3/2}(\partial\Omega),H^{-3/2}(\partial\Omega)}.
\end{equation}
Thus, by setting $F=\Delta u$ in \eqref{Eq1proofAstar}, and with $\zeta_1=0$, $\zeta_2\in H^2(\mathcal{I})$ and $w\in H^2(\Omega)$ such that $w=\Lambda \zeta_2$ on $\partial\Omega$, we deduce
\begin{equation}\label{f1}
\int_\Omega\Delta w\,u-\int_\Omega w\,\Delta u+\int_{\mathcal{I}} A_1^{1/2} \zeta_2 \, A_1^{1/2}\eta_1-\int_{\partial\Omega}\partial_n w\,\Lambda \eta_2-\int_{\mathcal{I}}A_2\zeta_2\,\eta_2=\int_{\mathcal{I}}\zeta_2\,H.
\end{equation}
Then with \eqref{ipp}, the equality \eqref{f1} rewrites:
\begin{multline}\label{f2}
\langle\partial_n w,u\rangle_{H^{1/2}(\partial\Omega),H^{-1/2}(\partial\Omega)}-\langle w,\partial_n u\rangle_{H^{3/2}(\partial\Omega),H^{-3/2}(\partial\Omega)}\\+\int_{\mathcal{I}} A_1^{1/2} \zeta_2 \, A_1^{1/2}\eta_1-\int_{\partial\Omega}\partial_n w\,\Lambda \eta_2-\int_{\mathcal{I}}A_2\zeta_2\,\eta_2=\int_{\mathcal{I}}\zeta_2\,H.
\end{multline}
Let $d\in H^{1/2}(\partial\Omega)$. By using the surjectivity of the map $w\in H^2(\Omega)\mapsto (w,\partial_n w)\in H^{3/2}(\partial\Omega)\times H^{1/2}(\partial\Omega)$ we can choose $w\in H^2(\Omega)$ such that $\partial_n w=d$ on $\partial\Omega$. Hence by applying \eqref{f2} for such a $w$ and with $\zeta_2=0$ on $\mathcal{I}$ we finally deduce that
\begin{equation}\label{f3} 
\langle d,u\rangle_{H^{1/2}(\Gamma_1),H^{-1/2}(\Gamma_1)}-\int_{\partial\Omega}d \,\Lambda \eta_2=0 \quad \forall d\in H^{1/2}(\partial\Omega).
\end{equation}
Then we deduce that $u=\Lambda \eta_2$ on $\partial\Omega$. Moreover, since $\eta_2\in H^2(\mathcal{I})$ it is clear that $u=\Lambda \eta_2\in H^{2}(\partial\Omega)$. Since we also have $\Delta u\in L^2(\Omega)$, elliptic regularity results yield $u\in H^2(\Omega)$. In particular, we have $\partial_n u\in H^{1/2}(\partial\Omega)$ and $\Lambda^*\partial_n u\in H^{1/2}(\mathcal{I})$.

Finally, by coming back to \eqref{f2}, that holds for all $w\in H^2(\Omega)$ and $\zeta_2\in H^2(\mathcal{I})$ such that $w=\Lambda \zeta_2$ on $\partial\Omega$, and by taking into account that $u=\Lambda\eta_2$ on $\partial\Omega$, we find that 
$$
-\int_{\partial\Omega} \Lambda \zeta_2\,\partial_n u+\int_{\mathcal{I}} A_1^{1/2} \zeta_2 \, A_1^{1/2}\eta_1-\int_{\mathcal{I}}A_2\zeta_2\,\eta_2=\int_{\mathcal{I}}\zeta_2\,H
.$$
Then $A_1 \eta_1=H+A_2\eta_2 +\Lambda^*\partial_n u\in L^2(\mathcal{I})$ from which we deduce $\eta_1\in \mathcal{D}(A_1)$. 

If we summarize the above results, we have proved $(u,\eta_1,\eta_2)\in H^2(\Omega)\times \mathcal{D}(A_1)\times \mathcal{D}(A_1^{1/2})$ and $u=\Lambda\eta_2$ on $\partial\Omega$, and that $F=\Delta u$, $G=-\eta_2$ and $H=A_1 \eta_1-A_2\eta_2 -\Lambda^*\partial_n u$. Then with  \eqref{non-b-opDomain} it means that $(u,\eta_1,\eta_2)\in \mathcal{D}(\mathcal{A})$, and we have proved $\mathcal{D}(\mathcal{A}^*)\subset \mathcal{D}(\mathcal{A})$. Moreover, since $(F,G,H)=\mathcal{A^*}(u,\eta_1,\eta_2)$ we have that \eqref{non-b-op-adj} is true. 

To conclude,  the converse inclusion $\mathcal{D}(\mathcal{A})\subset \mathcal{D}(\mathcal{A}^*)$ follows from an application of the Green formula. 
\end{proof}

Next, the adjoint $B^*:\mathscr{H}\to L^2(\omega)\times L^2(J)$, of the operator $B$ defined in \eqref{DefControlOperator}, is given by
\begin{equation}\label{DefControlOperatorAdj}B^*\begin{bmatrix}
u\\
\eta_1\\
\eta_2 \end{bmatrix}
=\begin{bmatrix}
u|_\omega\\
\eta_2|_J \end{bmatrix}.
\end{equation}
Let us consider the following adjoint system:
\begin{equation}\label{sys-simple_eta}
\frac{d V}{dt}=\mathcal{A}^*V~~\text{in}~~(0,T).
\end{equation}

It is well-known (see e.g. \cite[Theorem 11.2.1, p.357]{Tuc}) that the null-controllability is equivalent to the final-state observability of the adjoint system, that is: system \eqref{sys-simple} is null-controllable in time $T>0$ if, and only if, there exist $C_T>0$ such that any solution of \eqref{sys-simple_eta} satisfies
\begin{equation}\label{obs-form-simple0}
    \|V(T)\|^2_{\mathscr{H}}\leq C_T\int_0^T\|B^* V\|_{L^2(\omega)\times L^2(J)}^2.
\end{equation}
In view of \eqref{non-b-op-adj} the above equality is equivalent to
\begin{align*}
\begin{cases}
\partial_t u =\Delta u &\text{in}~Q,\\
u=\Lambda \eta_2  &\text{on}~(0,T) \times\partial\Omega,\\
\partial_{t}\eta_1=-\eta_2 & \text{in}~(0,T) \times\mathcal{I},\\
\partial_{t}\eta_2=A_1\eta_1-A_2\eta_2-\Lambda^*\partial_n u & \text{in}~(0,T) \times\mathcal{I}.
\end{cases}
\end{align*}
Hence, we deduce that $(u, \eta)=(u,-\eta_1)$ satisfies
\begin{equation}\label{sys-simple_eta-ode}
\frac{d}{dt}\begin{bmatrix}
    u\\
    \eta\\
    \partial_t\eta
\end{bmatrix}=\mathcal{A}\begin{bmatrix}
    u\\
    \eta\\
    \partial_t\eta
\end{bmatrix}~~~\text{in}~~(0,T)
\end{equation}
 or equivalently,
\begin{align}\label{adjoint-final}
\begin{cases}
\partial_t u =\Delta u &\text{in}~Q,\\
u=\Lambda \partial_t\eta  &\text{on}~(0,T) \times\partial\Omega,\\
\partial_{t}^2\eta+A_1\eta+A_2\partial_t \eta=-\Lambda^*\partial_n u & \text{in}~(0,T) \times\mathcal{I}.\\
\end{cases}
\end{align}
Then the null-controllabitity of system \eqref{sys-simple} in time $T>0$ is equivalent to the existence of $C_T>0$ such that any solution of \eqref{adjoint-final} satisfies

\begin{equation}\label{obs-form-simple}
    \|[u,\eta,\partial_t\eta](T)\|^2_{\mathscr{H}}\leq C_T\left(\int_0^T\int_\omega |u|^2+\int_0^T\int_J|\partial_t\eta |^2\right).
\end{equation}


\section{Carleman estimate for the coupled heat and damped beam system}\label{SectCarlemanEst}
In this section we state the Carleman estimate for the heat equation that we combine with Theorem \ref{poutre} to deduce Theorem \ref{Carleman-couplé}, and we prove the observability inequality for the adjoint system \eqref{eq: parabolic adjoint}.

\paragraph{Notation.}
In what follows, we use the notation $X\lesssim Y$ and $X\gtrsim Y$ if there exists a constant $C>0$ which is independent of
the parameters $(s,\lambda,\mu,T,\alpha)$ such that we have the inequalities $X\leq C Y$ and $X\geq C Y$ respectively.

\subsection{About some properties of the weight functions}\label{weight-function-for-Carleman}

We now state some properties that are satisfied by the functions  $\varphi$, $\xi$, $\varphi_0$ and $\xi_0$
defined by equations \eqref{weights2} and \eqref{weights1} in the introduction, and that we need in the following calculations.

Let $b>a$, $i=1,2$ and $j\in \mathbb{N}$. The following properties are true.

\begin{align}\label{prop-phi}
|\partial_{x_i}^j\varphi|+|\partial_{x_i}^j\xi|&\lesssim \lambda^j\xi,\\\notag
|\partial_t\partial_{x_i}^j\varphi|+|\partial_t\partial_{x_i}^j\xi|&\lesssim \lambda^j T\xi^{1+\frac{2}{k}},\\ \notag
|\partial^2_t\partial_{x_i}^j \varphi|+|\partial^2_t\partial_{x_i}^j \xi|&\lesssim \lambda^j T^2\xi^{1+\frac4k},\\ \notag
\xi^a&\leq T^{(b-a)k}~\xi^b,
\end{align}
and
\begin{align}
|\partial^j_{x_1}\varphi_0|+|\partial^j_{x_1}\xi_0|\lesssim \mu^j\xi_0,\label{poutre1} 
\\
\label{n2}
|\partial_t\partial^j_{x_1}\xi_0|+|\partial_t\partial^j_{x_1}\varphi_0|&\lesssim\mu^j~T~\xi_0^{1+\frac{2}{k}},\\ \label{n3}
|\partial^2_t\partial^j_{x_1}\xi_0|+|\partial^2_t\partial^j_{x_1}\varphi_0|&\lesssim\mu^j~T^2~\xi_0^{1+\frac{4}{k}},\\
\xi_0^a \leq T^{(b-a)k}~\xi^b_0.
\end{align}

For the following we suppose that for some $s_0\geq 1$ we have $s\geq s_0 (T^k+T^{k-1})$. Then the above estimates yield
\begin{align}
\label{n2-1}
|\partial_t\partial^j_{x_i}\xi|+|\partial_t\partial^j_{x_i}\varphi|&\lesssim\lambda^j\left(s/s_0\right)~\xi^2,\\ \label{n3-1}
|\partial^2_t\partial^j_{x_i}\xi|+|\partial^2_t\partial^j_{x_i}\varphi|&\lesssim\mu^j~\left(s/s_0\right)^2~\xi^3,\\
\xi^a \lesssim \left(s/s_0\right)^{b-a}~\xi^b.\label{n1-1}
\end{align}
and
\begin{align}
\label{n2-2}
|\partial_t\partial^j_{x_1}\xi_0|+|\partial_t\partial^j_{x_1}\varphi_0|&\lesssim\mu^j\left(s/s_0\right)~\xi_0^2,\\ \label{n3-2}
|\partial^2_t\partial^j_{x_1}\xi_0|+|\partial^2_t\partial^j_{x_1}\varphi_0|&\lesssim\mu^j~\left(s/s_0\right)^2~\xi_0^3,\\
\xi_0^a\lesssim \left(s/s_0\right)^{b-a}~\xi^b_0.\label{n1-2}
\end{align}
Moreover there exists $\mu_0\geq 1$ such that for $\lambda\geq\mu\geq\mu_0$, for $t\in [0,T]$ and for $x_1\in\mathcal{I}\smallsetminus J_0$ (where $J_0\Subset J$),
\begin{equation}\label{poutre2-0}
\mu\xi_0\lesssim |\partial_{x_1}\varphi_0|,~~ \mu^2\xi_0\lesssim \partial^2_{x_1}\varphi_0.
\end{equation}

\subsection{Carleman estimate for the heat equation}\label{preuve-carleman} 
 We introduce the following boundary integral, where we have denoted $\Sigma = (0,T) \times \partial \Omega$:
\begin{align}\label{termes-au-bord}\notag
I_{\Sigma}(s,\lambda,\mu,v)&=-\iint_\Sigma\frac{\partial v}{\partial n}~\partial_tv+s^3\lambda^3\iint_\Sigma\xi_0^3~|v|^2-s\lambda^2\iint_\Sigma\xi_0~\partial_n v~v\\ 
&+\lambda s^2\iint_\Sigma\xi_0~\partial_t\varphi~|v|^2+\lambda s\iint_\Sigma\xi_0~|\partial_n v|^2-\lambda s\iint_\Sigma\xi_0~|\partial_{x_1} v|^2\\ \notag
&+\mu s\iint_\Sigma\xi_0~\partial_n v~\psi'_{\mathcal{I}}~\partial_{x_1}v.
\end{align}
\begin{theorem}\label{chaleur}
    Let $\omega$ be a non-empty open subset of $\Omega$ and let $k\geq 2$. There exist $c_0>0$
    and $s_0\geq 1$ such that for all $\lambda$, $\mu$, $\lambda\geq \mu\geq 1$,
    for all $T>0$ and for all $s\geq s_0(T^k+T^{k-1}),$ the following inequality holds for all $u\in C^1([0,T];C^2(\overline{\Omega})$ such that with $u=0$ on $(0,T)\times\Gamma_0:$
\begin{align} 
& s^{-1}\iint_{(0,T)\times\Omega}~\xi^{-1}~e^{2s\varphi}\left(|\partial_tu|^2+|\Delta u|^2\right)+s\lambda^2\iint_{(0,T)\times\Omega}~\xi~e^{2s\varphi}|\nabla u|^2\notag\\ 
   &+s^3\lambda^4\iint_{(0,T)\times\Omega}~\xi^3~e^{2s\varphi}|u|^2+I_{\Sigma}(s,\lambda,\mu,v=e^{s\varphi}u)\label{CarlemanHeat}\\ &\leq c_0 
   \left(\iint_{(0,T)\times\Omega}~e^{2s\varphi}|\partial_tu-\Delta u|^2+s^3\lambda^4\iint_{(0,T)\times\omega} \xi^3~~e^{2s\varphi}| u|^2\right),\notag
\end{align}
where $I_{\Sigma}(s,\lambda,\mu,v=e^{s\varphi}u)$ is given by \eqref{termes-au-bord}.
\end{theorem}
Such Carleman estimate for the heat equation can easily be obtained by following the classical techniques developed in \cite{Cara,Fursikov}. The only particularity here is that we keep all the boundary terms coming from the integration by parts as in \cite{Badra}, which gives the term $I_\Sigma$, and will be used in the combination with the Carleman estimate for the damped beam equation.

\subsection{Carleman estimate for the coupled system -- Proof of Theorem \ref{Carleman-couplé}}\label{Sect33}
\begin{proof}[Proof of Theorem \ref{Carleman-couplé}]
We denote by $I^\ell_\Sigma$, $\ell=1,\dots, 7$, the $\ell$th term on the right hand side of \eqref{termes-au-bord}. 
Let $\varepsilon>0$. First, using Young's inequality, we deduce:
\begin{align*}    |I^1_\Sigma|&=\left|\iint_\Sigma\partial_nv~\partial_tv\right|\leq \frac{1}{4\varepsilon s\lambda}\iint_\Sigma\frac1\xi_0~|\partial_tv|^2+\varepsilon s\lambda\iint_\Sigma\xi_0~|\partial_nv|^2.
\end{align*}
Next, using Young's inequality and  \eqref{n1-2}, we obtain
$$
|I^3_\Sigma|=\left|s\lambda^2\iint_\Sigma\xi_0~\partial_n v~v\right|\leq \frac{s\lambda}{2}\iint_\Sigma\xi_0~|\partial_nv|^2+C \frac{1}{s_0^2}s^3\lambda^3\iint_\Sigma\xi_0^3~|v|^2.
$$
Using \eqref{n2-2}, we deduce that,
$$
    |I^4_\Sigma|=\left| s^2\lambda\iint_\Sigma\xi_0~\partial_t\varphi_0~|v|^2\right|\lesssim  \frac{1}{s_0}s^3\lambda \iint_\Sigma\xi_0^3~|v|^2.
$$
Next, from $\lambda\geq \mu$, Young's inequality yields
\begin{multline*}
|I^7_\Sigma+I^6_\Sigma|=\left|-\lambda s\iint_\Sigma\xi_0~|\partial_{x_1} v|^2+\mu s\iint_\Sigma\xi_0~\partial_nv~\psi'_{\mathcal{I}}~\partial_{x_1}v\right|\\
\leq 4\varepsilon s\lambda \iint_\Sigma\xi_0~|\partial_nv|^2+(1+\|\psi_{\mathcal{I}}'\|_{L^\infty(\mathcal{I})}/\varepsilon) s\lambda \iint_\Sigma\xi_0~|\partial_{x_1}v|^2.
\end{multline*}
Then by combining the above inequalities, by fixing $\varepsilon>0$ small enough and for $s_0\geq 1$ large enough, from \eqref{CarlemanHeat} we deduce that 
\begin{align}\label{est7-03} 
& s^{-1}\iint_{(0,T)\times\Omega}~\xi^{-1}~e^{2s\varphi}\left(|\partial_tu|^2+|\Delta u|^2\right)+s\lambda^2\iint_{(0,T)\times\Omega}~\xi~e^{2s\varphi}|\nabla u|^2\notag\\ 
   &+s^3\lambda^4\iint_{(0,T)\times\Omega}~\xi^3~e^{2s\varphi}|u|^2+s^3\lambda^3\iint_\Sigma\xi_0^3~|v|^2+\lambda s\iint_\Sigma\xi_0~|\partial_n v|^2\\ &\leq c_0 
   \Big(\iint_{(0,T)\times\Omega}~e^{2s\varphi}|\partial_tu-\Delta u|^2+s^3\lambda^4\iint_{(0,T)\times\omega} \xi^3~~e^{2s\varphi}| u|^2\notag\\
   & + \frac{1}{s\lambda}\iint_\Sigma\frac1\xi_0~|\partial_tv|^2+\lambda s\iint_\Sigma\xi_0~|\partial_{x_1} v|^2\Big).\notag
\end{align}
From $v=e^{s\varphi}u$ we deduce $\partial_n v= s\partial_n \varphi v+e^{s\varphi} \partial_n u$. Then with \eqref{n1-2} we get
$$
\lambda s\iint_\Sigma\xi_0~e^{2s\varphi_0}|\partial_n u|^2\lesssim \lambda s\iint_\Sigma\xi~|\partial_n v|^2+\frac{1}{s_0}s^3\lambda^3 \iint_\Sigma\xi_0^3~|v|^2.
$$
and \eqref{est7-03} yields
\begin{multline}\label{EstH-1}
H(s, \lambda, u) \leq c_0 
   \Big(\iint_{(0,T)\times\Omega}~e^{2s\varphi}|\partial_tu-\Delta u|^2+s^3\lambda^4\iint_{(0,T)\times\omega} \xi^3~e^{2s\varphi}| u|^2 \\
   + \frac{1}{s\lambda}\iint_\Sigma\frac1\xi_0~|\partial_tv|^2+\lambda s\iint_\Sigma\xi_0~|\partial_{x_1} v|^2\Big).
\end{multline}
Next, from $v=e^{s\varphi}u$ and $u=\Lambda (\partial_t\eta)$ on $\Sigma$, where $\Lambda$ is defined in \eqref{defLambda}, we deduce 
\begin{align*}
\partial_tv=\partial_t (e^{s\varphi_0}u)&= s(\partial_t \varphi_0) e^{s\varphi_0}u+e^{s\varphi_0} \partial_t u=\Lambda(s(\partial_t \varphi_0) e^{s\varphi_0}\partial_t \eta+e^{s\varphi_0} \partial_t^2 \eta)\quad \mbox{ on $\Sigma$,}\\
\partial_{x_1} v&= s \mu (\psi_{\mathcal{I}}')\xi_0 e^{s\varphi_0}u+e^{s\varphi_0}\partial_{x_1}u= \Lambda(s \mu (\psi_{\mathcal{I}}')\xi_0 e^{s\varphi_0}\partial_t \eta+e^{s\varphi_0}\partial_t\partial_{x_1} \eta)\quad \mbox{on $\Sigma$,}
\end{align*}
and then with \eqref{n1-2} and $\lambda\geq \mu$,
\begin{multline}\label{Est8-03}
\frac{1}{s\lambda}\iint_\Sigma\frac1\xi_0~|\partial_tv|^2+\lambda s\iint_\Sigma\xi_0~|\partial_{x_1} v|^2\lesssim  \frac{1}{s\lambda}\iint_{(0,T)\times \mathcal{I}}  \frac{1}{\xi_0}~e^{2s\varphi_0}|\partial_t^2 \eta|^2\\
+s^3\lambda \mu^2 \iint_{(0,T)\times \mathcal{I}}  \xi_0^3~e^{2s\varphi_0}|\partial_t \eta|^2+s\lambda \iint_{(0,T)\times \mathcal{I}}  \xi_0~e^{2s\varphi_0}|\partial_t\partial_{x_1} \eta|^2\lesssim \left(\frac{1+\alpha^2}{\lambda \alpha^2}+\frac{\lambda}{(1+\alpha^2)\mu^2}\right) B_\alpha(s, \mu, \eta).
\end{multline}
Moreover, from  $u=\Lambda (\partial_t\eta)$ on $\Sigma$ we deduce 
\begin{equation}\label{Eq21-04}
\alpha^2s^3\mu^4\iint_{(0,T)\times\mathcal{I}}\xi_0^3~ e^{2s\varphi_0}|\partial_t\eta|^2=\alpha^2 s^3\mu^4\iint_{\Sigma}\xi_0^3~ e^{2s\varphi_0}|u|^2\leq \frac{\alpha^2 \mu^4}{\lambda^3}H(s, \lambda, u).
\end{equation}
Then by combining \eqref{EstH-1}, \eqref{Est8-03}, \eqref{poutre-theorem}  and \eqref{Eq21-04}  we obtain
\begin{multline}\label{EstHetB-1}
H(s, \lambda, u) +B_\alpha (s, \mu, \eta) \lesssim
   \iint_{(0,T)\times\Omega}~e^{2s\varphi}|\partial_tu+\Delta u|^2+\iint_{(0,T)\times\mathcal{I}}e^{2s\varphi_0}|\partial^2_t\eta-\alpha\partial^2_{x_1}\partial_t\eta+\partial^4_{x_1}\eta|^2 \\
   +s^3\lambda^4\iint_{(0,T)\times\omega} \xi^3~e^{2s\varphi}| u|^2+ \left(\alpha^{-8}+\alpha^4\right)\iint_{(0,T)\times J_0}s^7\mu^8\xi^7_0~e^{2s\varphi_0}|\eta|^2+\\ 
   +\left(\frac{1+\alpha^2}{\lambda \alpha^2}+\frac{\lambda}{(1+\alpha^2)\mu^2}\right) B_\alpha(s, \mu, \eta)+\frac{\alpha^2\mu^4}{\lambda^3}H(s, \lambda, u).
\end{multline}
Moreover, note that from \eqref{n1-2} we deduce
\begin{multline*}
\iint_{(0,T)\times\mathcal{I}}e^{2s\varphi_0}|\partial^2_t\eta-\alpha\partial^2_{x_1}\partial_t\eta+\partial^4_{x_1}\eta|^2\lesssim \iint_{(0,T)\times\mathcal{I}}e^{2s\varphi_0}|\partial^2_t\eta-\alpha\partial^2_{x_1}\partial_t\eta+\alpha\partial_t\eta+\partial^4_{x_1}\eta+\eta+\partial_nu|^2\\+\alpha^2(s/s_0)^3\iint_{(0,T)\times\mathcal{I}}e^{2s\varphi_0}\xi_0^3 |\partial_t\eta|^2+(s/s_0)^7\iint_{(0,T)\times\mathcal{I}}e^{2s\varphi_0}\xi_0^7 |\eta|^2.
\end{multline*}
Then by choosing $s_0$ large enough we can replace  to replace $\partial^2_t\eta-\alpha\partial^2_{x_1}\partial_t\eta+\partial^4_{x_1}\eta$ by $\partial^2_t\eta-\alpha\partial^2_{x_1}\partial_t\eta+\alpha\partial_t\eta+\partial^4_{x_1}\eta+\eta+\partial_nu$ in \eqref{EstHetB-1}. 

Finally, for the choice of parameters $\lambda$, $\mu$ given in the statement of the theorem the last two terms at the right side of the inequality \eqref{EstHetB-1} can be compensated by its left side.
\end{proof}

\section{Observability inequality for the coupled system}\label{section_observability}

In this section, we fix the parameters $\lambda$, $s$ as in Theorem \ref{Carleman-couplé}.
We now prove Proposition \ref{obs-dteta}. First, we recall the definition of the weight functions appearing in inequality \eqref{dteta}:
\begin{equation}\label{poids en t1}
\varphi_1(t)=\frac{e^{8\lambda\Psi}-e^{10\lambda\Psi}}{\ell^{k/2}(t)},~~\xi_1(t)=\frac{e^{8\lambda\Psi}}{\ell^{k/2}(t)},
\end{equation}
\begin{equation}\label{poids en t2}
\varphi_2(t)=\frac{e^{9\lambda\Psi}-e^{10\lambda\Psi}}{\ell^{k/2}(t)},~~\xi_2(t)=\frac{e^{9\lambda\Psi}}{\ell^{k/2}(t)}.
\end{equation}
Observe that, by definition of $\Psi$ (see equation \eqref{eq_def_PSI}), we have 
\begin{equation}\label{OrdreVarphi12}\varphi_1(t)\leq\varphi(t,\cdot)\leq\varphi_2(t),~~\xi_1(t)\leq \xi(t,\cdot)\leq\xi_2(t)~~\text{for}~~t\in(0,T),
\end{equation}
and that, analogously to \eqref{n2-1}-\eqref{n1-2}, we have for $a,b\in\R,$ $k\in\N$ and $j=1,2$:
\begin{align}\label{E1}
 &\xi_j^a\lesssim \left(s/s_0\right)^{b-a}\xi^b_j,\notag\\ 
 &|\partial_t\xi_j|+|\partial_t\varphi_j|\lesssim\left(s/s_0\right) \xi_j^{2}, \\ 
 &|\partial^2_t\xi_j|+|\partial^2_t\varphi_j|\lesssim\left(s/s_0\right)^2 \xi_j^3.\notag
\end{align}

\begin{proof}[Proof of Proposition \ref{obs-dteta}]
First, for the following we set
\begin{equation}\label{DefVarrho01}
\varrho_0(t)=s^{3/2}\lambda~\xi^{3/2}_1~e^{s\varphi_1},\quad \varrho_1(t)=s^{7/2}\lambda^4\xi_2^{7/2}~e^{s\varphi_2},\quad \varrho_2(t)=s^{11/2}\lambda^{3}\xi_2^{11/2}~e^{2s\varphi_2-s\varphi_1},
\end{equation}
$$
\varrho_3(t)=s^{-1/2}\lambda~\xi_1^{-1/2}~e^{s\varphi_1},\quad \varrho_4(t)=s^{-5/2}\lambda~\xi_1^{-5/2}e^{s\varphi_1}\quad\mbox{and}\quad \varrho_5(t)=s^{3/2}\lambda^2\xi_2^{3/2}~e^{s\varphi_2}.
$$

Observe in particular that $\varrho_i\in\mathscr{C}^1([0,T])$, $\varrho_i(0)=\varrho_i(T)=0$, $i=0,\dots, 5$, and thanks to \eqref{E1}:
\begin{equation}\label{bound-rho}
|\varrho'_3|\lesssim\varrho_0~~\text{and}~~|\varrho'_4|\lesssim\varrho_3.
\end{equation}

On the other hand, thanks to Theorem \ref{Carleman-couplé}, and with $\mu \leq \lambda \leq (1+\alpha^2)\mu^2/c_1$, \eqref{OrdreVarphi12} and \eqref{E1} we get 
\begin{multline}\label{Carl-gele1}
\frac{1}{(1+\alpha^2)^2}\int_{0}^Ts^3\lambda^2\xi_1^3~~e^{2s\varphi_1}\big(\|u\|_{L^2(\Omega)}^2+\|\eta\|^2_{H^2(\mathcal{I})}+\|\partial_t\eta\|^2_{L^2(\mathcal{I})}\big)\lesssim 
H(s,\lambda, u)+B_\alpha(s,\mu,\eta) \\
\lesssim \iint_{(0,T)\times\omega}s^3\lambda^4~ \xi^3_2~~e^{2s\varphi_2}| u|^2+ \left(\alpha^{-8}+\alpha^4\right)\iint_{(0,T)\times J}s^7\lambda^8~\xi^7_2~e^{2s\varphi_2}|\eta|^2.
\end{multline}

From \eqref{Carl-gele1} we have:
\begin{multline}\label{Carl-gele2}
\int_{0}^T\varrho^2_0(t)~\big(\|u\|_{L^2(\Omega)}^2+\|\eta\|^2_{H^2(\mathcal{I})}+\|\partial_t\eta\|^2_{L^2(\mathcal{I})}\big)\\
\lesssim (1+\alpha^4)\iint_{(0,T)\times\omega}\varrho^2_5(t)| u|^2+ \left(\alpha^{-8}+\alpha^8\right)\iint_{(0,T)\times J}\varrho^2_1(t)|\eta|^2.
\end{multline}
We observe that 
$$
\frac{d}{dt}\begin{bmatrix}
    \partial_tu\\
    \partial_t\eta\\
    \partial_t^2\eta
\end{bmatrix}=\mathcal{A}\begin{bmatrix}
    \partial_tu\\
    \partial_t\eta\\
    \partial_t^2\eta
\end{bmatrix}~~~\text{in}~~(0,T).
$$
Then we can apply \eqref{Carl-gele2} to $[\partial_tu,\partial_t\eta,\partial^2_t\eta]$ and we get
\begin{multline}\label{Carl-gele3}
\int_{0}^T\varrho^2_0(t)~\big(\|\partial_tu\|_{L^2(\Omega)}^2+\|\partial_t\eta\|^2_{H^2(\mathcal{I})}+\|\partial^2_t\eta\|^2_{L^2(\mathcal{I})}\big)\\
\lesssim (1+\alpha^4)\iint_{(0,T)\times\omega}\varrho^2_5(t)|\partial_tu|^2+ \left(\alpha^{-8}+\alpha^8\right)\iint_{(0,T)\times J}\varrho^2_1(t)|\partial_t\eta|^2.
\end{multline}
 We integrate by parts the first term on the right hand side of \eqref{Carl-gele3}. We then have 
 \begin{align*}
    \iint_{(0,T)\times\omega}\varrho^2_5(t)|\partial_tu|^2&=-\iint_{(0,T)\times\omega}\partial_t(\varrho^2_5(t)\partial_tu)u
    \\&=-\iint_{(0,T)\times\omega}\partial_t(\varrho^2_5(t))~\partial_tu~u-\iint_{(0,T)\times\omega}\varrho^2_5(t)~\partial^2_tu~u\\
    &=\frac12\iint_{(0,T)\times\omega}\partial^2_t\big(s^3\lambda^4~ \xi^3_2~~e^{2s\varphi_2}\big)~|u|^2-\iint_{(0,T)\times\omega}\varrho^2_5(t)~\partial^2_tu~u.
 \end{align*}
From \eqref{E1} we deduce
 \begin{align*}
  \big|\partial^2_t\big(s^3\lambda^4~ \xi^3_2~e^{2s\varphi_2}\big)\big|&\lesssim  s^{11}\lambda^4~\xi^{11}_2~e^{2s\varphi_2}
 \end{align*}
 and since $\lambda \geq 1$ and $\varphi_2-\varphi_1\geq 0$ we get
 $$
  \iint_{(0,T)\times\omega}\varrho^2_5(t)|\partial_tu|^2\lesssim \iint_{(0,T)\times\omega}\varrho^2_2(t) |u|^2-\iint_{(0,T)\times\omega}\varrho^2_5(t)~\partial^2_tu~u.
$$
 On the other hand, 
 \begin{align*}
-\iint_{(0,T)\times\omega}\varrho^2_5(t)~\partial^2_tu~u&=-\iint_{(0,T)\times\omega}s^3\lambda^4~ \xi^3_2~~e^{2s\varphi_2}~\partial^2_tu~u\\
&=-\iint_{(0,T)\times\omega}(\underbrace{s^{11/2}\lambda^{3}~\xi^{11/2}_2e^{2s\varphi_2-s\varphi_1}}_{\varrho_2}~u)\times(\underbrace{s^{-5/2}\lambda~\xi_2^{-5/2}e^{s\varphi_1}}_{\varrho_4})\partial^2_tu.
 \end{align*}
 Therefore, for $\varepsilon>0$ we get
 \begin{multline}\label{IPP dtu}
(1+\alpha^4) \iint_{(0,T)\times\omega}\varrho^2_5(t)|\partial_tu|^2\leq C\frac{(1+\alpha^4)^2(1+\alpha)^2M_\alpha^4}{\varepsilon} \iint_{(0,T)\times\omega}\varrho_2^2|u|^2\\ +\frac{\varepsilon}{2(1+\alpha)^2M_\alpha^4}\iint_{(0,T)\times\Omega}\varrho_4^2~|\partial_t^2u|^2,
 \end{multline}
 where $M_\alpha$ is the constant in \eqref{estimate-semig}.

The next step is dedicated to the estimate of the last term in \eqref{IPP dtu}.
For that we are going to use the parabolic regularizing property of \eqref{sys-simple_eta-ode}. We first deduce from \eqref{sys-simple_eta-ode} that 
\begin{equation}\label{od-e-adjoint1}
\frac{d}{dt}\left(\varrho_3\begin{bmatrix}
    u\\
    \eta\\
    \partial_t\eta
\end{bmatrix}\right)=\mathcal{A}\left(\varrho_3\begin{bmatrix}
    u\\
    \eta\\
    \partial_t\eta
\end{bmatrix}\right)+\varrho'_3\begin{bmatrix}
  u\\
  \eta\\
  \partial_t\eta
\end{bmatrix}~~~\text{in}~~(0,T),\begin{bmatrix}
  \varrho_3u\\
  \varrho_3\eta\\
  \varrho_3\partial_t\eta
\end{bmatrix}(0)=0.
\end{equation}
Thanks to \eqref{estimate-semig}, we have
\begin{align}\label{estimate-semig1}
 \|\varrho_3[u,\eta,\partial_t\eta]\|_{H^1(0,T; ~L^2(\Omega)\times\mathcal{D}(A_1^{1/2})\times L^2(\mathcal{I}))}\leq  M_\alpha\|\varrho'_3[u,\eta,\partial_t\eta]\|_{L^2(0,T; ~L^2(\Omega)\times\mathcal{D}(A_1^{1/2})\times L^2(\mathcal{I}))}.
\end{align}
From \eqref{estimate-semig1} and \eqref{bound-rho}, we find
\begin{multline}\label{estimate-rho3-rho0}
\|\varrho_3\partial_tu\|_{L^2(0,T; L^2(\Omega))}+\|\varrho_3\partial_t\eta\|_{L^2(0,T; H^2(\mathcal{I}))}+\|\varrho_3\partial^2_t\eta\|_{L^2(0,T; L^2(\mathcal{I}))}\\ \lesssim M_\alpha\left(\|\varrho_0u\|_{L^2(0,T; L^2(\Omega))}+\|\varrho_0\eta\|_{L^2(0,T; H^2(\mathcal{I}))}+\|\varrho_0\partial_t\eta\|_{L^2(0,T; L^2(\mathcal{I}))}\right).
\end{multline}
In addition, we deduce from \eqref{sys-simple_eta-ode}, that 
\begin{equation}\label{od-e-adjoint2}
\frac{d}{dt}\begin{pmatrix}\varrho_4\frac{d}{dt}\begin{bmatrix}
    u\\
    \eta\\
    \partial_t\eta
\end{bmatrix}\end{pmatrix}=\mathcal{A}\begin{pmatrix}\varrho_4\frac{d}{dt}\begin{bmatrix}
    u\\
    \eta\\
    \partial_t\eta
\end{bmatrix}\end{pmatrix}+\varrho'_4\frac{d}{dt}\begin{bmatrix}
  u\\
  \eta\\
  \partial_t\eta
\end{bmatrix}~~~\text{in}~~(0,T),\quad \begin{pmatrix}\varrho_4\frac{d}{dt}\begin{bmatrix}
  u\\
  \eta\\
  \partial_t\eta
\end{bmatrix}\end{pmatrix}(0)=0.
\end{equation}
From from \eqref{estimate-semig}, \eqref{bound-rho} and \eqref{estimate-rho3-rho0}, we obtain 

\begin{multline}\label{estimate-rho4-rho0}
\|\varrho_4\partial^2_tu\|_{L^2(0,T; L^2(\Omega))}+\|\varrho_4\partial^2_t\eta\|_{L^2(0,T; H^2(\mathcal{I}))}+\|\varrho_4\partial^3_t\eta\|_{L^2(0,T; L^2(\mathcal{I}))}\\ \lesssim M_\alpha^2\left(\|\varrho_0u\|_{L^2(0,T; L^2(\Omega))}+\|\varrho_0\eta\|_{L^2(0,T; H^2(\mathcal{I}))}+\|\varrho_0\partial_t\eta\|_{L^2(0,T; L^2(\mathcal{I}))}\right).
\end{multline}
Since $(u,\eta,\partial_t \eta)$ satisfies \eqref{sys-simple_eta-ode} by Proposition \ref{positivité-de-A} we have that inequality \eqref{estimate-MT} is true for $[u,\eta,\partial_t \eta]$. Hence, by multiplying it by $\varrho_0(t)$, and integrating over time, we obtain
\begin{multline}
    \label{estimate-MT2}
\|\varrho_0u\|_{L^2(0,T; L^2(\Omega))}+\|\varrho_0\eta\|_{L^2(0,T; H^2(\mathcal{I}))}+\|\varrho_0\partial_t\eta\|_{L^2(0,T; L^2(\mathcal{I}))}\\
\lesssim (1+\alpha)(\|\varrho_0 \partial_t u\|_{L^2(0,T; L^2(\Omega))}+\|\varrho_0\partial_t \eta\|_{L^2(0,T; H^2(\mathcal{I}))}+\|\varrho_0\partial_t^2\eta\|_{L^2(0,T; L^2(\mathcal{I}))}).
\end{multline}
Thus by combining \eqref{Carl-gele3}, \eqref{IPP dtu}, \eqref{estimate-rho4-rho0} and \eqref{estimate-MT2}, by choosing $\varepsilon>0$ small enough we deduce
\begin{multline}\label{Carl-final}
\int_{0}^T\varrho^2_0(t)~\big(\|\partial_tu\|_{L^2(\Omega)}^2+\|\partial_t\eta\|^2_{H^2(\mathcal{I})}+\|\partial^2_t\eta\|^2_{L^2(\mathcal{I})}\big)\\
\lesssim (1+\alpha^{10})M_\alpha^4\iint_{(0,T)\times\omega}\varrho^2_2(t)|u|^2+ \left(\alpha^{-8}+\alpha^8\right)\iint_{(0,T)\times J}\varrho^2_1(t)|\partial_t\eta|^2.
\end{multline}

Finally, \eqref{estimate-MT2} with \eqref{Carl-final} and $M_\alpha\lesssim \alpha^{-1}+\alpha$ proves Proposition \ref{obs-dteta}.
\end{proof}

We are now in position to prove an observability inequality for system \eqref{adjoint-final}.
\begin{theorem}\label{observability}
Let $\omega$ and $J$ be non-empty open subsets of $\Omega$ and $\mathcal{I}$ respectively. There exists $C>0$ and $\tau>0$ such that for all $T>0$ any solution of \eqref{sys-simple_eta-ode} satisfies:
\begin{equation}\label{inegalite-obs}
\|u(T,\cdot)\|^2_{L^2(\Omega)}+\|\eta(T,\cdot)\|^2_{H^2(\mathcal{I})}+\|\partial_t\eta(T,\cdot)\|^2_{L^2(\mathcal{I})} \leq Ce^{e^{C\lambda_\alpha}\left(1+\frac{1}{T}\right)}\Big(\iint_{(0,T)\times \omega}|u|^2+\iint_{(0,T)\times J}|\partial_t\eta|^2\Big),
\end{equation}
where $\lambda_\alpha=\tau\alpha^{-2}$ if $\alpha<1$ and $\lambda_\alpha=\tau\alpha^{2/3}$ if $\alpha\geq 1$.
\end{theorem}
\begin{proof}
First, from \eqref{dteta} and \eqref{DefVarrho01} we deduce that
\begin{equation}\label{Est11-03}
\int_{T/4}^{3T/4}\varrho^2_0(t)\|[u,\eta,\partial_t\eta]\|^2_{\mathscr{H}}\lesssim (\alpha^{-4}+\alpha^{16})\iint_{(0,T)\times\omega}\varrho^2_2(t)|u|^2+ (\alpha^{-8}+\alpha^{10})\iint_{(0,T)\times J}\varrho^2_1(t)|\partial_t\eta|^2.
\end{equation}
Set $s=\widehat{s}_0(1+\alpha)(T^k+T^{k-1})$ and $\lambda_\alpha=\tau\alpha^{-2}$ if $\alpha<1$ and $\lambda_\alpha=\tau\alpha^{2/3}$ if $\alpha\geq 1$ with $\tau>0$ large enough so that it satisfies the assumptions of Theorem \ref{Carleman-couplé} (see Remark \ref{RqCarleman-couplé}). Since $\lambda_\alpha\geq \tau$ we can choose $\tau$ large enough so that
$$
\forall\alpha>0\quad 2-e^{\lambda_\alpha\Psi}-e^{-\lambda_\alpha\Psi}<-\varepsilon,
$$
for some $\varepsilon>0$ independent on $\alpha$. The above inequality guarantee that $\varphi_2-\varphi_1\leq -\varepsilon s \xi_2$, and that 
$$
\varrho_2^2\leq \lambda_\alpha^6(s\xi_2)^{11} e^{-2\varepsilon  (s \xi_2)}\leq C \lambda_\alpha^6.
$$
Similarly, we can choose $\tau$ large enough so that $1-e^{2\lambda_\alpha\Psi}<-\varepsilon$ and $\varphi_1\leq -\varepsilon s \xi_1$. Thus we deduce $\varrho_1^2\leq C \lambda_\alpha^4$.

Next, for $\tau$ large enough we have $e^{10\lambda_\alpha\Psi}-e^{8\lambda_\alpha\Psi}>\varepsilon e^{10\lambda_\alpha\Psi}$. Then with $s=\widehat{s}_0(1+\alpha)(T^k+T^{k-1})$ and using the fact that $\frac{3 T^2}{16}\leq \ell(t)\leq \frac{T^2}{4} $ on $(\frac{T}{4}, \frac{3T }{4})$ we deduce 
$$
s\varphi_1(t)\gtrsim -(1+\alpha)e^{10\lambda_\alpha\Psi}\frac{T^k+T^{k-1}}{\ell^{k/2}(t)}\gtrsim- e^{C\lambda_\alpha}\frac{T^k+T^{k-1}}{T^k}=-e^{C\lambda_\alpha}\left(1+\frac{1}{T}\right)\quad \mbox{ for all $t\in (\frac{T}{4}, \frac{3T }{4})$,}
$$
and from \eqref{E1} and $\lambda_\alpha\geq \tau$, we deduce
$$
    \varrho_0^2(t)=s^{3} \lambda_\alpha^2~\xi_1^{3}(t)~e^{2s\varphi_1(t)}\gtrsim Ce^{2s\varphi_1(t)}\gtrsim e^{-e^{C\lambda_\alpha}\left(1+\frac{1}{T}\right)}.
$$
By combining \eqref{Est11-03} with the above estimates of $\varrho_0^2$, $\varrho_1^2$ and $\varrho_2^2$ we deduce
\begin{equation}\label{Est11-03-2}
\int_{T/4}^{3T/4}\|[u,\eta,\partial_t\eta]\|^2_{\mathscr{H}}\lesssim e^{e^{C\lambda_\alpha}\left(1+\frac{1}{T}\right)}\left(\iint_{(0,T)\times\omega}|u|^2+ \iint_{(0,T)\times J}|\partial_t\eta|^2\right).
\end{equation}
Finally, since $(e^{\mathcal{A}t})_{t\geq 0}$ is strongly continuous and stable in $\mathscr{H}$ we have that
$\|[u,\eta,\partial_t\eta](T)\|_{\mathscr{H}}\lesssim \|[u,\eta,\partial_t\eta](t)\|_{\mathscr{H}}$ for all $t\in (\frac{T}{4},\frac{3T}{4}).$ 
Then
$$\frac T2 \|[u,\eta,\partial_t\eta](T)\|_{\mathscr{H}} \leq \int_{T/4}^{3T/4}\|[u,\eta,\partial_t\eta](t)\|_{\mathscr{H}},$$ 
and with \eqref{Est11-03-2} this concludes the proof.
\end{proof}


\section{Carleman estimates for the damped beam equation}\label{carleman-poutre}
In this section, we obtain a Carleman estimate for the damped beam equation.

We set 
$$f_\eta:=\partial^2_t\eta+\partial^4_{x_1}\eta-\alpha\partial_t\partial^2_{x_1}\eta\quad \text{and}\quad z=e^{s\varphi_0}\eta.$$



\noindent Direct computations yield to
\begin{equation}\label{decomp}
    \mathcal{M}_1z+\mathcal{M}_2z=e^{s\varphi}f_\eta-\mathcal{N}z,
\end{equation}
with
\begin{equation}\label{DefM1M2}
\mathcal{M}_1 z=\mathcal{M}_{11}z+\mathcal{M}_{12}z\quad \mbox{ and }\quad \mathcal{M}_2 z=\mathcal{M}_{21}z+\mathcal{M}_{22}z,
\end{equation}where
\begin{align}\label{DefM11}
\mathcal{M}_{11}z&=s^4(\partial_{x_1}\varphi_0)^4z+6s^2(\partial_{x_1}\varphi_0)^2\partial^2_{x_1}z+\partial^4_{x_1}z+2s\alpha(\partial_{x_1}\varphi_0)\partial_t\partial_{x_1}z+\partial^2_tz,\\
\mathcal{M}_{12}z&=s^2(\partial_t\varphi_0)^2z+s\alpha(\partial_t\varphi_0)\partial^2_{x_1}z+s^3\alpha(\partial_{x_1}\varphi_0)^2(\partial_t\varphi_0)z,\label{DefM12}
\end{align}
and
\begin{align}\notag
\mathcal{M}_{21}z&=-4s^3(\partial_{x_1}\varphi_0)^3\partial_{x_1}z-4s(\partial_{x_1}\varphi_0)\partial^3_{x_1}z-\alpha\partial_t\partial^2_{x_1}z\\
&-s^2\alpha(\partial_{x_1}\varphi_0)^2\partial_tz-6(1+\beta)s^3(\partial_{x_1}\varphi_0)^2(\partial^2_{x_1}\varphi_0)z, \label{DefM21}\\ 
\mathcal{M}_{22}z=&-2s(\partial_t\varphi_0)\partial_tz-2s^2\alpha(\partial_t\varphi_0)(\partial_{x_1}\varphi)\partial_{x_1}z,\label{DefM22}
\end{align}
and
\begin{align*}
\mathcal{N}z=&\alpha(s(\partial^2_{x_1}\varphi_0)\partial_tz+2s(\partial_t\partial_{x_1}\varphi_0)\partial_{x_1}z+s(\partial_t\partial^2_{x_1}\varphi_0)z-s^2(\partial^2_{x_1}\varphi_0)(\partial_t\varphi_0)z-2s^2(\partial_t\partial_{x_1}\varphi_0)(\partial_{x_1}\varphi_0)z)\\
&-s(\partial^2_t\varphi_0)z-6s(\partial^2_{x_1}\varphi_0)\partial^2_{x_1}z-4s(\partial^3_{x_1}\varphi_0)\partial_{x_1}z+12s^2(\partial^2_{x_1}\varphi_0)(\partial_{x_1}\varphi_0)\partial_{x_1}z\\
&-s(\partial^4_{x_1}\varphi_0)z+4s^2(\partial^3_{x_1}\varphi_0)(\partial_{x_1}\varphi_0)z+3s^2(\partial^2_{x_1}\varphi_0)^2z+6\beta s^3(\partial^2_{x_1}\varphi_0)(\partial_{x_1}\varphi_0)^2z.
\end{align*}
In the above setting $\beta$ is a free real-valued parameter that will play a role only in the Proof of Theorem \ref{poutre3}. 

For the following we set
\begin{equation}\label{DefR0}
\mathcal{R}_0=2\iint_{(0,T)\times\mathcal{I}}|\mathcal{N}z|^2
\end{equation}
and from \eqref{decomp} we deduce
\begin{align}\label{ineqalite1}\notag
&\|\mathcal{M}_1z\|^2_{L^2((0,T)\times\mathcal{I})}+\|\mathcal{M}_2z\|^2_{L^2((0,T)\times\mathcal{I})}+2\iint_{(0,T)\times\mathcal{I}} \mathcal{M}_1z\cdot \mathcal{M}_2z\\&=\iint_{(0,T)\times\mathcal{I}}|e^{s\varphi}f_\eta-\mathcal{N}z|^2
\leq 2\iint_{(0,T)\times\mathcal{I}}e^{2s\varphi_0}|f_\eta|^2+\mathcal{R}_0.
\end{align}
From by using  \eqref{poutre1},\eqref{n2-2},\eqref{n3-2},\eqref{n1-2}, and with $s_0\geq 1$, we deduce that $\mathcal{R}_0$ defined in \eqref{DefR0} satisfies
\begin{equation}\label{EstR0}
\left|\mathcal{R}_0\right|\lesssim \left(\frac{1}{s_0}+\frac{\alpha^2}{s_0^2}\right)\iint_{(0,T)\times\mathcal{I}}\left(s^7\mu^8\xi^7_0|z|^2+s^5\mu^6\xi_0^5|\partial_{x_1}z|^2+s^3\mu^4\xi_0^3|\partial^2_{x_1}z|^2+(1+\alpha^2)s^3\mu^4\xi_0^3|\partial_tz|^2\right).
\end{equation}

Let us denote by $I_{ij}$ the cross product of the $i$'th term of $\mathcal{M}_1$ with the $j$'th term of $\mathcal{M}_2$. We have:
\begin{align*}
&\iint_{(0,T)\times\mathcal{I}}\mathcal{M}_1z\cdot \mathcal{M}_2z=\iint_{(0,T)\times\mathcal{I}}\mathcal{M}_{11}z\cdot \mathcal{M}_{21}z+\mathcal{R}_1,
\end{align*}
with 
\begin{equation}\label{DefCPM11M21}
\iint_{(0,T)\times\mathcal{I}}\mathcal{M}_{11}z\cdot \mathcal{M}_{21}z=\sum\limits_{i=1,j=1}^{i=5,j=5}I_{ij},
\end{equation}
and
\begin{equation}\label{DefR1}
\mathcal{R}_1=\sum\limits_{(i,j)\in \mathbb{I}}I_{ij},\quad\mathbb{I}=\{1,\dots 8\}\times \{1,\dots 7\}\backslash \{1,\dots, 5\}\times \{1,\dots,5\}. 
\end{equation}
We prove in Section \ref{SectionEstR1} that 
\begin{multline}\label{EstR1}
\left|\mathcal{R}_1\right|\lesssim \left(\frac{1+\alpha}{s_0}+\frac{\alpha^2}{s_0^2}\right)\iint_{(0,T)\times\mathcal{I}}\left(s^7\mu^8\xi^7_0|z|^2+s^5\mu^6\xi_0^5|\partial_{x_1}z|^2+s^3\mu^4\xi_0^3|\partial^2_{x_1}z|^2\right.\\\left.+(1+\alpha^2)s^3\mu^4\xi_0^3|\partial_tz|^2+s\mu^2 \xi_0|\partial_t\partial_{x_1}z|^2\right).
\end{multline}
From \eqref{ineqalite1} we deduce
\begin{align}\label{ineqalite12}
\|\mathcal{M}_1z\|^2_{L^2((0,T)\times\mathcal{I})}+\|\mathcal{M}_2z\|^2_{L^2((0,T)\times\mathcal{I})}+2\iint_{(0,T)\times\mathcal{I}} \mathcal{M}_{11}z\cdot \mathcal{M}_{21}z
\leq 2\iint_{(0,T)\times\mathcal{I}}e^{2s\varphi_0}|f_\eta|^2+\mathcal{R}_0+\mathcal{R}_1.
\end{align}
\subsection{Computation of the cross product}
The present section is dedicated to the computation of $\iint_{(0,T)\times\mathcal{I}}\mathcal{M}_{11}z\cdot \mathcal{M}_{21}z$, that is the computations of each $I_{ij}$ defined in \eqref{DefCPM11M21}.  
This leads to tedious but straightforward computations. 
\paragraph{-- Computation of $I_{11}$ --}
\begin{align*}
  I_{11}&=-4s^7\iint_{(0,T)\times\mathcal{I}}(\partial_{x_1}\varphi_0)^7z\partial_{x_1}z
  =14s^7\iint_{(0,T)\times\mathcal{I}} (\partial_{x_1}\varphi_0)^6(\partial^2_{x_1}\varphi_0)|z|^2.
\end{align*}
\paragraph{-- Computation of $I_{12}$ --}
\begin{align*}
    I_{12}&=-4s^5 \iint_{(0,T)\times\mathcal{I}}(\partial_{x_1}\varphi_0)^5z\partial^3_{x_1}z\\
    &=20s^5\iint_{(0,T)\times\mathcal{I}}(\partial_{x_1}\varphi_0)^4(\partial_{x_1}^2\varphi_0)z\partial^2_{x_1}z+2s^5\iint_{(0,T)\times\mathcal{I}}(\partial_{x_1}\varphi_0)^5\partial_{x_1}|\partial_{x_1}z|^2\\
&=10s^5\iint_{(0,T)\times\mathcal{I}}\partial_{x_1}^2[(\partial_{x_1}\varphi_0)^4(\partial_{x_1}^2\varphi_0)]|z|^2-30s^5\iint_{(0,T)\times\mathcal{I}}(\partial_{x_1}\varphi_0)^4(\partial_{x_1}^2\varphi_0)|\partial_{x_1}z|^2\\
&=-30s^5\iint_{(0,T)\times\mathcal{I}}(\partial_{x_1}\varphi_0)^4(\partial_{x_1}^2\varphi_0)|\partial_{x_1}z|^2+\mathcal{R}_{12},
\end{align*}
where
\begin{align*}
\mathcal{R}_{12}=10s^5\iint_{(0,T)\times\mathcal{I}}\partial_{x_1}^2[(\partial_{x_1}\varphi_0)^4(\partial_{x_1}^2\varphi_0)]|z|^2.
\end{align*}
Moreover, from \eqref{poutre1} and \eqref{n1-2} we deduce
\begin{equation}\label{EstR12}
|\mathcal{R}_{12}|\lesssim \frac{1}{s_0^2} s^7 \mu^8\iint_{(0,T)\times\mathcal{I}}\xi_0^7 |z|^2.
\end{equation}
\paragraph{-- Computation of $I_{13}$ --}
\begin{align*}
I_{13}&=-s^4\alpha \iint_{(0,T)\times\mathcal{I}}(\partial_{x_1}\varphi_0)^4\partial_t\partial_{x_1}^2zz\\
&=4s^4\alpha \iint_{(0,T)\times\mathcal{I}}(\partial_{x_1}\varphi_0)^3(\partial^2_{x_1}\varphi_0)\partial_t\partial_{x_1}zz-2s^4\alpha\iint_{(0,T)\times\mathcal{I}}(\partial_t\partial_{x_1}\varphi_0)(\partial_{x_1}\varphi_0)^3|\partial_{x_1}z|^2\\
&=-4s^4\alpha\iint_{(0,T)\times\mathcal{I}}(\partial_{x_1}\varphi_0)^3(\partial^2_{x_1}\varphi_0)\partial_{x_1}z\partial_tz+2s^4\alpha \iint_{(0,T)\times\mathcal{I}}\partial_{x_1}\partial_t[(\partial_{x_1}\varphi_0)^3(\partial^2_{x_1}\varphi_0)]|z|^2\\
&\quad -2s^4\alpha\iint_{(0,T)\times\mathcal{I}}(\partial_t\partial_{x_1}\varphi_0)(\partial_{x_1}\varphi_0)^3|\partial_{x_1}z|^2\\
&=-4s^4\alpha\iint_{(0,T)\times\mathcal{I}}(\partial_{x_1}\varphi_0)^3(\partial^2_{x_1}\varphi_0)\partial_{x_1}z\partial_tz+\mathcal{R}_{13},
\end{align*}
with
\begin{align*}
\mathcal{R}_{13}&=2s^4\alpha \iint_{(0,T)\times\mathcal{I}}\partial_{x_1}\partial_t[(\partial_{x_1}\varphi_0)^3(\partial^2_{x_1}\varphi_0)]|z|^2-2s^4\alpha\iint_{(0,T)\times\mathcal{I}}(\partial_t\partial_{x_1}\varphi_0)(\partial_{x_1}\varphi_0)^3|\partial_{x_1}z|^2.
\end{align*}
Moreover, from \eqref{poutre1}, \eqref{n2-2} and \eqref{n1-2} we deduce
\begin{align}
|\mathcal{R}_{13}|\lesssim \frac{\alpha}{s_0^3}s^7\mu^6\iint_{(0,T)\times\mathcal{I}}\xi_0^7|z|^2+\frac{\alpha}{s_0} s^5\mu^4\iint_{(0,T)\times\mathcal{I}}\xi_0^5|\partial_{x_1}z|^2.\label{EstR13}
\end{align}
\paragraph{-- Computation of $I_{14}$ --}
\begin{align*}
I_{14}&=-s^6\alpha\iint_{(0,T)\times\mathcal{I}}(\partial_{x_1}\varphi_0)^6\partial_tzz
=3s^6\alpha\iint_{(0,T)\times\mathcal{I}}(\partial_{x_1}\varphi_0)^5(\partial_t\partial_{x_1}\varphi_0)|z|^2=\mathcal{R}_{14}.
\end{align*}
Moreover, from \eqref{poutre1} and \eqref{n2-2} we deduce
\begin{align}\label{EstI14}
|\mathcal{R}_{14}|&\lesssim \frac{\alpha}{s_0}s^7\mu^6 \iint_{(0,T)\times\mathcal{I}}\xi_0^7|z|^2.
\end{align}
\paragraph{-- Computation of $I_{15}$ --}
$$I_{15}=-6(1+\beta)s^7\iint_{(0,T)\times\mathcal{I}}(\partial_{x_1}\varphi_0)^6(\partial^2_{x_1}\varphi_0)|z|^2.$$
\paragraph{-- Computation of $I_{21}$ --}
\begin{align*}
I_{21}&=-24s^5\iint_{(0,T)\times\mathcal{I}}(\partial_{x_1}\varphi_0)^5\partial_{x_1}z\partial^2_{x_1}z
=60s^5\iint_{(0,T)\times\mathcal{I}}(\partial_{x_1}\varphi_0)^4(\partial^2_{x_1}\varphi_0)|\partial_{x_1}z|^2.
\end{align*}
\paragraph{-- Computation of $I_{22}$ --}
\begin{align*}
I_{22}&=-24s^3\iint_{(0,T)\times\mathcal{I}}(\partial_{x_1}\varphi_0)^3\partial^2_{x_1}z\partial^3_{x_1}z
=36s^3\iint_{(0,T)\times\mathcal{I}}(\partial_{x_1}\varphi_0)^2(\partial^2_{x_1}\varphi_0)|\partial^2_{x_1}z|^2.
\end{align*}
\paragraph{-- Computation of $I_{23}$ --}
\begin{align*}
I_{23}&=-6s^2\alpha\iint_{(0,T)\times\mathcal{I}}(\partial_{x_1}\varphi_0)^2\partial_t\partial^2_{x_1}z\partial^2_{x_1}z
=6s^2\alpha\iint_{(0,T)\times\mathcal{I}}(\partial_t\partial_{x_1}\varphi_0)(\partial_{x_1}\varphi_0)|\partial^2_{x_1}z|^2=\mathcal{R}_{23}.
\end{align*}
Moreover, from \eqref{poutre1} and \eqref{n2-2} we deduce
\begin{align}
|\mathcal{R}_{23}|&\lesssim \frac{\alpha}{s_0} s^3\mu^2\iint_{(0,T)\times\mathcal{I}}\xi_0^3|\partial^2_{x_1}z|^2.\label{EstI23}
\end{align}
\paragraph{-- Computation of $I_{24}$ --}
\begin{align*}
I_{24}&=-6s^4\alpha\iint_{(0,T)\times\mathcal{I}}(\partial_{x_1}\varphi_0)^4\partial_tz\partial^2_{x_1}z\\
&=24s^4\alpha\iint_{(0,T)\times\mathcal{I}}(\partial_{x_1}\varphi_0)^3(\partial^2_{x_1}\varphi_0)\partial_tz\partial_{x_1}z+\mathcal{R}_{24},
\end{align*}
with $$\mathcal{R}_{24}=-12s^4\alpha\iint_{(0,T)\times\mathcal{I}}(\partial_{x_1}\varphi_0)^3(\partial_t\partial_{x_1}\varphi_0)|\partial_{x_1}z|^2.$$
Moreover, from \eqref{poutre1} and \eqref{n2-2} we deduce
\begin{align}
|\mathcal{R}_{24}|\lesssim \frac{\alpha}{s_0}s^5\mu^4\iint_{(0,T)\times\mathcal{I}} \xi_0^5|\partial_{x_1}z|^2.\label{EstR24}
\end{align}
\paragraph{-- Computation of $I_{25}$ --}
\begin{align*}
I_{25}&=-36(1+\beta)s^5\iint_{(0,T)\times\mathcal{I}}(\partial_{x_1}\varphi_0)^4(\partial^2_{x_1}\varphi_0)\partial^2_{x_1}zz\\
&=36(1+\beta)s^5\iint_{(0,T)\times\mathcal{I}}(\partial_{x_1}\varphi_0)^4(\partial^2_{x_1}\varphi_0)|\partial_{x_1}z|^2+\mathcal{R}_{25},
\end{align*}
with
\begin{align*}
\mathcal{R}_{25}&=-18(1+\beta) s^5\iint_{(0,T)\times\mathcal{I}}\partial_{x_1}^2[(\partial_{x_1}\varphi_0)^4(\partial^2_{x_1}\varphi_0)]|z|^2.
\end{align*}
Moreover, from \eqref{poutre1} and \eqref{n1-2} we deduce
\begin{align}\label{EstR25}
|\mathcal{R}_{25}|&\lesssim \frac{1}{s_0^2}s^5\mu^8\iint_{(0,T)\times\mathcal{I}}\xi_0^7 |z|^2.
\end{align}
\paragraph{-- Computation of $I_{31}$ --}
\begin{align*}
I_{31}&=-4s^3\iint_{(0,T)\times\mathcal{I}}(\partial_{x_1}\varphi_0)^3\partial_{x_1}z\partial^4_{x_1}z\\
&=12s^3\iint_{(0,T)\times\mathcal{I}}(\partial_{x_1}\varphi_0)^2(\partial^2_{x_1}\varphi_0)\partial_{x_1}z\partial^3_{x_1}z+4s^3\iint_{(0,T)\times\mathcal{I}}(\partial_{x_1}\varphi_0)^3\partial^2_{x_1}z\partial^3_{x_1}z\\
&=6s^3\iint_{(0,T)\times\mathcal{I}}\partial_{x_1}^2[(\partial_{x_1}\varphi_0)^2(\partial^2_{x_1}\varphi_0)]|\partial_{x_1}z|^2-18s^3\iint_{(0,T)\times\mathcal{I}}(\partial_{x_1}\varphi_0)^2(\partial^2_{x_1}\varphi_0)|\partial_{x_1}^2z|^2\\
&=-18s^3\iint_{(0,T)\times\mathcal{I}}(\partial_{x_1}\varphi_0)^2(\partial^2_{x_1}\varphi_0)|\partial^2_{x_1}z|^2+\mathcal{R}_{31},
\end{align*}
with
\begin{align*}
\mathcal{R}_{31}&=6s^3\iint_{(0,T)\times\mathcal{I}}\partial_{x_1}^2[(\partial_{x_1}\varphi_0)^2(\partial^2_{x_1}\varphi_0)]|\partial_{x_1}z|^2.
\end{align*}
Moreover, from \eqref{poutre1} and \eqref{n1-2} we deduce
\begin{align}\label{EstR31}
|\mathcal{R}_{31}|&\lesssim \frac{1}{s_0^2}s^5\mu^6 \iint_{(0,T)\times\mathcal{I}}\xi_0^5|\partial_{x_1}z|^2.
\end{align}
\paragraph{-- Computation of $I_{32}$ --}
$$I_{32}=-4s\iint_{(0,T)\times\mathcal{I}}(\partial_{x_1}\varphi_0)\partial^3_{x_1}z\partial^4_{x_1}z=2s\iint_{(0,T)\times\mathcal{I}}(\partial^2_{x_1}\varphi_0)|\partial^3_{x_1}z|^2.$$
\paragraph{-- Computation of $I_{33}$ --}
$$I_{33}=-\alpha\iint_{(0,T)\times\mathcal{I}}\partial^4_{x_1}z \partial_t\partial^2_{x_1}z=0.$$
\paragraph{-- Computation of $I_{34}$ --}
\begin{align*}
I_{34}&=-s^2\alpha\iint_{(0,T)\times\mathcal{I}}(\partial_{x_1}\varphi_0)^2\partial^4_{x_1}z\partial_tz\\
&=2s^2\alpha\iint_{(0,T)\times\mathcal{I}}(\partial_{x_1}\varphi_0)(\partial_{x_1}^2\varphi_0)\partial^3_{x_1}z\partial_tz+s^2\alpha\iint_{(0,T)\times\mathcal{I}}(\partial_{x_1}\varphi_0)^2\partial^3_{x_1}z\partial_t\partial_{x_1}z.\\
\end{align*}
A new integration by parts leads to 
\begin{align*}
I_{34}&=-4s^2\alpha~\iint_{(0,T)\times\mathcal{I}}(\partial_{x_1}\varphi_0)(\partial_{x_1}^2\varphi_0)\partial^2_{x_1}z\partial_t\partial_{x_1}z-s^2\alpha~\iint_{(0,T)\times\mathcal{I}}(\partial_{x_1}\varphi_0)^2\partial^2_{x_1}z\partial_t\partial_{x_1}^2z\\
&\quad -2s^2\alpha\iint_{(0,T)\times\mathcal{I}}\partial_{x_1}[(\partial_{x_1}\varphi_0)(\partial_{x_1}^2\varphi_0)]\partial^2_{x_1}z\partial_tz\\
&=-4s^2\alpha\iint_{(0,T)\times\mathcal{I}}(\partial_{x_1}\varphi_0)(\partial_{x_1}^2\varphi_0)\partial^2_{x_1}z\partial_t\partial_{x_1}z+s^2\alpha\iint_{(0,T)\times\mathcal{I}}(\partial_{x_1}\varphi_0)(\partial_t\partial_{x_1}\varphi_0)|\partial^2_{x_1}z|^2\\
&\quad -2s^2\alpha\iint_{(0,T)\times\mathcal{I}}\partial_{x_1}[(\partial_{x_1}\varphi_0)(\partial_{x_1}^2\varphi_0)]\partial^2_{x_1}z\partial_tz\\
&=-4s^2\alpha\iint_{(0,T)\times\mathcal{I}}(\partial_{x_1}\varphi_0)(\partial_{x_1}^2\varphi_0)\partial^2_{x_1}z\partial_t\partial_{x_1}z+\mathcal{R}_{34}
\end{align*}
with
$$
\mathcal{R}_{34}=s^2\alpha\iint_{(0,T)\times\mathcal{I}}(\partial_{x_1}\varphi_0)(\partial_t\partial_{x_1}\varphi_0)|\partial^2_{x_1}z|^2-2s^2\alpha\iint_{(0,T)\times\mathcal{I}}\partial_{x_1}[(\partial_{x_1}\varphi_0)(\partial_{x_1}^2\varphi_0)]\partial^2_{x_1}z\partial_tz.
$$
Moreover, from \eqref{poutre1}, \eqref{n2-2} and \eqref{n1-2} we deduce
\begin{align}\notag
|\mathcal{R}_{34}|&\lesssim \frac{\alpha}{s_0}s^3\mu^2\iint_{(0,T)\times\mathcal{I}}\xi_0^{3}|\partial^2_{x_1}z|^2+\frac{\alpha}{s_0}s^3\mu^4\iint_{(0,T)\times\mathcal{I}}\xi_0^3|\partial^2_{x_1}z||\partial_tz|\\
&\lesssim \frac{\alpha}{s_0}s^3\mu^4\iint_{(0,T)\times\mathcal{I}}\xi_0^{3}|\partial^2_{x_1}z|^2+\frac{\alpha}{s_0}s^3\mu^4\iint_{(0,T)\times\mathcal{I}}\xi_0^2|\partial_tz|^2. \label{EstR34}
\end{align}
\paragraph{-- Computation of $I_{35}$ --}
\begin{align*}
 I_{35}=& -6(1+\beta)s^3\iint_{(0,T)\times\mathcal{I}}(\partial_{x_1}\varphi_0)^2(\partial^2_{x_1}\varphi_0)z\partial^4_{x_1}z\\
 =& 6(1+\beta)s^3\iint_{(0,T)\times\mathcal{I}}\partial_{x_1}[(\partial_{x_1}\varphi_0)^2(\partial^2_{x_1}\varphi_0)]z\partial^3_{x_1}z+6(1+\beta)s^3\iint_{(0,T)\times\mathcal{I}}(\partial_{x_1}\varphi_0)^2(\partial^2_{x_1}\varphi_0)\partial_{x_1}z\partial^3_{x_1}z\\
 =& -6(1+\beta)s^3\iint_{(0,T)\times\mathcal{I}}\partial_{x_1}^2[(\partial_{x_1}\varphi_0)^2(\partial^2_{x_1}\varphi_0)]z\partial^2_{x_1}z\\
 &-12(1+\beta)s^3\iint_{(0,T)\times\mathcal{I}}\partial_{x_1}[(\partial_{x_1}\varphi_0)^2(\partial^2_{x_1}\varphi_0)]\partial_{x_1}z\partial^2_{x_1}z\\
 &-6(1+\beta)s^3\iint_{(0,T)\times\mathcal{I}}(\partial_{x_1}\varphi_0)^2(\partial^2_{x_1}\varphi_0)|\partial^2_{x_1}z|^2\\
 =& -3(1+\beta)s^3\iint_{(0,T)\times\mathcal{I}}\partial_{x_1}^4[(\partial_{x_1}\varphi_0)^2(\partial^2_{x_1}\varphi_0)]|z|^2+6(1+\beta)s^3\iint_{(0,T)\times\mathcal{I}}\partial_{x_1}^2[(\partial_{x_1}\varphi_0)^2(\partial^2_{x_1}\varphi_0)]|\partial_{x_1}z|^2\\
 & -6(1+\beta)s^3\iint_{(0,T)\times\mathcal{I}}(\partial_{x_1}\varphi_0)^2(\partial^2_{x_1}\varphi_0)|\partial^2_{x_1}z|^2\\
 =&-6(1+\beta)s^3\iint_{(0,T)\times\mathcal{I}}(\partial_{x_1}\varphi_0)^2(\partial^2_{x_1}\varphi_0)|\partial^2_{x_1}z|^2+\mathcal{R}_{35},
\end{align*}
where 
\begin{align*}
   \mathcal{R}_{35}&= -3(1+\beta)s^3\iint_{(0,T)\times\mathcal{I}}\partial_{x_1}^4[(\partial_{x_1}\varphi_0)^2(\partial^2_{x_1}\varphi_0)]|z|^2+6(1+\beta)s^3\iint_{(0,T)\times\mathcal{I}}\partial_{x_1}^2[(\partial_{x_1}\varphi_0)^2(\partial^2_{x_1}\varphi_0)]|\partial_{x_1}z|^2.
\end{align*}
Moreover, from \eqref{poutre1} and \eqref{n1-2} we deduce
\begin{align}
   |\mathcal{R}_{35}|&\lesssim \frac{1}{s_0^4}s^7\mu^8\iint_{(0,T)\times\mathcal{I}}\xi_0^7|z|^2+\frac{1}{s_0^2}s^5\mu^6\iint_{(0,T)\times\mathcal{I}}\xi_0^5|\partial_{x_1}z|^2.\label{EstR35}
\end{align}
\paragraph{-- Computation of $I_{41}$ --}
\begin{align*}
I_{41}&=-8s^4\alpha\iint_{(0,T)\times\mathcal{I}}(\partial_{x_1}\varphi_0)^4\partial_{x_1}z\partial_t\partial_{x_1}z=16s^4\alpha\iint_{(0,T)\times\mathcal{I}}\partial_t\partial_{x_1}\varphi_0(\partial_{x_1}\varphi_0)^3|\partial_{x_1}z|^2=\mathcal{R}_{41}.
\end{align*}
Moreover, from \eqref{poutre1} and \eqref{n1-2} we deduce
\begin{align}
|\mathcal{R}_{41}|&\lesssim \frac{\alpha}{s_0}s^5\mu^4\iint_{(0,T)\times\mathcal{I}}\xi_0^5|\partial_{x_1}z|^2.\label{EstI41}
\end{align}
\paragraph{-- Computation of $I_{42}$ --}
\begin{align*}
I_{42}&=-8s^2\alpha\iint_{(0,T)\times\mathcal{I}}(\partial_{x_1}\varphi_0)^2\partial^3_{x_1}z\partial_t\partial_{x_1}z=16s^2\alpha\iint_{(0,T)\times\mathcal{I}}(\partial_{x_1}\varphi_0)(\partial^2_{x_1}\varphi_0)\partial^2_{x_1}z\partial_t\partial_{x_1}z+\mathcal{R}_{42},
\end{align*}
with 
$$\mathcal{R}_{42}=-8s^2\alpha\iint_{(0,T)\times\mathcal{I}}(\partial_{x_1}\varphi_0)(\partial_t\partial_{x_1}\varphi_0)|\partial^2_{x_1}z|^2.$$
Moreover, from \eqref{poutre1} and \eqref{n1-2} we deduce
\begin{align}
|\mathcal{R}_{42}|&\lesssim \frac{\alpha}{s_0}s^3\mu^2\iint_{(0,T)\times\mathcal{I}}\xi_0^3|\partial^2_{x_1}z|^2.\label{EstR42}
\end{align}
\paragraph{-- Computation of $I_{43}$ --}
\begin{align*}
I_{43}&=-2s\alpha^2\iint_{(0,T)\times\mathcal{I}}(\partial_{x_1}\varphi_0)\partial_t\partial_{x_1}z\partial_t\partial^2_{x_1}z=s\alpha^2\iint_{(0,T)\times\mathcal{I}}(\partial^2_{x_1}\varphi_0)|\partial_t\partial_{x_1}z|^2.
\end{align*}
\paragraph{-- Computation of $I_{44}$ --}
\begin{align*}
I_{44}&=-2s^3\alpha^2\iint_{(0,T)\times\mathcal{I}}(\partial_{x_1}\varphi_0)^3\partial_t\partial_{x_1}z\partial_tz=3s^3\alpha^2\iint_{(0,T)\times\mathcal{I}}(\partial_{x_1}\varphi_0)^2(\partial^2_{x_1}\varphi_0)|\partial_tz|^2.
\end{align*}
\paragraph{-- Computation of $I_{45}$ --}
\begin{align*}
 I_{45}&=-12(1+\beta)s^4\alpha\iint_{(0,T)\times\mathcal{I}}(\partial_{x_1}\varphi_0)^3(\partial^2_{x_1}\varphi_0)\partial_t\partial_{x_1}zz\\
 &=12(1+\beta)s^4\alpha\iint_{(0,T)\times\mathcal{I}}(\partial_{x_1}\varphi_0)^3(\partial^2_{x_1}\varphi_0)\partial_{x_1}z\partial_tz+\mathcal{R}_{45},
\end{align*}
with
$$\mathcal{R}_{45}=-6(1+\beta)s^4\alpha\iint_{(0,T)\times\mathcal{I}}\partial_{x_1}\partial_t[(\partial_{x_1}\varphi_0)^3(\partial^2_{x_1}\varphi_0)]|z|^2.$$
Moreover, from \eqref{poutre1}, \eqref{n2-2} and \eqref{n1-2} we deduce
\begin{align}
 |\mathcal{R}_{45}|&\lesssim \frac{\alpha}{s_0^3}s^7\mu^6\iint_{(0,T)\times\mathcal{I}}\xi_0^7|z|^2.\label{EstR45}
\end{align}
\paragraph{-- Computation of $I_{51}$ --}
\begin{align*}
I_{51}&=-4s^3\iint_{(0,T)\times\mathcal{I}}(\partial_{x_1}\varphi_0)^3\partial_{x_1}z\partial^2_tz=-6s^3\iint_{(0,T)\times\mathcal{I}}(\partial_{x_1}\varphi_0)^2(\partial^2_{x_1}\varphi_0)|\partial_tz|^2+\mathcal{R}_{51},
\end{align*}
with
$$\mathcal{R}_{51}=12s^3\iint_{(0,T)\times\mathcal{I}}(\partial_{x_1}\varphi_0)^2(\partial_t\partial_{x_1}\varphi_0)\partial_{x_1}z\partial_tz.$$
Moreover, from \eqref{poutre1} and \eqref{n1-2} we deduce
\begin{align}
|\mathcal{R}_{51}|&\lesssim \frac{1}{s_0}s^4\mu^3 \iint_{(0,T)\times\mathcal{I}}\xi_0^4|\partial_{x_1}z||\partial_tz|\lesssim \frac{1}{s_0}s^5\mu^3 \iint_{(0,T)\times\mathcal{I}}\xi_0^5|\partial_{x_1}z|^2+\frac{1}{s_0}s^3\mu^3 \iint_{(0,T)\times\mathcal{I}}\xi_0^3|\partial_tz|^2.\label{EstR51}
\end{align}
\paragraph{-- Computation of $I_{52}$ --}
\begin{align*}
I_{52}&=-4s\iint_{(0,T)\times\mathcal{I}}(\partial_{x_1}\varphi_0)\partial^3_{x_1}z\partial^2_tz\\
&=4s\iint_{(0,T)\times\mathcal{I}}(\partial_t\partial_{x_1}\varphi_0)\partial^3_{x_1}z\partial_tz+4s\iint_{(0,T)\times\mathcal{I}}(\partial_{x_1}\varphi_0)\partial_t\partial^3_{x_1}z\partial_tz\\
&=4s\iint_{(0,T)\times\mathcal{I}}(\partial_t\partial_{x_1}\varphi_0)\partial^3_{x_1}z\partial_tz-4s\iint_{(0,T)\times\mathcal{I}}(\partial^2_{x_1}\varphi_0)\partial_t\partial^2_{x_1}z\partial_tz-2s\iint_{(0,T)\times\mathcal{I}}(\partial_{x_1}\varphi_0)\partial_{x_1}|\partial_t\partial_{x_1}z|^2\\
&=4s\iint_{(0,T)\times\mathcal{I}}(\partial_t\partial_{x_1}\varphi_0)\partial^3_{x_1}z\partial_tz+4s\iint_{(0,T)\times\mathcal{I}}(\partial^3_{x_1}\varphi_0)\partial_t\partial_{x_1}z\partial_tz+6s\iint_{(0,T)\times\mathcal{I}}(\partial^2_{x_1}\varphi_0)|\partial_t\partial_{x_1}z|^2\\
&=6s\iint_{(0,T)\times\mathcal{I}}(\partial^2_{x_1}\varphi_0)|\partial_t\partial_{x_1}z|^2+\mathcal{R}_{52},
\end{align*}
with
\begin{align*}
\mathcal{R}_{52}
&=4s\iint_{(0,T)\times\mathcal{I}}(\partial_t\partial_{x_1}\varphi_0)\partial^3_{x_1}z\partial_tz-2s\iint_{(0,T)\times\mathcal{I}}(\partial^4_{x_1}\varphi_0)|\partial_tz|^2.
\end{align*}
Moreover, from \eqref{poutre1}, \eqref{n2-2} and \eqref{n1-2} we deduce
\begin{align}
|\mathcal{R}_{52}|
&\lesssim \frac{1}{s_0}s^2\mu \iint_{(0,T)\times\mathcal{I}}\xi_0^2|\partial^3_{x_1}z||\partial_tz|+\frac{1}{s_0^2}s^3\mu^4\iint_{(0,T)\times\mathcal{I}}\xi_0^3|\partial_tz|^2\notag \\
&\lesssim \frac{1}{s_0}s\mu \iint_{(0,T)\times\mathcal{I}}\xi_0|\partial^3_{x_1}z|^2~+\frac{1}{s_0}s^3\mu \iint_{(0,T)\times\mathcal{I}}\xi_0^3|\partial_tz|^2+\frac{1}{s_0^2}s^3\mu^4\iint_{(0,T)\times\mathcal{I}}\xi_0^3|\partial_tz|^2.\label{EstR52}
\end{align}
\paragraph{-- Computation of $I_{53}$ --}
$$I_{53}=-\alpha\iint_{(0,T)\times\mathcal{I}}\partial_t\partial^2_{x_1}z\partial^2_tz=\alpha\iint_{(0,T)\times\mathcal{I}}\partial_t\partial_{x_1}z\partial^2_t\partial_{x_1}z=\frac{\alpha}{2}\iint_{(0,T)\times\mathcal{I}}\partial_t|\partial_t\partial_{x_1}z|^2=0.$$
\paragraph{-- Computation of $I_{54}$ --}
\begin{align*}
I_{54}&=-s^2\alpha~\iint_{(0,T)\times\mathcal{I}}(\partial_{x_1}\varphi_0)^2\partial_tz\partial^2_tz=s^2\alpha\iint_{(0,T)\times\mathcal{I}}(\partial_{x_1}\varphi_0)(\partial_t\partial_{x_1}\varphi_0)|\partial_tz|^2=\mathcal{R}_{54}.
\end{align*}
Moreover, from \eqref{poutre1} and \eqref{n1-2} we deduce
\begin{align}
|\mathcal{R}_{54}| \lesssim \frac{\alpha}{s_0}s^3\mu^2\iint_{(0,T)\times\mathcal{I}}\xi_0^3|\partial_tz|^2.\label{EstR54}
\end{align}
\paragraph{-- Computation of $I_{55}$ --}
\begin{align*}
I_{55}&=-6(1+\beta)s^3\iint_{(0,T)\times\mathcal{I}}(\partial_{x_1}\varphi_0)^2(\partial^2_{x_1}\varphi_0)\partial^2_tzz=6(1+\beta)s^3\iint_{(0,T)\times\mathcal{I}}(\partial_{x_1}\varphi_0)^2(\partial^2_{x_1}\varphi_0)|\partial_tz|^2+\mathcal{R}_{55},
\end{align*}
with
\begin{align*}
\mathcal{R}_{55}&=-3(1+\beta)s^3\iint_{(0,T)\times\mathcal{I}}\partial_t^2[(\partial_{x_1}\varphi_0)^2(\partial^2_{x_1}\varphi_0)]|z|^2.
\end{align*} 
Moreover, from \eqref{poutre1}, \eqref{n2-2}, \eqref{n3-2} and \eqref{n1-2} we deduce
\begin{align}
|\mathcal{R}_{55}|\lesssim \frac{1}{s_0^4}s^7\mu^4 \iint_{(0,T)\times\mathcal{I}}\xi_0^7|z|^2\label{EstR55}.
\end{align} 
Finally, by combining the above expressions we obtain
\begin{align} \label{produit-croise}
&\iint_{(0,T)\times\mathcal{I}}\mathcal{M}_{11} z\cdot \mathcal{M}_{21} z=\sum\limits_{i,j=1}^{i=5,j=5}I_{ij}\\ 
 \notag
&=(8-6\beta)s^7\iint_{(0,T)\times\mathcal{I}} (\partial_{x_1}\varphi_0)^6(\partial^2_{x_1}\varphi_0)|z|^2+(66+36\beta)s^5\iint_{(0,T)\times\mathcal{I}}(\partial_{x_1}\varphi_0)^4(\partial_{x_1}^2\varphi_0)|\partial_{x_1}z|^2\\ \notag
&+(12-6\beta)s^3\iint_{(0,T)\times\mathcal{I}}(\partial_{x_1}\varphi_0)^2(\partial^2_{x_1}\varphi_0)|\partial^2_{x_1}z|^2+(3\alpha^2+6\beta)s^3\iint_{(0,T)\times\mathcal{I}}(\partial_{x_1}\varphi_0)^2(\partial^2_{x_1}\varphi_0)|\partial_tz|^2\\ \notag
&+2s\iint_{(0,T)\times\mathcal{I}}(\partial^2_{x_1}\varphi_0)|\partial^3_{x_1}z|^2+(\alpha^2+6)s\iint_{(0,T)\times\mathcal{I}}(\partial^2_{x_1}\varphi_0)|\partial_t\partial_{x_1}z|^2\\ \notag
&+(32+12\beta)s^4\alpha\iint_{(0,T)\times\mathcal{I}}(\partial_{x_1}\varphi_0)^3(\partial^2_{x_1}\varphi_0)\partial_tz\partial_{x_1}z+12s^2\alpha\iint_{(0,T)\times\mathcal{I}}(\partial_{x_1}\varphi_0)(\partial^2_{x_1}\varphi_0)\partial^2_{x_1}z\partial_t\partial_{x_1}z\\
&+ \mathcal{R}_2,\notag
\end{align}
where $\mathcal{R}_2$ is the sum of the $\mathcal{R}_{ij}$ that are estimated in \eqref{EstR12}-\eqref{EstR55} and that satisfies
\begin{align}\label{EstR3}
\left|\mathcal{R}_2\right|&\lesssim \left(\frac{1+\alpha}{s_0}\right)\iint_{(0,T)\times\mathcal{I}}\left(s^7\mu^8\xi^7_0|z|^2+s^5\mu^6\xi_0^5|\partial_{x_1}z|^2+s^3\mu^4\xi_0^3|\partial^2_{x_1}z|^2+s^3\mu^4\xi_0^3|\partial_tz|^2+s\mu^2\xi_0|\partial_{x_1}^3z|^2\right).
\end{align}
For later use, we rewrite \eqref{produit-croise} as 
\begin{multline} \label{produit-croise-I_l}
\iint_{(0,T)\times\mathcal{I}}\mathcal{M}_{11} z\cdot \mathcal{M}_{21} z=(8-6\beta)I_1+(66+36\beta) I_2+(12-6\beta) I_3\\
+(3\alpha^2+6\beta) I_4+2 I_5+(\alpha^2+6)I_6+(32+12\beta)\alpha I_7+ 12 \alpha I_8+\mathcal{R}_2,
\end{multline}
where $I_\ell$, $\ell = 1, \ldots ,8$, denote the integral terms.

\subsection{Estimate of the cross product from below for arbitrary \texorpdfstring{$\alpha\geq 0$}{alpha} -- Proof of Theorem \ref{poutre2}} \label{section_proof_CarlEstpoutre_alphanonneg}
In what follows, the value of the free parameter $\beta$ plays no role, so  for simplicity we fix $\beta=1$. Then \eqref{produit-croise-I_l} gives
\begin{equation}\label{Eq11062025}
\iint_{(0,T)\times\mathcal{I}}\mathcal{M}_{11} z\cdot \mathcal{M}_{21} z= 2I_1+ 102I_2+ 6I_3+ (3\alpha^2+6) I_4+2 I_5+ (\alpha^2+6) I_6+44 \alpha I_7+12 \alpha I_8+\mathcal{R}_2.
\end{equation}
Using the fact that $z=e^{s\varphi_0}\eta$ we observe that
\begin{equation}\label{exp}
\partial_tz=s\partial_t\varphi_0~z+e^{s\varphi_0}\partial_t \eta.
\end{equation}
Hence, we have
\begin{multline*}
44 \alpha I_7=44 \alpha s^4\iint_{(0,T)\times\mathcal{I}}(\partial_{x_1}\varphi_0)^3(\partial^2_{x_1}\varphi_0)\partial_tz\partial_{x_1}z \\
=44 \alpha s^4\iint_{(0,T)\times\mathcal{I}}(\partial_{x_1}\varphi_0)^3(\partial^2_{x_1}\varphi_0)e^{s\varphi_0}\partial_t \eta \partial_{x_1}z+44\alpha s^5\iint_{(0,T)\times\mathcal{I}}(\partial_{x_1}\varphi_0)^3(\partial^2_{x_1}\varphi_0)(\partial_t\varphi_0)z\partial_{x_1}z\\
=44 \alpha s^4\iint_{(0,T)\times\mathcal{I}}(\partial_{x_1}\varphi_0)^3(\partial^2_{x_1}\varphi_0)(e^{s\varphi_0}\partial_t \eta)\partial_{x_1}z+\mathcal{R}_3.
\end{multline*}
with 
$$
\mathcal{R}_3=22 \alpha s^5\iint_{(0,T)\times\mathcal{I}}\partial_{x_1}[(\partial_{x_1}\varphi_0)^3(\partial^2_{x_1}\varphi_0)(\partial_t\varphi_0)]|z|^2.
$$
Moreover, we have
\begin{multline*}
12 \alpha I_8=12\alpha s^2\iint_{(0,T)\times\mathcal{I}}(\partial_{x_1}\varphi_0)(\partial^2_{x_1}\varphi_0)\partial^2_{x_1}z\partial_t\partial_{x_1}z\\
=- 12\alpha s^2\iint_{(0,T)\times\mathcal{I}}(\partial_{x_1}\varphi_0)(\partial^2_{x_1}\varphi_0)\partial^3_{x_1}z\partial_tz- 12\alpha s^2\iint_{(0,T)\times\mathcal{I}}\partial_{x_1}[(\partial_{x_1}\varphi_0)(\partial^2_{x_1}\varphi_0)]\partial^2_{x_1}z\partial_tz\\
=-12 \alpha s^2\iint_{(0,T)\times\mathcal{I}}(\partial_{x_1}\varphi_0)(\partial^2_{x_1}\varphi_0)\partial^3_{x_1}z(e^{s\varphi_0}\partial_t \eta)+\mathcal{R}_4,
\end{multline*}
with 
$$
\mathcal{R}_4=-12 \alpha s^3\iint_{(0,T)\times\mathcal{I}}(\partial_{x_1}\varphi_0)(\partial^2_{x_1}\varphi_0)(\partial_t\varphi_0)\partial^3_{x_1}z z- 12 \alpha s^2\iint_{(0,T)\times\mathcal{I}}\partial_{x_1}[(\partial_{x_1}\varphi_0)(\partial^2_{x_1}\varphi_0)]\partial^2_{x_1}z\partial_tz.
$$
Note that from \eqref{poutre1}, \eqref{n2-2} and \eqref{n1-2} we deduce
\begin{multline}\label{EstR5}
|\mathcal{R}_3|+|\mathcal{R}_4|\lesssim  \frac{\alpha}{s_0}\iint_{(0,T)\times\mathcal{I}}(s^7\mu^6 \xi_0^7 |z|^2+s^4\mu^3 \xi_0^4 |\partial^3_{x_1}z|~|z|+s^3\mu^4 \xi_0^3 |\partial^2_{x_1}z|~|\partial_tz|)\\
\lesssim \frac{\alpha}{s_0}\iint_{(0,T)\times\mathcal{I}} s^7\mu^6 \xi_0^7|z|^2+s \xi_0 |\partial^3_{x_1}z|^2+s^3\mu^4 \xi_0^3 (|\partial^2_{x_1}z|^2+|\partial_tz|^2)
\end{multline}

Moreover, from \eqref{poutre2-0} and Cauchy Schwarz inequality we deduce
\begin{multline*}
\left|12\alpha s^2\iint_{(0,T)\times\mathcal{I}\backslash J_0}(\partial_{x_1}\varphi_0)(\partial^2_{x_1}\varphi_0)(e^{s\varphi_0}\partial_t \eta)\partial^3_{x_1}z\right|+\left|44 \alpha s^4\iint_{(0,T)\times\mathcal{I}\backslash J_0}(\partial_{x_1}\varphi_0)^3(\partial^2_{x_1}\varphi_0)(e^{s\varphi_0}\partial_t \eta)\partial_{x_1}z\right|\\
\leq 6\varepsilon s\iint_{(0,T)\times\mathcal{I}\backslash J_0}(\partial^2_{x_1}\varphi_0)|\partial^3_{x_1}z|^2+22\varepsilon s^5\iint_{(0,T)\times\mathcal{I}\backslash J_0}(\partial^2_{x_1}\varphi_0)(\partial_{x_1}\varphi_0)^4|\partial_{x_1}z|^2\\+\frac{28}{\varepsilon}\alpha^2 s^3\iint_{(0,T)\times\mathcal{I}} (\partial_{x_1}\varphi_0)^2|\partial^2_{x_1}\varphi_0|e^{2s\varphi_0}|\partial_t \eta|^2\\
=6\varepsilon I_5+22\varepsilon I_2-6\varepsilon s\iint_{(0,T)\times J_0}(\partial^2_{x_1}\varphi_0)|\partial^3_{x_1}z|^2-22\varepsilon s^5\iint_{(0,T)\times J_0}(\partial^2_{x_1}\varphi_0)(\partial_{x_1}\varphi_0)^4|\partial_{x_1}z|^2\\
+\frac{28}{\varepsilon}\alpha^2 s^3\iint_{(0,T)\times\mathcal{I}}(\partial_{x_1}\varphi_0)^2|\partial^2_{x_1}\varphi_0| e^{2s\varphi_0}|\partial_t \eta|^2
\end{multline*}
and then, with \eqref{Eq11062025}, by choosing $\varepsilon$ small enough we get
\begin{multline*}
\iint_{(0,T)\times\mathcal{I}}\mathcal{M}_{11} z\cdot \mathcal{M}_{21} z+s\iint_{(0,T)\times J_0}|\partial^2_{x_1}\varphi_0||\partial^3_{x_1}z|^2+ s^5\iint_{(0,T)\times J_0}|\partial^2_{x_1}\varphi_0|(\partial_{x_1}\varphi_0)^4|\partial_{x_1}z|^2\\
+\alpha^2 s^3\iint_{(0,T)\times\mathcal{I}}|\partial^2_{x_1}\varphi_0| (\partial_{x_1}\varphi_0)^2~e^{2s\varphi_0}|(\partial_t \eta)|^2-\mathcal{R}_2-\mathcal{R}_3-\mathcal{R}_4\\
\gtrsim (I_1+ I_2+ I_3+ (1+\alpha^2)I_4+I_5+ (1+\alpha^2)I_6).
\end{multline*}
Thus, with \eqref{ineqalite12}, using \eqref{poutre2-0}, \eqref{poutre1} and \eqref{EstR0}, \eqref{EstR1}, \eqref{EstR3}, \eqref{EstR5} we find that for $\lambda_0\geq 1$ and $\widehat{s}_0=\frac{s_0}{1+\alpha}\geq 1 $ large enough, for all $\lambda\geq\lambda_0$ and all $s\geq \widehat{s}_0(1+\alpha)(T^k+T^{k-1})=s_0(T^k+T^{k-1})$, 

\begin{multline}\label{midle-estimate}
\|\mathcal{M}_1z\|^2_{L^2((0,T)\times\mathcal{I})}+\|\mathcal{M}_2z\|^2_{L^2((0,T)\times\mathcal{I})}\\
+ \iint_{(0,T)\times\mathcal{I}}\Big(s^7\mu^8\xi^7_0~|z|^2+s^5\mu^6\xi_0^5|\partial_{x_1}z|^2+s^3\mu^4\xi_0^3(|\partial^2_{x_1}z|^2+(1+\alpha^2)|\partial_tz|^2)  +s\mu^2\xi_0(|\partial^3_{x_1}z|^2+(1+\alpha^2)|\partial_t\partial_{x_1}z|^2)\Big)\\
\lesssim \Big(\iint_{(0,T)\times\mathcal{I}}e^{2s\varphi_0}|f_\eta|^2+\alpha^2s^3\mu^4\iint_{(0,T)\times\mathcal{I}}\xi_0^3~e^{2s\varphi_0}|\partial_t \eta|^2\\
+\iint_{(0,T)\times J_0}\Big(s^7\mu^8\xi^7_0~|z|^2+s^5\mu^6\xi_0^5|\partial_{x_1}z|^2+s^3\mu^4\xi_0^3(|\partial^2_{x_1}z|^2+(1+\alpha^2)|\partial_tz|^2)+s\mu^2\xi_0(|\partial^3_{x_1}z|^2+(1+\alpha^2)|\partial_t\partial_{x_1}z|^2)\Big)
\end{multline}
From inequality \eqref{midle-estimate}, we directly obtain the Theorem \ref{poutre2}.
\subsection{The case \texorpdfstring{$\alpha > 0$ } -- Proof of Theorem \ref{poutre}}\label{Sect3}
From now on, we assume $\alpha$ to be positive. We will now prove our main Carleman estimates, that is Theorem \ref{poutre}.
\subsubsection{Estimate of the local terms in \texorpdfstring{$\partial_{x_1}z$}{der1} and \texorpdfstring{$\partial_{x_1}^2z$}{der2}}\label{SectLocalTerm1}

Let $J_0\Subset J_1\Subset J$ and $\chi_1\in C_0^\infty(J_1)$ such that $\chi_1\equiv 1$ in $J_0$ and let $\varepsilon_1>0$. By integrating by parts and by using \eqref{poutre1}, \eqref{n1-2} we get 
\begin{multline*}
    s^3\mu^4 \iint_{(0,T)\times J_0}\xi_0^3|\partial_{x_1}^2z|^2\leq s^3\mu^4 \iint_{(0,T)\times J_1}\chi_1\xi_0^3|\partial_{x_1}^2 z|^2= - \iint_{(0,T)\times J_1}\partial_{x_1}(s^3\mu^4\chi_1\xi_0^3\partial_{x_1}^2 z) \partial_{x_1}z\\
    \leq C_1 \iint_{(0,T)\times J_1}(s^4\mu^5\xi_0^4|\partial_{x_1}^2 z||\partial_{x_1} z|+s^3\mu^4\xi_0^3|\partial_{x_1}^3 z||\partial_{x_1} z|) \\
    \leq C_1 \varepsilon_1 \iint_{(0,T)\times \mathcal{I}}(s\mu^2\xi_0|\partial_{x_1}^3 z|^2
    +s^3\mu^4\xi_0^3|\partial_{x_1}^2 z|^2)+\frac{C_1}{\varepsilon_1}\iint_{(0,T)\times J_1}s^5\mu^6\xi_0^5|\partial_{x_1} z|^2.
\end{multline*}
where $C_1>0$ is independent on $\varepsilon_1$. Then by choosing $\varepsilon_1>0$ small enough we can remove the local term $s^3\mu^4 \iint_{(0,T)\times J_0}\xi_0^3|\partial_{x_1}^2z|^2$ and replace $J_0$ by $J_1$ in the equality \eqref{midle-estimate}, up to a modification of the constant $C$ there. 

Next, let $J_1\Subset J_2 \Subset J$ and let $\chi_2\in C_0^\infty(J_2)$ such that $\chi_2\equiv 1$ in $J_1$ and let $\varepsilon_2>0$. By integrating by parts and by using \eqref{poutre1}, \eqref{n1-2} we get 
\begin{multline*}
    s^5\mu^6 \iint_{(0,T)\times J_1}\xi_0^5|\partial_{x_1}z|^2\leq s^5\mu^6 \iint_{(0,T)\times J}\chi_2\xi_0^5|\partial_{x_1}z|^2= - \iint_{(0,T)\times J_2}\partial_{x_1}(s^5\mu^6\chi_2\xi_0^5\partial_{x_1}z) z\\
    \leq C_2\iint_{(0,T)\times J_2}(s^5\mu^6\xi_0^5|\partial_{x_1}^2 z||z|+s^6\mu^7\xi_0^{6}|\partial_{x_1}z||z|) \\
    \leq C_2 \varepsilon_2 \iint_{(0,T)\times\mathcal{I}}(s^3\mu^4\xi_0^3|\partial_{x_1}^2 z|^2
    +s^5\mu^6\xi_0^5|\partial_{x_1}z|^2)+\frac{C_2}{\varepsilon_2} \iint_{(0,T)\times J_2}s^7\mu^8\xi_0^7|z|^2,
\end{multline*}
where $C_2>0$ is independent on $\varepsilon_2$. Then, similarly as above, by choosing $\varepsilon_2>0$ small enough we can remove the local term $s^5\mu^6 \iint_{(0,T)\times J_1}\xi_0^5|\partial_{x_1}z|^2$ and finally obtain
\begin{multline}\label{midle-estimate2}
\|\mathcal{M}_1z\|^2_{L^2((0,T)\times\mathcal{I})}+\|\mathcal{M}_2z\|^2_{L^2((0,T)\times\mathcal{I})}\\
+ \iint_{(0,T)\times\mathcal{I}}\Big(s^7\mu^8\xi^7_0~|z|^2+s^5\mu^6\xi_0^5|\partial_{x_1}z|^2+s^3\mu^4\xi_0^3(|\partial^2_{x_1}z|^2+(1+\alpha^2)|\partial_tz|^2)  +s\mu^2\xi_0(|\partial^3_{x_1}z|^2+(1+\alpha^2)|\partial_t\partial_{x_1}z|^2)\Big)\\
\leq C\Big(\iint_{(0,T)\times\mathcal{I}}e^{2s\varphi_0}|f_\eta|^2+\alpha^2s^3\mu^4\iint_{(0,T)\times\mathcal{I}}\xi_0^3~e^{2s\varphi_0}|\partial_t \eta|^2\\
+\iint_{(0,T)\times J_2}\Big(s^7\mu^8\xi^7_0~|z|^2+(1+\alpha^2)s^3\mu^4\xi_0^3 |\partial_tz|^2+s\mu^2\xi_0(|\partial^3_{x_1}z|^2+(1+\alpha^2)|\partial_t\partial_{x_1}z|^2\Big)
\end{multline}
for some $C>0$ that can be chosen independent on $\alpha$.

\subsubsection{A Carleman estimate with high order terms in  \texorpdfstring{$\partial^2_tz$}{dt2}, \texorpdfstring{$\partial_t\partial^2_{x_1}z$}{dtdx2} and \texorpdfstring{$\partial^4_{x_1}z$}{dx4}}
Now our goal is to estimate $\partial^2_tz,$ $\alpha\partial_t\partial^2_{x_1}z$ and $\partial^4_{x_1}z.$ 

First, from \eqref{DefM1M2} and \eqref{DefM21} we have
\begin{align*}
\alpha\partial_t\partial^2_{x_1}z&=-4s^3(\partial_{x_1}\varphi_0)^3\partial_{x_1}z-4s\partial_{x_1}\varphi_0\partial^3_{x_1}z-s^2\alpha(\partial_{x_1}\varphi_0)^2\partial_tz-12s^3(\partial_{x_1}\varphi_0)^2\partial^2_{x_1}\varphi_0z+\mathcal{M}_{22}-\mathcal{M}_2z,
\end{align*}
and with \eqref{poutre1} and \eqref{n1-2} we deduce
\begin{multline}\label{Eq14-0325-1}
\frac{\alpha^2}{s}\iint_{(0,T)\times\mathcal{I}}\frac1\xi_0|\partial_t\partial^2_{x_1}z|^2\lesssim \iint_{(0,T)\times\mathcal{I}}(s^7\mu^8 \xi_0^7|z|^2+s^5\mu^6 \xi_0^5|\partial_{x_1}z|^2+\alpha^2 s^3\mu^4 \xi_0^3|\partial_tz|^2+s\mu^2 \xi_0~|\partial^3_{x_1}z|^2)\\+\iint_{(0,T)\times\mathcal{I}}\frac{1}{s\xi_0}(|\mathcal{M}_{22}|^2+|\mathcal{M}_{2}|^2).
\end{multline}
Moreover, from \eqref{DefM22} and from \eqref{poutre1} and \eqref{n2-2} we deduce
\begin{equation}\label{EstM22}
\frac1s\iint_{(0,T)\times\mathcal{I}}\frac1\xi_0~|\mathcal{M}_{22}|^2\\
\lesssim \left(\frac{1+\alpha^2}{s_0^2}\right)\iint_{(0,T)\times\mathcal{I}}\left(s^5\mu^6\xi_0^5|\partial_{x_1}z|^2+s^3\mu^4\xi_0^3|\partial_tz|^2\right)
\end{equation}
and from \eqref{n1-2} we deduce
\begin{equation}\label{EstM2}
\iint_{(0,T)\times\mathcal{I}}\frac{1}{s\xi_0}~|\mathcal{M}_{2}|^2\lesssim \frac{1}{s_0}\|\mathcal{M}_2z\|^2_{L^2((0,T)\times\mathcal{I})}.
\end{equation}
Then by combining \eqref{midle-estimate2} with \eqref{Eq14-0325-1}, \eqref{EstM22} and \eqref{EstM2} we obtain for $\widehat{s}_0=\frac{s_0}{1+\alpha}\geq 1$,
\begin{multline}\label{Eqs15-03-2025-0}
\alpha^2\iint_{(0,T)\times\mathcal{I}}\frac{1}{s\xi_0}|\partial_t\partial^2_{x_1}z|^2\\
+ \iint_{(0,T)\times\mathcal{I}} (s^7\mu^8\xi_0^7~|z|^2+s^5\mu^6\xi_0^6~|\partial_{x_1}z|^2+s^3\mu^4\xi_0^3~(|\partial^2_{x_1}z|^2+(1+\alpha^2)|\partial_{t}z|^2)+s\mu^2\xi_0~((1+\alpha^2)|\partial_t\partial_{x_1}z|^2+|\partial_{x_1}^3z|^2)\\
\lesssim \Big(\iint_{(0,T)\times\mathcal{I}}e^{2s\varphi_0}|f_\eta|^2+\alpha^2s^3\mu^4\iint_{(0,T)\times\mathcal{I}}\xi_0^3~e^{2s\varphi_0}|\partial_t \eta|^2\\
+\iint_{(0,T)\times J_2}\Big(s^7\mu^8\xi^7_0~|z|^2+s^3\mu^4\xi_0^3 (1+\alpha^2)|\partial_tz|^2+s\mu^2\xi_0(|\partial^3_{x_1}z|^2+(1+\alpha^2)|\partial_t\partial_{x_1}z|^2\Big).
\end{multline}

Next, by observing that
$$
\frac1s\iint_{(0,T)\times\mathcal{I}}\frac1\xi_0|\partial^4_{x_1}z+\partial^2_tz|^2=\frac1s\iint_{(0,T)\times\mathcal{I}}\frac1\xi_0(|\partial^4_{x_1}z|^2+|\partial^2_tz|^2+2\partial^4_{x_1}z\partial^2_tz)
$$
and with 
\begin{multline*}
\frac2s\iint_{(0,T)\times\mathcal{I}}\frac1\xi_0\partial^4_{x_1}z\partial^2_tz=-\frac2s\iint_{(0,T)\times\mathcal{I}}\frac1\xi_0\partial^4_{x_1}\partial_tz\partial_tz-\frac2s\iint_{(0,T)\times\mathcal{I}}\partial_t(\xi_0^{-1})\partial^4_{x_1}z\partial_tz \\
=\frac2s\iint_{(0,T)\times\mathcal{I}}\frac1\xi_0\partial^3_{x_1}\partial_tz\partial_ {x_1}\partial_tz+\frac2s\iint_{(0,T)\times\mathcal{I}}\partial_{x_1}(\xi_0^{-1})\partial^3_{x_1}\partial_tz\partial_tz\\
+\frac2s\iint_{(0,T)\times\mathcal{I}}\partial_t(\xi_0^{-1})\partial^3_{x_1}z \partial_{x_1} \partial_tz+\frac2s\iint_{(0,T)\times\mathcal{I}}\partial_{x_1}\partial_t(\xi_0^{-1})\partial^3_{x_1}z\partial_tz \\
=-\frac2s\iint_{(0,T)\times\mathcal{I}}\frac1\xi_0|\partial^2_{x_1}\partial_tz|^2-\frac4s\iint_{(0,T)\times\mathcal{I}}\partial_{x_1}(\xi_0^{-1})\partial^2_{x_1}\partial_tz\partial_ {x_1}\partial_tz\\
-\frac2s\iint_{(0,T)\times\mathcal{I}}\partial_{x_1}^2(\xi_0^{-1})\partial_{x_1}^2\partial_tz\partial_tz
-\frac2s\iint_{(0,T)\times\mathcal{I}}\partial_t(\xi_0^{-1})\partial^2_{x_1}\partial_tz \partial_{x_1}^2z\\
-\frac4s\iint_{(0,T)\times\mathcal{I}}\partial_{x_1}\partial_t(\xi_0^{-1})\partial^2_{x_1}z \partial_{x_1}\partial_tz 
-\frac2s\iint_{(0,T)\times\mathcal{I}}\partial^2_{x_1}\partial_t(\xi_0^{-1})\partial^2_{x_1}z \partial_tz,
\end{multline*}
we deduce
\begin{equation}\label{EqR6}
\frac{\alpha^2}{s}\iint_{(0,T)\times\mathcal{I}}\frac1\xi_0(|\partial^4_{x_1}z|^2+|\partial^2_tz|^2)=\frac{\alpha^2}{s}\iint_{(0,T)\times\mathcal{I}}\frac1\xi_0|\partial^4_{x_1}z+\partial^2_tz|^2+\frac{2\alpha^2}{s}\iint_{(0,T)\times\mathcal{I}}\frac1\xi_0|\partial^2_{x_1}\partial_tz|^2+\mathcal{R}_6,
\end{equation}
where
\begin{multline*}
    \mathcal{R}_6=-\frac{4\alpha^2}{s}\iint_{(0,T)\times\mathcal{I}}\partial_{x_1}^2(\xi_0^{-1})|\partial_{x_1}\partial_tz|^2
+\frac{\alpha^2}{s}\iint_{(0,T)\times\mathcal{I}}\partial_{x_1}^4(\xi_0^{-1})|\partial_tz|^2-\frac{\alpha^2}{s}\iint_{(0,T)\times\mathcal{I}}\partial_t^2(\xi_0^{-1})|\partial^2_{x_1}z|^2\\
+\frac{4\alpha^2}{s}\iint_{(0,T)\times\mathcal{I}}\partial_{x_1}\partial_t(\xi_0^{-1})\partial_{x_1}\partial_tz \partial_{x_1}^2z + \frac{2\alpha^2}{s}\iint_{(0,T)\times\mathcal{I}}\partial^2_{x_1}\partial_t(\xi_0^{-1})\partial^2_{x_1}z \partial_tz.
\end{multline*}

Moreover, from \eqref{poutre1}, \eqref{n2-2}, \eqref{n3-2} and \eqref{n1-2} we deduce
\begin{equation}\label{EstR6}
\left|\mathcal{R}_6\right|\lesssim \frac{\alpha^2}{s_0^2}\iint_{(0,T)\times\mathcal{I}}\big(s^3\mu^4\xi_0^3|\partial^2_{x_1}z|^2                                                                                         +s^3\mu^4\xi_0^3|\partial_tz|^2+s\mu^2\xi_0|\partial_{x_1}\partial_t z|^2\big).
\end{equation}

Thus, from \eqref{DefM1M2} and \eqref{DefM11} we deduce
\begin{align*}
\alpha(\partial^2_tz+\partial^4_{x_1}z)&=-\alpha s^4(\partial_{x_1}\varphi_0)^4z-6\alpha s^2(\partial_{x_1}\varphi_0)^2\partial^2_{x_1}z-2\alpha^2 s(\partial_{x_1}\varphi_0)\partial_t\partial_{x_1}z+\alpha \mathcal{M}_1z-\alpha \mathcal{M}_{12}z,
\end{align*}
and with \eqref{poutre1} we deduce
\begin{multline}\label{Eq14-0325-2}
\frac{\alpha^2}{s}\iint_{(0,T)\times\mathcal{I}}\frac1\xi_0|\partial^4_{x_1}z+\partial^2_tz|^2\lesssim \alpha^2\left(\iint_{(0,T)\times\mathcal{I}} s^7\mu^8\xi_0^7~|z|^2+ s^3\mu^4\xi_0^3~|\partial^2_{x_1}z|^2+\alpha^2 s\mu^2\xi_0~|\partial_t\partial_{x_1}z|^2\right)\\
+ \alpha^2\iint_{(0,T)\times\mathcal{I}}\frac{1}{s\xi_0}~(|\mathcal{M}_{12}|^2+|\mathcal{M}_{1}|^2).
\end{multline}

Moreover, from \eqref{DefM12} and from \eqref{poutre1} and \eqref{n2-2} we deduce
\begin{equation}\label{EstM12}
\frac{\alpha^2}{s}\iint_{(0,T)\times\mathcal{I}}\frac1\xi_0~|\mathcal{M}_{12}|^2\\
\lesssim \left(\frac{\alpha^2}{s_0^4}+\frac{\alpha^4}{s_0^2}\right)\iint_{(0,T)\times\mathcal{I}}\left(s^7\mu^8\xi^7_0|z|^2+s^3\mu^4\xi_0^3|\partial^2_{x_1}z|^2\right).
\end{equation}
Moreover, from \eqref{n1-2} we deduce
\begin{equation}\label{EstM1Final}
\alpha^2\iint_{(0,T)\times\mathcal{I}}\frac{1}{s\xi_0}~|\mathcal{M}_{1}|^2\lesssim \frac{\alpha^2}{s_0}\|\mathcal{M}_1z\|^2_{L^2((0,T)\times\mathcal{I})}.
\end{equation}

Then by combining  \eqref{EqR6}, \eqref{EstR6}, \eqref{Eq14-0325-2},  \eqref{EstM12}, \eqref{EstM1Final}, \eqref{Eq14-0325-1}, \eqref{EstM22} and \eqref{EstM2} with $\widehat{s}_0=\frac{s_0}{1+\alpha}\geq 1$ we deduce
\begin{multline*}
\iint_{(0,T)\times\mathcal{I}}\frac{\alpha^2}{s\xi_0}(|\partial^4_{x_1}z|^2+|\partial^2_tz|^2)\lesssim (1+\alpha^2)\iint_{(0,T)\times\mathcal{I}}(|\mathcal{M}_{1}|^2+|\mathcal{M}_{2}|^2)\\
+ (1+\alpha^2)\iint_{(0,T)\times\mathcal{I}} (s^7\mu^8\xi_0^7~|z|^2+s^5\mu^6\xi_0^6~|\partial_{x_1}z|^2+s^3\mu^4\xi_0^3~(|\partial^2_{x_1}z|^2+(1+\alpha^2)|\partial^2_{t}z|^2)+s\mu^2\xi_0~(\alpha^2|\partial_t\partial_{x_1}z|^2+|\partial_{x_1}^3z|^2).
\end{multline*}

Then with \eqref{midle-estimate2} and \eqref{Eqs15-03-2025-0} it follows that
\begin{multline}\label{Eqs15-03-2025}
\frac{\alpha^2}{(1+\alpha^2)}\iint_{(0,T)\times\mathcal{I}}\frac{1}{s\xi_0}(|\partial^4_{x_1}z|^2+|\partial^2_tz|^2)+\alpha^2\iint_{(0,T)\times\mathcal{I}}\frac{1}{s\xi_0}|\partial_t\partial^2_{x_1}z|^2\\
+ \iint_{(0,T)\times\mathcal{I}} (s^7\mu^8\xi_0^7~|z|^2+s^5\mu^6\xi_0^6~|\partial_{x_1}z|^2+s^3\mu^4\xi_0^3~(|\partial^2_{x_1}z|^2+(1+\alpha^2)|\partial_{t}z|^2)+s\mu^2\xi_0~((1+\alpha^2)|\partial_t\partial_{x_1}z|^2+|\partial_{x_1}^3z|^2)\\
\lesssim \Big(\iint_{(0,T)\times\mathcal{I}}e^{2s\varphi_0}|f_\eta|^2+\alpha^2s^3\mu^4\iint_{(0,T)\times\mathcal{I}}\xi_0^3~e^{2s\varphi_0}|\partial_t \eta|^2\\
+\iint_{(0,T)\times J_2}\Big(s^7\mu^8\xi^7_0~|z|^2+s^3\mu^4\xi_0^3 (1+\alpha^2)|\partial_tz|^2+s\mu^2\xi_0(|\partial^3_{x_1}z|^2+(1+\alpha^2)|\partial_t\partial_{x_1}z|^2\Big).
\end{multline}
\subsubsection{Estimate of the local terms in \texorpdfstring{$\partial_{x_1}^3z$}{dx3} and \texorpdfstring{$\partial_t z$}{dtt1} and \texorpdfstring{$\partial_t \partial_{x_1}z$}{dertx}}
Let $J_2\Subset J_3\Subset J$ and $\chi_3\in C_0^\infty(J_3)$ such that $\chi_3\equiv 1$ in $J_2$ and let $\varepsilon_3>0$. By integrating by parts and by using \eqref{poutre1}, \eqref{n1-2} we get
\begin{multline*}
    s\mu^2 \iint_{(0,T)\times J_2}\xi_0(|\partial^3_{x_1}z|^2+(1+\alpha^2)|\partial_t\partial_{x_1}z|^2)\leq s\mu^2 \iint_{(0,T)\times J_3}\chi_3\xi_0(|\partial^3_{x_1}z|^2+(1+\alpha^2)|\partial_t\partial_{x_1}z|^2)\\
    = - \iint_{(0,T)\times J_3}\partial_{x_1}(s\mu^2\chi_3\xi_0\partial_{x_1}^3 z) \partial_{x_1}^2 z- (1+\alpha^2)\iint_{(0,T)\times J_3}\partial_{x_1}(s\mu^2\chi_3\xi_0\partial_t \partial_{x_1} z) \partial_{t} z\\
    \leq C_3 \iint_{(0,T)\times J_3}(s\mu^2\xi_0(|\partial_{x_1}^4 z||\partial_{x_1}^2 z|+(1+\alpha^2)|\partial_{t}\partial_{x_1}^2 z||\partial_{t} z|)+s\mu^3\xi_0(|\partial_{x_1}^3 z||\partial_{x_1}^2 z|+(1+\alpha^2)|\partial_{t}\partial_{x_1} z||\partial_{t} z|)) \\
    \leq C_3 \varepsilon_3 \iint_{(0,T)\times \mathcal{I}}\frac{1}{s\xi_0}\left(\frac{\alpha^2}{1+\alpha^2}|\partial_{x_1}^4 z|^2+\alpha^2|\partial_{t}\partial_{x_1}^2 z|^2\right)+s\mu^2\xi_0(|\partial_{x_1}^3 z|^2+(1+\alpha^2)|\partial_{t}\partial_{x_1} z|^2))\\
    +\left(\frac{1+\alpha^2}{\alpha^2}\right)\frac{C_3}{\varepsilon_3}\iint_{(0,T)\times J_3}s^3\mu^4\xi_0^3(|\partial_{x_1}^2 z|^2+(1+\alpha^2)|\partial_{t} z|^2),
\end{multline*}
where $C_3>0$ is independent on $\varepsilon_3$. Then by choosing $\varepsilon_3>0$ small enough, in \eqref{Eqs15-03-2025} we can replace the local term 
$$
\iint_{(0,T)\times J_2}\Big(s\mu^2\xi_0(|\partial^3_{x_1}z|^2+(1+\alpha^2)|\partial_t\partial_{x_1}z|^2\Big)
$$
by
$$
\left(\frac{1+\alpha^2}{\alpha^2}\right)\iint_{(0,T)\times J_3}s^3\mu^4\xi_0^3(|\partial_{x_1}^2 z|^2+(1+\alpha^2)|\partial_{t} z|^2).
$$

Next, to remove the local term $\iint_{(0,T)\times J_3}s^3\mu^4\xi^3_0|\partial_tz|^2$ we introduce $J_3\Subset J_4\Subset J$ and $\chi_4\in C_0^\infty(J_4)$ such that $\chi_4\equiv 1$ in $J_3$. Let $\varepsilon_4>0$. 
By integrating by parts and by using \eqref{poutre1}, \eqref{n2-2}, \eqref{n1-2} we get 
\begin{multline*}
    \frac{(1+\alpha^2)^2}{\alpha^2} s^3\mu^4 \iint_{(0,T)\times J_3}\xi_0^3|\partial_t z|^2\leq \frac{(1+\alpha^2)^2}{\alpha^2} s^3\mu^4 \iint_{(0,T)\times J_4}\chi_4\xi^3_0 |\partial_tz|^2\\
    =
     - \frac{(1+\alpha^2)^2}{\alpha^2} \iint_{(0,T)\times J_4}\partial_t(s^3\mu^4\chi_4\xi_0^3\partial_t z) z\\
    \leq \frac{(1+\alpha^2)^2}{\alpha^2} C_4 \iint_{(0,T)\times J_4}(s^3\mu^4\xi_0^3|\partial_t^2 z||z|+s^4\mu^4\xi_0^4|\partial_t z||z|) \\
    \leq C_4 \varepsilon_4 \iint_{(0,T)\times \mathcal{I}}\left(\frac{\alpha^2}{1+\alpha^2}\frac{1}{s\xi_0}|\partial_t^2 z|^2+(1+\alpha^2)s^3\mu^4\xi_0^4|\partial_tz|^2\right)+\left(\frac{(1+\alpha^{5})^2}{\alpha^6}\right)\frac{C_4}{\varepsilon_4}\iint_{(0,T)\times J_4}s^7\mu^8\xi_0^7|z|^2.
\end{multline*}
where $C_4>0$ is independent on $\varepsilon_4$. Then by choosing $\varepsilon_4>0$ small enough we can remove the desired local term.

Finally, to remove the local term $\frac{1+\alpha^2}{\alpha^2}\iint_{(0,T)\times J_3}s^3\mu^4\xi^3_0|\partial^2_{x_1}z|^2$ we proceed as in Section \ref{SectLocalTerm1}. Then finally we obtain:

\begin{multline}\label{Eqs15-03-2025-Final}
\frac{\alpha^2}{(1+\alpha^2)}\iint_{(0,T)\times\mathcal{I}}\frac{1}{s\xi_0}(|\partial^4_{x_1}z|^2+|\partial^2_tz|^2)+\alpha^2\iint_{(0,T)\times\mathcal{I}}\frac{1}{s\xi_0}|\partial_t\partial^2_{x_1}z|^2\\
+ \iint_{(0,T)\times\mathcal{I}} (s^7\mu^8\xi_0^7~|z|^2+s^5\mu^6\xi_0^6~|\partial_{x_1}z|^2+s^3\mu^4\xi_0^3~(|\partial^2_{x_1}z|^2+(1+\alpha^2)|\partial_{t}z|^2)+s\mu^2\xi_0~((1+\alpha^2)|\partial_t\partial_{x_1}z|^2+|\partial_{x_1}^3z|^2)\\
\lesssim \Big(\iint_{(0,T)\times\mathcal{I}}e^{2s\varphi_0}|f_\eta|^2+\alpha^2s^3\mu^4\iint_{(0,T)\times\mathcal{I}}\xi_0^3~e^{2s\varphi_0}|\partial_t \eta|^2+\left(\frac{1}{\alpha^{8}}+\alpha^{4}\right)\iint_{(0,T)\times J}\Big(s^7\mu^8\xi^7_0~|z|^2\Big).
\end{multline}
Going back to the initial variable $\eta = z e^{-s\varphi_0}$, the estimate \eqref{Eqs15-03-2025-Final}
gives immediately Theorem \ref{poutre}.
\subsection{Estimate of the cross product from below in the case \texorpdfstring{$\alpha<\alpha^*$}{etoilealpha} -- Proof of Theorem \ref{poutre3}} 
We recall that,  for any non-negative real parameters $\alpha$ and $\beta$, the equation \eqref{produit-croise-I_l} reads
\begin{multline*}
\iint_{(0,T)\times\mathcal{I}}\mathcal{M}_{11} z\cdot \mathcal{M}_{21} z=(8-6\beta)I_1+(66+36\beta) I_2+(12-6\beta) I_3\\
+(3\alpha^2+6\beta) I_4+2 I_5+(\alpha^2+6)I_6+(32+12\beta)\alpha I_7+ 12 \alpha I_8+\mathcal{R}_2.
\end{multline*}
To have positive coefficients in front of $I_\ell$, for $\ell=1,\ldots ,6$, we must impose $\beta\in \left[0,\frac{4}{3}\right)$. 
It remains to estimate the terms $(32+12\beta)\alpha I_7$ and $ 12 \alpha I_8$. For that, we follow the strategy of \cite{Sourav} which consists in using Cauchy-Schwarz inequality to bound those terms by $I_2$, $I_3$, $I_4$, $I_8$. We will observe that with such an approach the optimal value of $\beta$ is $\beta^*$, and it allows to obtain a lower bound for the cross product for $\alpha<\alpha^*$, with $\beta^*$ and $\alpha^*$ defined in \eqref{DefAlphastar}.

First, recall that
$$
(32+12\beta)\alpha I_7=(32+12\beta)s^4\alpha\iint_{(0,T)\times\mathcal{I}}(\partial_{x_1}\varphi_0)^3(\partial^2_{x_1}\varphi_0)\partial_tz\partial_{x_1}z.
$$
For any $\theta_1>0$, we have by \eqref{poutre2-0} and Young's inequality
\begin{multline}\label{F7CS}
\left|(32+12\beta)s^4\alpha\iint_{(0,T)\times\mathcal{I}\backslash J_0}(\partial_{x_1}\varphi_0)^3(\partial^2_{x_1}\varphi_0)\partial_tz\partial_{x_1}z\right|\\
\leq \frac{(16+6\beta)\alpha^2}{(3\alpha^2+6\beta)\theta_1}s^5\iint_{(0,T)\times\mathcal{I}\backslash J_0}(\partial_{x_1}\varphi_0)^4(\partial^2_{x_1}\varphi_0)|\partial_{x_1}z|^2\\
+(16+6\beta)\theta_1(3\alpha^2+6\beta) s^3\iint_{(0,T)\times\mathcal{I}\backslash J_0}(\partial_{x_1}\varphi_0)^2(\partial^2_{x_1}\varphi_0)|\partial_tz|^2.
\end{multline}
Then to bound $|(32+12\beta)\alpha I_7|$ by $(66+36\beta) I_2$ and $(3\alpha^2+6\beta) I_4$, and some corresponding local terms in $J_0$, we need $\frac{(16+6\beta)\alpha^2}{(3\alpha^2+6\beta)\theta_1}<66+36\beta$ and $(16+6\beta)\theta_1<1$, or equivalently
$$\frac{(8+3\beta)\alpha^2}{(3\alpha^2+6\beta)(33+18\beta)}<\theta_1<\frac{1}{(16+6\beta)}.$$
The limiting case corresponds to $ \frac{(8+3\beta)\alpha^2}{(3\alpha^2+6\beta)(33+18\beta)}=\frac{1}{(16+6\beta)}$ or equivalently $ \alpha^2=\frac{6\beta(33+18\beta)}{18\beta^2+42\beta+29}$. 
As a consequence, we control the term $|(32+12\beta)\alpha I_7|$ for any $\displaystyle \alpha< \alpha_1^*(\beta)=\sqrt{\frac{6\beta(33+18\beta)}{18\beta^2+42\beta+29}}$.

Similarly, we recall that
$$
12 \alpha I_8=12s^2\alpha\iint_{(0,T)\times\mathcal{I}}(\partial^2_{x_1}\varphi_0)(\partial_{x_1}\varphi_0)\partial^2_{x_1}z\partial_t\partial_{x_1}z
$$
and for $\theta_2>0$ we have
\begin{multline}\label{F8CS}
\big|12s^2\alpha\iint_{(0,T)\times\mathcal{I}\backslash J_0}(\partial^2_{x_1}\varphi_0)(\partial_{x_1}\varphi_0)\partial^2_{x_1}z\partial_t\partial_{x_1}z\big|\\
\leq \frac{6\alpha^2}{(\alpha^2+6)\theta_2}s^3\iint_{(0,T)\times\mathcal{I}\backslash J_0}(\partial^2_{x_1}\varphi_0)(\partial_{x_1}\varphi_0)^2(\partial^2_{x_1}z)^2+6\theta_2(\alpha^2+6)s\iint_{(0,T)\times\mathcal{I}\backslash J_0}(\partial^2_{x_1}\varphi_0)(\partial_t\partial_{x_1}z)^2.    
\end{multline}
Then to bound $|12 \alpha I_8|$ by $(12-6\beta) I_3$ and $(\alpha^2+6)I_6$, and some corresponding local terms in $J_0$, we need $\frac{6\alpha^2}{(\alpha^2+6)\theta_2}<12-6\beta$ and $6\theta_2<1$ which leads to
 $$\frac{\alpha^2}{(\alpha^2+6)(2-\beta)}<\theta_2<\frac{1}{6}.$$
The limiting case  corresponds to $\frac{\alpha^2}{(\alpha^2+6)(2-\beta)}=\frac16$ or equivalently $\alpha^2=\frac{6(2-\beta)}{4+\beta}$.
As a consequence, we control the term $|12 \alpha I_8|$ for any $\displaystyle \alpha < \alpha_2^*(\beta)=\sqrt{\frac{6(2-\beta)}{4+\beta}}$. 

To conclude, we obtain a lower bound on the cross product for $\alpha$ less than $\alpha^*(\beta)=\min\{\alpha_1^*(\beta),\alpha_2^*(\beta)\}$. It is easily seen that the value $\alpha^*(\beta)$ is maximal for $\beta$ such that $\alpha_1^*(\beta)=\alpha_2^*(\beta)$ which is to say
$$
\frac{6(2-\beta)}{4+\beta}=\frac{6\beta(33+18\beta)}{18\beta^2+42\beta+29}
$$
or equivalently
$$
P(\beta) = 36\beta^3+111\beta^2+77\beta-58=0.
$$
It turns out that $P$ admits a unique positive roots $\beta_*$ which verifies $\beta_* < \frac{4}{3}$.
For $\beta=\beta^*$  and $\alpha<\alpha^*=\sqrt{\frac{6(2-\beta^*)}{4+\beta^*}}$ we have
\begin{multline}\label{ERstCPalphaloweralphastar}
\iint_{(0,T)\times\mathcal{I}}\mathcal{M}_{11} z\cdot \mathcal{M}_{21} z+\iint_{(0,T)\times J_0}\Big(s^5\mu^6\xi_0^5|\partial_{x_1}z|^2+s^3\mu^4\xi_0^3(|\partial^2_{x_1}z|^2+|\partial_tz|^2)+|\partial_t\partial_{x_1}z|^2)\Big)-\mathcal{R}_2\\
\geq C( I_1+ I_2+ I_3+ I_4+ I_5+I_6),
\end{multline}
for some constant $C>0$ that may depend on $\alpha\in (0,\alpha^*)$.
Combining \eqref{ERstCPalphaloweralphastar} with \eqref{EstR0}, \eqref{EstR1} and \eqref{ineqalite12}, we find that for $\lambda_0\geq 1$ and $s_0\geq 1 $ large enough, for all $\lambda\geq\lambda_0$ and all $s\geq s_0(T^k+T^{k-1})=s_0(T^k+T^{k-1})$, 

\begin{multline}\label{midle-estimateBis}
\|\mathcal{M}_1z\|^2_{L^2((0,T)\times\mathcal{I})}+\|\mathcal{M}_2z\|^2_{L^2((0,T)\times\mathcal{I})}\\
+ \iint_{(0,T)\times\mathcal{I}}\Big(s^7\mu^8\xi^7_0~|z|^2+s^5\mu^6\xi_0^5|\partial_{x_1}z|^2+s^3\mu^4\xi_0^3(|\partial^2_{x_1}z|^2+|\partial_tz|^2)  +s\mu^2\xi_0(|\partial^3_{x_1}z|^2+|\partial_t\partial_{x_1}z|^2)\Big)\\
\leq C \Big(\iint_{(0,T)\times J_0}\Big(s^7\mu^8\xi^7_0~|z|^2+s^5\mu^6\xi_0^5|\partial_{x_1}z|^2+s^3\mu^4\xi_0^3(|\partial^2_{x_1}z|^2+|\partial_tz|^2)+s\mu^2\xi_0(|\partial^3_{x_1}z|^2+|\partial_t\partial_{x_1}z|^2)\\
\iint_{(0,T)\times\mathcal{I}}e^{2s\varphi_0}|f_\eta|^2\Big)
\end{multline}
for some constant $C>0$ that may depend on $\alpha\in (0,\alpha^*)$. Finally, we proceed as in Section \ref{Sect3} to end the proof of Theorem \ref{poutre3}.


\subsection{Bounding \texorpdfstring{$\mathcal{R}_1$ -- Proof of the estimate \texorpdfstring{\eqref{EstR1}}{ref}}{reste}}\label{SectionEstR1}
To complete our study, it remains to  prove that $\mathcal{R}_1$, defined in \eqref{DefR1}, satisfies \eqref{EstR1}. This is the purpose of the present section.

For each $(i,j)\in \mathbb{I}$, where the set $\mathbb{I}$ is defined in \eqref{DefR1}, we first give an adequate expressions of $I_{ij}$ by using integration by parts, and then an estimate of $|I_{ij}|$ obtained by using  \eqref{poutre1}, \eqref{n2-2}, \eqref{n3-2} and \eqref{n1-2}. The estimate \eqref{EstR1} follows by combining all these estimates.
\paragraph{-- Computation of $I_{16}$ --}
$$
I_{16}=-2\iint_{(0,T)\times\mathcal{I}}s^5(\partial_{x_1}\varphi_0)^4(\partial_t\varphi_0)z\partial_tz=s^5\iint_{(0,T)\times\mathcal{I}}\partial_t[(\partial_{x_1}\varphi_0)^4(\partial_t\varphi_0)]|z|^2,
$$
which yields to the following estimate:
$$
|I_{16}|\lesssim \frac{1}{s_0^2}s^7\mu^4\iint_{(0,T)\times\mathcal{I}}\xi_0^{7}|z|^2.
$$
\paragraph{-- Computation of $I_{17}$ --}
$$
I_{17}=-2\alpha s^6\iint_{(0,T)\times\mathcal{I}}(\partial_{x_1}\varphi_0)^5(\partial_t\varphi_0)z\partial_{x_1}z=\alpha s^6\iint_{(0,T)\times\mathcal{I}}\partial_{x_1}\Big[(\partial_{x_1}\varphi_0)^5(\partial_t\varphi_0)\Big]|z|^2,$$
which implies
$
\displaystyle |I_{17}|\lesssim  \frac{\alpha}{s_0} s^7\mu^6\iint_{(0,T)\times\mathcal{I}}\xi_0^{7}|z|^2.
$
\paragraph{-- Computation of $I_{26}$ --}
\begin{align*}
I_{26}&=-12\alpha s^3\iint_{(0,T)\times\mathcal{I}}(\partial_{x_1}\varphi_0)^2 (\partial_t\varphi_0) \partial_tz \partial^2_{x_1}z\\
&=12\alpha s^3\iint_{(0,T)\times\mathcal{I}}\partial_{x_1}\Big[(\partial_{x_1}\varphi_0)^2 (\partial_t\varphi_0)\Big]~\partial_{x_1}z \partial_tz-6\alpha s^3\iint_{(0,T)\times\mathcal{I}}\partial_t\Big[(\partial_{x_1}\varphi_0)^2 (\partial_t\varphi_0)\Big] |\partial_{x_1}z|^2,
\end{align*}
which yields the estimate
\begin{align*}
|I_{26}|&\lesssim  \frac{\alpha}{s_0} s^4 \mu^3\iint_{(0,T)\times\mathcal{I}}\xi_0^{4} |\partial_{x_1}z||\partial_tz|+\frac{\alpha}{s_0^2} s^5 \mu^2 \iint_{(0,T)\times\mathcal{I}}\xi_0^{5} |\partial_{x_1}z|^2\\
& \lesssim \frac{\alpha}{s_0} s^3 \mu^3\iint_{(0,T)\times\mathcal{I}}\xi_0^3 |\partial_tz|^2+\left(\frac{\alpha}{s_0}+\frac{\alpha}{s_0^2}\right) s^5 \mu^3 \iint_{(0,T)\times\mathcal{I}}\xi_0^{5}|\partial_{x_1}z|^2.
\end{align*}
\paragraph{-- Computation of $I_{36}$ --}
\begin{align*}
I_{36}&=-2s\iint_{(0,T)\times\mathcal{I}} (\partial_t\varphi_0)\partial_tz \partial^4_{x_1}z\\
&=2s\iint_{(0,T)\times\mathcal{I}} (\partial_t\partial_{x_1}\varphi_0) \partial_tz \partial^3_{x_1}z+2s\iint_{(0,T)\times\mathcal{I}} (\partial_t\varphi_0) \partial_t\partial_{x_1}z \partial^3_{x_1}z\\
&=-2s\iint_{(0,T)\times\mathcal{I}}(\partial_t\partial^2_{x_1}\varphi_0) \partial_tz \partial^2_{x_1}z-4s\iint_{(0,T)\times\mathcal{I}}(\partial_t\partial_{x_1}\varphi_0) \partial_t\partial_{x_1}z \partial^2_{x_1}z+s\iint_{(0,T)\times\mathcal{I}} (\partial^2_t\varphi_0) |\partial^2_{x_1}z|^2.
\end{align*}
This gives the following estimate
\begin{align*}
|I_{36}|&\lesssim \frac{1}{s_0^2}s^3 \mu^2 \iint_{(0,T)\times\mathcal{I}}\xi_0^{3}|\partial_tz||\partial^2_{x_1}z|+\frac{1}{s_0}s^2 \mu \iint_{(0,T)\times\mathcal{I}}\xi_0^{2}|\partial_t\partial_{x_1}z||\partial^2_{x_1}z|+\frac{1}{s_0^2}s^3\iint_{(0,T)\times\mathcal{I}}\xi_0^{3}|\partial^2_{x_1}z|^2\\
&\lesssim \frac{1}{s_0^2}s^3 \mu^2 \iint_{(0,T)\times\mathcal{I}}\xi_0^{3}(|\partial_tz|^2+|\partial^2_{x_1}z|^2)+\frac{1}{s_0}s\mu \iint_{(0,T)\times\mathcal{I}}\xi_0|\partial_t\partial_{x_1}z|^2.
\end{align*}
\paragraph{-- Computation of $I_{37}$ --}
\begin{align*}
    I_{37}&=-2s^2\alpha\iint_{(0,T)\times\mathcal{I}}(\partial_{x_1}\varphi_0)(\partial_t\varphi_0)\partial_{x_1}z\partial^4_{x_1}z\\
    &=2s^2\alpha\iint_{(0,T)\times\mathcal{I}}\partial_{x_1}\big[(\partial_{x_1}\varphi_0)(\partial_t\varphi_0)\big]\partial_{x_1}z\partial^3_{x_1}z-s^2\alpha\iint_{(0,T)\times\mathcal{I}}\partial_{x_1}\big[(\partial_{x_1}\varphi_0)(\partial_t\varphi_0)\big]|\partial^2_{x_1}z|^2\\
    &=s^2\alpha\iint_{(0,T)\times\mathcal{I}}\partial_{x_1}^3\big[(\partial_{x_1}\varphi_0)(\partial_t\varphi_0)\big]|\partial_{x_1}z|^2 -3\alpha s^2\iint_{(0,T)\times\mathcal{I}}\partial_{x_1}\big[(\partial_{x_1}\varphi_0)(\partial_t\varphi_0)\big]|\partial^2_{x_1}z|^2.
\end{align*}
As a consequence, we obtain
$ \displaystyle
    |I_{37}|\lesssim \frac{\alpha}{s_0^3} s^5\mu^4\iint_{(0,T)\times\mathcal{I}}\xi_0^{5}|\partial_{x_1}z|^2+\frac{\alpha}{s_0} s^3\mu^2\iint_{(0,T)\times\mathcal{I}}\xi_0^{3}|\partial^2_{x_1}z|^2.
$
\paragraph{-- Computation of $I_{46}$ --}
$$
I_{46}=-4s^2\alpha\iint_{(0,T)\times\mathcal{I}}(\partial_{x_1}\varphi_0)(\partial_t\varphi_0)\partial_t\partial_{x_1}z\partial_tz=2s^2\alpha \iint_{(0,T)\times\mathcal{I}}\partial_{x_1}\big[(\partial_{x_1}\varphi_0)(\partial_t\varphi_0)\big]|\partial_tz|^2.$$
From this, we obtain the estimate
$
\displaystyle |I_{46}| \lesssim \frac{\alpha}{s_0} s^3\mu^2\iint_{(0,T)\times\mathcal{I}}\xi_0^{3}|\partial_tz|^2.
$
\paragraph{-- Computation of $I_{47}$ --}
$$
I_{47}=-4s^3\alpha^2\iint_{(0,T)\times\mathcal{I}}(\partial_{x_1}\varphi_0)^2(\partial_t\varphi_0)\partial_{x_1}z\partial_t\partial_{x_1}z=2s^3\alpha^2\iint_{(0,T)\times\mathcal{I}}\partial_t\big[(\partial_{x_1}\varphi_0)^2(\partial_t\varphi_0)\big]|\partial_{x_1}z|^2.$$
As a consequence, we obtain $
\displaystyle |I_{47}|\lesssim \frac{\alpha^2}{s_0^2} s^5\mu^2\iint_{(0,T)\times\mathcal{I}}\xi_0^{5}|\partial_{x_1}z|^2.
$
\paragraph{-- Computation of $I_{56}$ --}
$$
I_{56}=-2s\iint_{(0,T)\times\mathcal{I}}(\partial_t\varphi_0)\partial_tz\partial^2_tz=s\iint_{(0,T)\times\mathcal{I}}(\partial^2_t\varphi_0)|\partial_tz|^2,
$$
which implies
$\displaystyle  |I_{56}|\lesssim \frac{1}{s_0^2}s^3\iint_{(0,T)\times\mathcal{I}}\xi_0^{3}|\partial_t z|^2.$
\paragraph{-- Computation of $I_{57}$ --}
\begin{align*}
I_{57}&=-2s^2\alpha\iint_{(0,T)\times\mathcal{I}}(\partial_{x_1}\varphi_0)(\partial_t\varphi_0)\partial_{x_1}z\partial^2_tz\\
&=2s^2\alpha\iint_{(0,T)\times\mathcal{I}}\partial_t\big[(\partial_{x_1}\varphi_0)(\partial_t\varphi_0)\big]\partial_{x_1}z\partial_tz-s^2\alpha\iint_{(0,T)\times\mathcal{I}}\partial_{x_1}\big[(\partial_{x_1}\varphi_0)(\partial_t\varphi_0)\big]|\partial_tz|^2,
\end{align*}
and this gives
\begin{align*}
|I_{57}|&\lesssim \frac{\alpha}{s_0^2} s^4\mu \iint_{(0,T)\times\mathcal{I}}\xi_0^4|\partial_{x_1}z||\partial_tz|+\frac{\alpha}{s_0} s^3\mu^2\iint_{(0,T)\times\mathcal{I}}\xi_0^3|\partial_tz|^2\\
&\lesssim \frac{\alpha}{s_0^2} s^5\mu \iint_{(0,T)\times\mathcal{I}}\xi_0^5|\partial_{x_1}z|^2+\left(\frac{\alpha}{s_0^2}+\frac{\alpha}{s_0}\right) s^3\mu^2\iint_{(0,T)\times\mathcal{I}}\xi_0^3|\partial_tz|^2.
\end{align*}
\paragraph{-- Computation of $I_{61}$ --}
\begin{align*}
I_{61}&=-4s^5\iint_{(0,T)\times\mathcal{I}}(\partial_{x_1}\varphi_0)^3(\partial_t\varphi_0)^2z\partial_{x_1}z=2s^5\iint_{(0,T)\times\mathcal{I}}\partial_{x_1}\big[(\partial_{x_1}\varphi_0)^3(\partial_t\varphi_0)^2\big]|z|^2,
\end{align*}
which leads to
$
\displaystyle |I_{61}|\lesssim  \frac{1}{s_0^2}s^7\mu^4\iint_{(0,T)\times\mathcal{I}}\xi_0^7|z|^2.
$
\paragraph{-- Computation of $I_{62}$ --}
\begin{align*}
I_{62}&=-4s^3\iint_{(0,T)\times\mathcal{I}}(\partial_{x_1}\varphi_0)(\partial_t\varphi_0)^2z\partial^3_{x_1}z\\
&=4s^3\iint_{(0,T)\times\mathcal{I}}\partial_{x_1}\big[(\partial_{x_1}\varphi_0)(\partial_t\varphi_0)^2\big]z\partial^2_{x_1}z-2s^3\iint_{(0,T)\times\mathcal{I}}\partial_{x_1}\big[(\partial_{x_1}\varphi_0)(\partial_t\varphi_0)^2\big]|\partial_{x_1}z|^2\\
&=-4s^3\iint_{(0,T)\times\mathcal{I}}\partial^2_{x_1}\big[(\partial_{x_1}\varphi_0)(\partial_t\varphi_0)^2\big]z\partial_{x_1}z-6s^3\iint_{(0,T)\times\mathcal{I}}\partial_{x_1}\big[(\partial_{x_1}\varphi_0)(\partial_t\varphi_0)^2\big]|\partial_{x_1}z|^2\\
&=2s^3\iint_{(0,T)\times\mathcal{I}}\partial^3_{x_1}\big[(\partial_{x_1}\varphi_0)(\partial_t\varphi_0)^2\big]|z|^2-6s^3\iint_{(0,T)\times\mathcal{I}}\partial_{x_1}\big[(\partial_{x_1}\varphi_0)(\partial_t\varphi_0)^2\big]|\partial_{x_1}z|^2.
\end{align*}
As a result, we obtain
$ \displaystyle
|I_{62}|\lesssim \frac{1}{s_0^4}s^7\mu^4\iint_{(0,T)\times\mathcal{I}}\xi_0^7|z|^2+\frac{1}{s_0^2}s^5\mu^2\iint_{(0,T)\times\mathcal{I}}\xi_0^5|\partial_{x_1}z|^2.
$
\paragraph{-- Computation of $I_{63}$ --}
\begin{align*}
I_{63}&=-s^2\alpha\iint_{(0,T)\times\mathcal{I}}(\partial_t\varphi_0)^2z\partial_t\partial^2_{x_1}z\\
&=2s^2\alpha\iint_{(0,T)\times\mathcal{I}}(\partial_t\partial_{x_1}\varphi_0)(\partial_t\varphi_0)z\partial_t\partial_{x_1}z-s^2\alpha\iint_{(0,T)\times\mathcal{I}}(\partial^2_t\varphi_0)(\partial_t\varphi_0)|\partial_{x_1}z|^2\\
&=s^2\alpha\iint_{(0,T)\times\mathcal{I}}\partial_t\partial_{x_1}\big[(\partial_t\partial_{x_1}\varphi_0)(\partial_t\varphi_0)\big]|z|^2-2s^2\alpha \iint_{(0,T)\times\mathcal{I}}(\partial_t\partial_{x_1}\varphi_0)(\partial_t\varphi_0)\partial_tz\partial_{x_1}z\\
&\quad -s^2\alpha\iint_{(0,T)\times\mathcal{I}}(\partial^2_t\varphi_0)(\partial_t\varphi_0)|\partial_{x_1}z|^2.
\end{align*}
From this, we get the following estimate:
\begin{align*}
|I_{63}|& \lesssim \frac{\alpha}{s_0^5} s^7\mu^2\iint_{(0,T)\times\mathcal{I}}\xi_0^7|z|^2+\frac{\alpha}{s_0^2} s^4\mu\iint_{(0,T)\times\mathcal{I}}\xi_0^4|\partial_tz||\partial_{x_1}z|+\frac{\alpha}{s_0^3} s^5\iint_{(0,T)\times\mathcal{I}}\xi_0^5|\partial_{x_1}z|^2\\
& \lesssim \frac{\alpha}{s_0^5} s^7\mu^2\iint_{(0,T)\times\mathcal{I}}\xi_0^7|z|^2+\frac{\alpha}{s_0^2} s^3\mu\iint_{(0,T)\times\mathcal{I}}\xi_0^3|\partial_tz|^2+\left(\frac{\alpha}{s_0^2}+\frac{\alpha}{s_0^3}\right) s^5\mu\iint_{(0,T)\times\mathcal{I}}\xi_0^5|\partial_{x_1}z|^2.
\end{align*}
\paragraph{-- Computation of $I_{64}$ --}
\begin{align*}
I_{64}&=-s^4\alpha\iint_{(0,T)\times\mathcal{I}}(\partial_{x_1}\varphi_0)^2(\partial_t\varphi_0)^2z\partial_tz=\frac{s^4\alpha}{2}\iint_{(0,T)\times\mathcal{I}}\partial_t\big[(\partial_{x_1}\varphi_0)^2(\partial_t\varphi_0)^2\big]|z|^2,
\end{align*}
It follows that
$ \displaystyle
|I_{64}| \lesssim \frac{\alpha}{s_0^3} s^7\mu^2\iint_{(0,T)\times\mathcal{I}}\xi_0^7|z|^2.
$
\paragraph{-- Computation of $I_{65}$ --} $~$
\smallskip

As we directly have
$ \displaystyle
I_{65}=-12s^5\iint_{(0,T)\times\mathcal{I}}(\partial_{x_1}\varphi_0)^2(\partial^2_{x_1}\varphi_0)(\partial_t\varphi_0)^2|z|^2,
$
we  obtain immediately
$ \displaystyle
|I_{65}|\lesssim \frac{1}{s_0^2}s^7\mu^4\iint_{(0,T)\times\mathcal{I}}\xi_0^7|z|^2.
$
\paragraph{-- Computation of $I_{66}$ --} $~$
\smallskip

$$ \displaystyle I_{66} =-2s^3\iint_{(0,T)\times\mathcal{I}}(\partial_t\varphi_0)^3z\partial_tz=3s^3\iint_{(0,T)\times\mathcal{I}}(\partial^2_t\varphi_0)(\partial_t\varphi_0)^2|z|^2$$
which leads to $\displaystyle |I_{66}| \lesssim \frac{1}{s_0^4}s^7\iint_{(0,T)\times\mathcal{I}}\xi_0^7|z|^2$.

\paragraph{-- Computation of $I_{67}$ --}
\begin{align*}
I_{67}&=-2s^4\alpha \iint_{(0,T)\times\mathcal{I}}(\partial_{x_1}\varphi_0)(\partial_t\varphi_0)^3z\partial_{x_1}z=s^4\alpha\iint_{(0,T)\times\mathcal{I}}\partial_{x_1}\big[(\partial_{x_1}\varphi_0)(\partial_t\varphi_0)^3\big]|z|^2. \end{align*}
Accordingly, we obtain $ \displaystyle |I_{67}| \lesssim \frac{\alpha}{s_0^3} s^7\mu^2\iint_{(0,T)\times\mathcal{I}}\xi_0^7|z|^2$.

\paragraph{-- Computation of $I_{71}$ --}
\begin{align*}
I_{71}& =-4s^4\alpha  \iint_{(0,T)\times\mathcal{I}}(\partial_{x_1}\varphi_0)^3(\partial_t\varphi_0)\partial_{x_1}z\partial^2_{x_1}z=2s^4\alpha  \iint_{(0,T)\times\mathcal{I}}\partial_{x_1}\big[(\partial_{x_1}\varphi_0)^3(\partial_t\varphi_0)\big]|\partial_{x_1}z|^2,\end{align*}
and so
$ \displaystyle
|I_{71}| \lesssim \frac{\alpha}{s_0} s^5\mu^4\iint_{(0,T)\times\mathcal{I}}\xi_0^5|\partial_{x_1}z|^2.
$
\paragraph{-- Computation of $I_{72}$ --}
\begin{align*}
I_{72}&=-4s^2\alpha\iint_{(0,T)\times\mathcal{I}}(\partial_{x_1}\varphi_0)(\partial_t\varphi_0)\partial^2_{x_1}z\partial^3_{x_1}z=2 s^2\alpha \iint_{(0,T)\times\mathcal{I}}\partial_{x_1}\big[(\partial_{x_1}\varphi_0)(\partial_t\varphi_0)\big]|\partial^2_{x_1}z|^2,\end{align*}
and it follows that
$ \displaystyle |I_{72}|\lesssim \frac{\alpha}{s_0} s^3\mu^2\iint_{(0,T)\times\mathcal{I}}\xi_0^3|\partial^2_{x_1}z|^2$.

\paragraph{-- Computation of $I_{73}$ --}
\begin{align*}
 I_{73}&=-s\alpha^2 \iint_{(0,T)\times\mathcal{I}} (\partial_t\varphi_0)\partial^2_{x_1}z\partial_t\partial^2_{x_1}z=\frac{s\alpha^2}{2} \iint_{(0,T)\times\mathcal{I}} (\partial^2_t\varphi_0)|\partial^2_{x_1}z|^2, \end{align*}
 from which we deduce
 $\displaystyle |I_{73}| \lesssim \frac{\alpha^2}{s_0^2} s^3\iint_{(0,T)\times\mathcal{I}}\xi_0^3|\partial^2_{x_1}z|^2$.

\paragraph{-- Computation of $I_{74}$ --}
\begin{align*}
  I_{74}&= -s^3\alpha^2 \iint_{(0,T)\times\mathcal{I}} (\partial_{x_1}\varphi_0)^2(\partial_t\varphi_0)\partial_tz\partial^2_{x_1}z\\
  &=s^3\alpha^2 \iint_{(0,T)\times\mathcal{I}} \partial_{x_1}\big[(\partial_{x_1}\varphi_0)^2(\partial_t\varphi_0)\big]\partial_tz\partial_{x_1}z-\frac{s^3\alpha^2}{2}\iint_{(0,T)\times\mathcal{I}}\partial_t\big[(\partial_{x_1}\varphi_0)^2(\partial_t\varphi_0)\big]|\partial_{x_1}z|^2,\end{align*}
  which in turn implies
  \begin{align*}
  |I_{74}| &\lesssim \frac{\alpha}{s_0} s^4\mu^3\iint_{(0,T)\times\mathcal{I}}\xi_0^4|\partial_{x_1}z|(\alpha |\partial_tz|)+\frac{\alpha^2}{s_0^2} s^5\mu^2\iint_{(0,T)\times\mathcal{I}}\xi_0^5|\partial_{x_1}z|^2\\
  &\lesssim \frac{\alpha}{s_0} s^3\mu^3\iint_{(0,T)\times\mathcal{I}}\xi_0^3(\alpha^2 |\partial_tz|^2)+\left(\frac{\alpha}{s_0}+\frac{\alpha^2}{s_0^2}\right) s^5\mu^3\iint_{(0,T)\times\mathcal{I}}\xi_0^5|\partial_{x_1}z|^2.
\end{align*}
\paragraph{-- Computation of $I_{75}$ --}
\begin{align*}
I_{75}&=-12s^4\alpha\iint_{(0,T)\times\mathcal{I}}(\partial_{x_1}\varphi_0)^2(\partial^2_{x_1}\varphi_0)(\partial_t\varphi_0)z\partial^2_{x_1}z\\
&=-6s^4\alpha\iint_{(0,T)\times\mathcal{I}}\partial^2_{x_1}\big[(\partial_{x_1}\varphi_0)^2(\partial^2_{x_1}\varphi_0)(\partial_t\varphi_0)\big]|z|^2+12s^4\alpha\iint_{(0,T)\times\mathcal{I}}(\partial_{x_1}\varphi_0)^2(\partial^2_{x_1}\varphi_0)(\partial_t\varphi_0)|\partial_{x_1}z|^2.\end{align*}
Thereby, it follows that
$\displaystyle |I_{75}|\lesssim \frac{\alpha}{s_0^3} s^7\mu^6\iint_{(0,T)\times\mathcal{I}}\xi_0^7|z|^2+\frac{\alpha}{s_0}s^5\mu^4\iint_{(0,T)\times\mathcal{I}}\xi_0^5|\partial_{x_1}z|^2$.

\paragraph{-- Computation of $I_{76}$ --}
\begin{align*}
I_{76}&=-2s^2\alpha\iint_{(0,T)\times\mathcal{I}}(\partial_t\varphi_0)^2\partial_tz\partial^2_{x_1}z\\
&=4s^2\alpha\iint_{(0,T)\times\mathcal{I}}(\partial_t\varphi_0)(\partial_t\partial_{x_1}\varphi_0)\partial_tz\partial_{x_1}z-2s^2\alpha\iint_{(0,T)\times\mathcal{I}}(\partial_t\varphi_0)(\partial_t^2\varphi_0)|\partial_{x_1}z|^2.\end{align*}
A direct estimation of these two terms gives
\begin{align*}
|I_{76}| &\lesssim \frac{\alpha}{s_0^2} s^4\mu \iint_{(0,T)\times\mathcal{I}}\xi_0^4|\partial_tz||\partial_{x_1}z|+\frac{\alpha}{s_0^3} s^5\iint_{(0,T)\times\mathcal{I}}\xi_0^5|\partial_{x_1}z|^2\\
&\lesssim \frac{\alpha}{s_0^2} s^3\mu \iint_{(0,T)\times\mathcal{I}}\xi_0^3|\partial_tz|^2+\left(\frac{\alpha}{s_0^3}+\frac{\alpha}{s_0^2}\right) s^5\iint_{(0,T)\times\mathcal{I}}\xi_0^5|\partial_{x_1}z|^2.
\end{align*}
\paragraph{-- Computation of $I_{77}$ --}
\begin{align*}
I_{77}&=-2s^3\alpha^2\iint_{(0,T)\times\mathcal{I}}(\partial_{x_1}\varphi_0)(\partial_t\varphi_0)^2\partial^2_{x_1}z\partial_{x_1}z=s^3\alpha^2\iint_{(0,T)\times\mathcal{I}}\partial_{x_1}\big[(\partial_{x_1}\varphi_0)(\partial_t\varphi_0)^2\big]|\partial_{x_1}z|^2,\end{align*}
therefore $\displaystyle |I_{77}| \lesssim \frac{\alpha^2}{s_0^2} s^5\mu^2\iint_{(0,T)\times\mathcal{I}}\xi_0^5|\partial_{x_1}z|^2.$

\paragraph{-- Computation of $I_{81}$ --}
\begin{align*}
I_{81}&=-4s^6\alpha\iint_{(0,T)\times\mathcal{I}}(\partial_{x_1}\varphi_0)^5(\partial_t\varphi_0)z\partial_{x_1}z=2s^6\alpha\iint_{(0,T)\times\mathcal{I}}\partial_{x_1}\big[(\partial_{x_1}\varphi_0)^5(\partial_t\varphi_0)\big]|z|^2,\end{align*}
from which we deduce
$\displaystyle |I_{81}| \lesssim \frac{\alpha}{s_0} s^7\mu^6\iint_{(0,T)\times\mathcal{I}}\xi_0^7|z|^2.$

\paragraph{-- Computation of $I_{82}$ --}
\begin{align*}
I_{82}&=-4s^4\alpha\iint_{(0,T)\times\mathcal{I}}(\partial_{x_1}\varphi_0)^3(\partial_t\varphi_0)z\partial^3_{x_1}z\\
&=4s^4\alpha\iint_{(0,T)\times\mathcal{I}}\partial_{x_1}\big[(\partial_{x_1}\varphi_0)^3(\partial_t\varphi_0)\big]z\partial^2_{x_1}z-2s^4\alpha\iint_{(0,T)\times\mathcal{I}}\partial_{x_1}\big[(\partial_{x_1}\varphi_0)^3(\partial_t\varphi_0)\big]|\partial_{x_1}z|^2\\
&=2s^4\alpha\iint_{(0,T)\times\mathcal{I}}\partial^3_{x_1}\big[(\partial_{x_1}\varphi_0)^3(\partial_t\varphi_0)\big]|z|^2-6s^4\alpha\iint_{(0,T)\times\mathcal{I}}\partial_{x_1}\big[(\partial_{x_1}\varphi_0)^3(\partial_t\varphi_0)\big]|\partial_{x_1}z|^2,\end{align*}
and as a consequence
$\displaystyle |I_{82}| \lesssim \frac{\alpha}{s_0^3} s^7\mu^6\iint_{(0,T)\times\mathcal{I}}\xi_0^7|z|^2+\frac{\alpha}{s_0} s^5\mu^4\iint_{(0,T)\times\mathcal{I}}\xi_0^5|\partial_{x_1}z|^2.$

\paragraph{-- Computation of $I_{83}$ --}
\begin{align*}
I_{83}&=-s^3\alpha^2\iint_{(0,T)\times\mathcal{I}}(\partial_{x_1}\varphi_0)^2(\partial_t\varphi_0)z\partial_t\partial^2_{x_1}z\\
&=s^3\alpha^2\iint_{(0,T)\times\mathcal{I}}\partial_{x_1}\big[(\partial_{x_1}\varphi_0)^2(\partial_t\varphi_0)\big]z\partial_t\partial_{x_1}z-\frac{s^3\alpha^2}{2}\iint_{(0,T)\times\mathcal{I}}\partial_t\big[(\partial_{x_1}\varphi_0)^2(\partial_t\varphi_0)\big]|\partial_{x_1}z|^2\\
&=\frac{s^3\alpha^2}{2}\iint_{(0,T)\times\mathcal{I}}\partial_t\partial^2_{x_1}\big[(\partial_{x_1}\varphi_0)^2(\partial_t\varphi_0)\big]|z|^2-s^3\alpha^2\iint_{(0,T)\times\mathcal{I}}\partial_{x_1}\big[(\partial_{x_1}\varphi_0)^2(\partial_t\varphi_0)\big]\partial_tz~\partial_{x_1}z\\
&-\frac{s^3\alpha^2}{2}\iint_{(0,T)\times\mathcal{I}}\partial_t\big[(\partial_{x_1}\varphi_0)^2(\partial_t\varphi_0)\big]|\partial_{x_1}z|^2. \end{align*}
This leads to the following estimate:
\begin{align*}
|I_{83}| &\lesssim \frac{\alpha^2}{s_0^4} s^7\mu^4\iint_{(0,T)\times\mathcal{I}}\xi_0^7|z|^2+\frac{\alpha}{s_0} s^4\mu^3\iint_{(0,T)\times\mathcal{I}}\xi_0^4(\alpha|\partial_tz|)|\partial_{x_1}z|+\frac{\alpha^2}{s_0^2} s^5\mu^2\iint_{(0,T)\times\mathcal{I}}\xi_0^5|\partial_{x_1}z|^2\\
&\lesssim \frac{\alpha^2}{s_0^4} s^7\mu^4\iint_{(0,T)\times\mathcal{I}}\xi_0^7|z|^2+\frac{\alpha}{s_0} s^3\mu^3\iint_{(0,T)\times\mathcal{I}}\xi_0^3(\alpha^2|\partial_tz|^2)+\left(\frac{\alpha}{s_0}+\frac{\alpha^2}{s_0^2}\right) s^5\mu^3\iint_{(0,T)\times\mathcal{I}}\xi_0^5|\partial_{x_1}z|^2.
\end{align*}
\paragraph{-- Computation of $I_{84}$ --}
\begin{align*}
I_{84}&=-s^5\alpha^2\iint_{(0,T)\times\mathcal{I}}(\partial_{x_1}\varphi_0)^4(\partial_t\varphi_0)z\partial_tz=\frac{s^5\alpha^2}{2}\iint_{(0,T)\times\mathcal{I}}\partial_t\big[(\partial_{x_1}\varphi_0)^4(\partial_t\varphi_0)\big]|z|^2,\end{align*}
and therefore
$\displaystyle |I_{84}|\lesssim \frac{\alpha^2}{s_0^2} s^7\mu^4\iint_{(0,T)\times\mathcal{I}}\xi_0^7|z|^2.$

\paragraph{-- Computation of $I_{85}$ --}
\par\noindent 
\medskip

As we have
$ \displaystyle I_{85} =-12s^6\alpha\iint_{(0,T)\times\mathcal{I}}(\partial_{x_1}\varphi_0)^4(\partial^2_{x_1}\varphi_0)(\partial_t\varphi_0)|z|^2,$ we directly obtain
$$|I_{85}|\lesssim \frac{\alpha}{s_0} s^7\mu^6\iint_{(0,T)\times\mathcal{I}}\xi_0^7|z|^2.$$

\paragraph{-- Computation of $I_{86}$ --}
\begin{align*}
I_{86}&=-2s^4\alpha\iint_{(0,T)\times\mathcal{I}}(\partial_{x_1}\varphi_0)^2(\partial_t\varphi_0)^2z\partial_tz=s^4\alpha\iint_{(0,T)\times\mathcal{I}}\partial_t\big[(\partial_{x_1}\varphi_0)^2(\partial_t\varphi_0)^2\big]|z|^2,\end{align*}
which directly gives
$\displaystyle |I_{86}|\lesssim \frac{\alpha}{s_0^3} s^7\mu^2\iint_{(0,T)\times\mathcal{I}}\xi_0^7|z|^2.$

\paragraph{-- Computation of $I_{87}$ --}
\begin{align*}
I_{87}&=-2s^5\alpha^2\iint_{(0,T)\times\mathcal{I}}(\partial_{x_1}\varphi_0)^3(\partial_t\varphi_0)^2z\partial_{x_1}z=s^5\alpha^2\iint_{(0,T)\times\mathcal{I}}\partial_{x_1}\big[(\partial_{x_1}\varphi_0)^3(\partial_t\varphi_0)^2\big]|z|^2,\end{align*}
which in turn implies 
$\displaystyle |I_{87}|\lesssim \frac{\alpha^2}{s_0^2} s^7\mu^4\iint_{(0,T)\times\mathcal{I}}\xi_0^7|z|^2,$ and gives the last estimate needed to obtain \eqref{EstR1}.

\appendix


\addcontentsline{toc}{section}{References}
\bibliography{refs}
\end{document}